\documentclass[a4paper]{report}

\usepackage{amsmath,amsthm,amsfonts,amscd}
\usepackage[utf8]{inputenc}
\usepackage[T1]{fontenc}
\usepackage{mathrsfs} 
\usepackage{tikz}
\pgfrealjobname{cours}
\usetikzlibrary{shapes}
\usepackage[3cm]{xiS}

\usepackage[alphabetic,initials]{amsrefs}

\newtheorem{theorem}{Theorem}

\newtheorem{lemma}[theorem]{Lemma}
\newtheorem{proposition}[theorem]{Proposition}
\newtheorem{corollary}[theorem]{Corollary}
\newtheorem{definition}[theorem]{Definition}
\newtheorem{example}[theorem]{Example}
\newtheorem{remark}[theorem]{Remark}
\newtheorem*{exercise}{Exercise}

\newenvironment{indented}{%
\par
\begingroup
\small
\it
\leftskip2em
\rightskip\leftskip
}{%
\par
\endgroup
}

\newcommand{\define}[1]{\emph{#1}}

\newcommand{\R}{\mathbb{R}}
\newcommand{\Z}{\mathbb{Z}}
\newcommand{\T}{\mathbb{T}}
\renewcommand{\S}{\mathbb{S}}
\newcommand{\F}{\mathscr{F}}
\DeclareMathOperator{\Div}{div}
\DeclareMathOperator{\SO}{SO}
\DeclareMathOperator{\SU}{SU}
\DeclareMathOperator{\Lie}{\mathcal{L}}
\newcommand{\ocan}{{\omega_{\mathrm{can}}}}
\newcommand{\xiOT}{\xi_\mathrm{\scriptscriptstyle OT}}

\begin{document}


\thispagestyle{empty}
\begin{flushright}
  {\huge 	\textsc{
  \begin{tabular}{r}
		\hbox{Topological methods in}\\
3--dimensional contact geometry
\end{tabular}
}

\vspace{.5cm}
\large
An illustrated introduction to Giroux's convex surfaces theory
}
\end{flushright}

\vspace{1.5cm}

\begin{center}

\hspace*{-.5cm}
\hbox{\includegraphics[width=14cm]{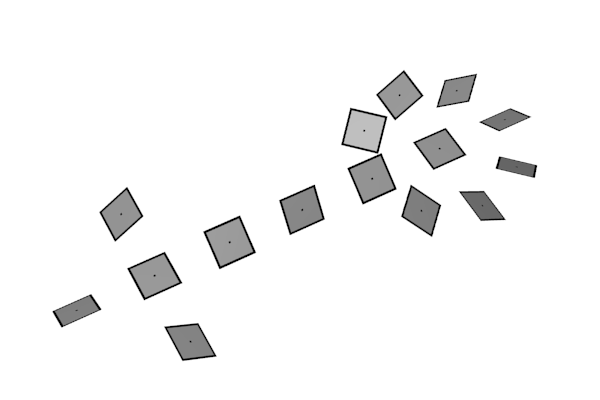}}

\vspace{1.5cm}
\Large
Patrick \textsc{Massot}

\vspace{1.5cm}
\large
Notes for the Nantes summer school in 

Contact and Symplectic Topology

June 2011

Revised in February 2013
\end{center}

\newpage

\chapter*{Introduction}

These lecture notes are an introduction to the study of global properties
of contact structures on 3-manifolds using topological rather than analytical
methods. From that perspective, the main tool to study a contact manifold 
$(V, \xi)$ is the study of its $\xi$-convex surfaces. These surfaces
embedded in $V$ are useful because all the information about $\xi$ near
each of them is encoded into a surprisingly small combinatorial data.  In order
to illustrate the power of $\xi$-convex surfaces without long developments, we
use them to reprove, following Giroux \cite{Giroux_2000}, two important
theorems which were originally proved using different techniques by Bennequin
\cite{Bennequin} and Eliashberg \cite{Eliashberg_20_ans}.

Besides Giroux's original papers \cite{Giroux_91, Giroux_2000}, there are
already two sets of lectures notes by Etnyre \cite{JohnNotes} and Honda
\cite{KoNotes} and a book by Geiges \cite{GeigesBook} which cover almost all
topics we will discuss as well as more advanced topics.
Our goal is not to replace those references but to complement them.
Mostly, we include many pictures that are not easily found in print and
can help to build intuition. We focus on a small set of contact manifolds and
illustrate all phenomena on those examples by showing explicit embedded
surfaces. On the other hand, we almost never give complete proofs.

Chapter 1 explains the local theory of contact structures starting with the
most basic definitions. There are many ways to define contact structures and
contact forms and we use unusual geometric definitions in order to complement
existing sources. We also try to explain the geometric intuition behind
the theorems of Darboux-Pfaff and Gray rather than using Moser's path method
without explanation.

Once enough definitions are given, an interlude states the theorems of Bennequin
and Eliashberg that are proved at the high point of these notes. It serves as
motivation for the rather long developments of Chapter 2.

Chapter 2 begins the study of surfaces in contact manifolds. The starting point
is the singular foliation printed by a contact structure on any surface. We
then work towards $\xi$-convex surfaces theory by simplifying gradually the
contact condition near a surface. Once the amazing realization lemma is proved,
we investigate obstructions to $\xi$-convexity and prove these obstructions are
generically not present. The last section of this chapter then get the first
fruits of this study by proving the Eliashberg-Bennequin inequalities.

Chapter 3 goes beyond the study of a single surface by studying some
one-parameter families of surfaces. In particular we describe what happens
exactly when one of the obstructions to $\xi$-convexity discussed in the
preceding chapter arises. This allows us to prove the theorems of Bennequin and
Eliashberg mentioned above. Until now, the proof of Bennequin's theorem using
$\xi$-convex surfaces was explained only in \cite{Giroux_2000}.

Of course this is only the beginning of a story which continues both by itself
and in combination with holomorphic curves techniques.

\paragraph{Conventions:}

A plane field $\xi$ on a 3--manifold $V$ is a (smooth) map associating to each
point $p$ of $V$ a 2-dimensional subspace $\xi(p)$ of $T_pV$. All plane fields
considered here will be coorientable, it means one can continuously choose one
of the half spaces cut out by $\xi(p)$ in $T_pV$. In this situation, $\xi$ can
be defined as the kernel of some nowhere vanishing 1--form $\alpha$:
$\xi(p) = \ker \alpha(p)$. The coorientation is given by the sign of $\alpha$.
We will always assume that $V$ is oriented. In this situation a coorientation of
$\xi$ combines with the ambient orientation to give an orientation on $\xi$.
All contact structures in these notes will be cooriented.

Occasionally, we will include remarks or comments that are not part of the main
flow of explanations. These remarks are typeset in small italic print.
\chapter{Local theory}

\section{Contact structures as rotating plane fields}

\subsection{The canonical contact structure on the space of contact
elements}

Let $S$ be a surface and $\pi: ST^*S \to S$ the bundle of cooriented lines
tangent to $S$ (also called contact elements for $S$). It can be seen as the
bundle of rays in $T^*S$, hence the notation. The canonical contact structure
on $ST^*S$ at a point $d$ is defined as the inverse image under $\pi_*$ of $d
\subset T_{\pi(d)}S$, see Figure~\ref{fig:xi_S}.
\begin{figure}
  \begin{center}
	\parbox{3.9cm}{\includegraphics[width=3.9cm]{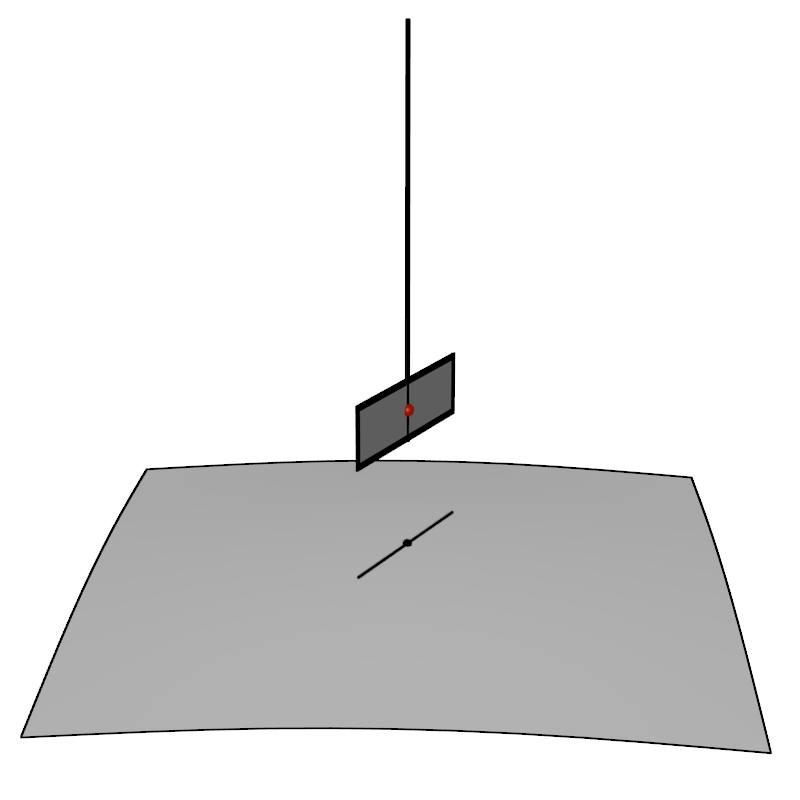}}
	\parbox{3.9cm}{\includegraphics[width=3.9cm]{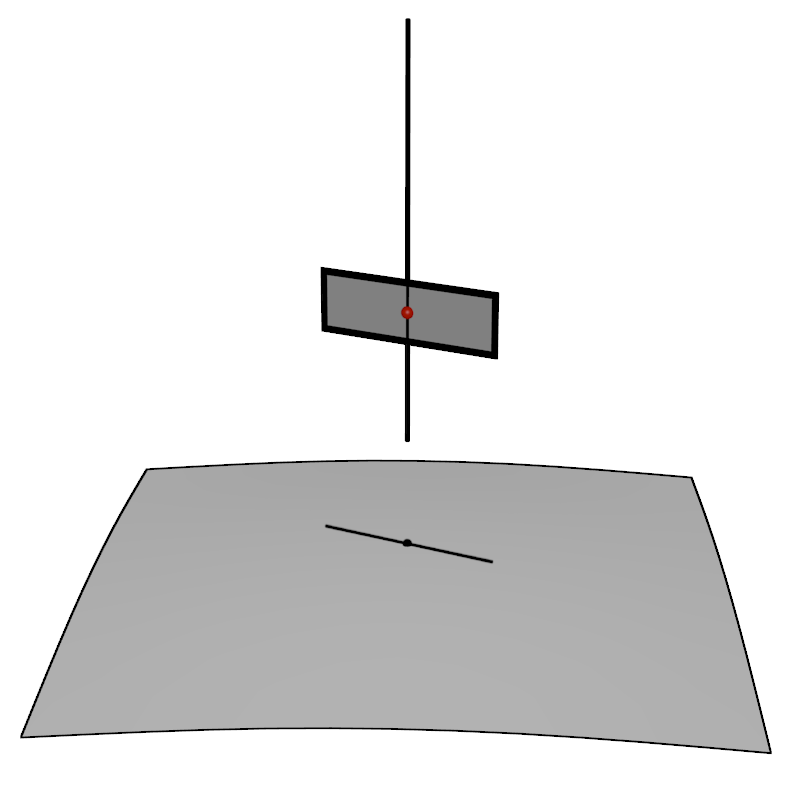}}
	\parbox{3.9cm}{\includegraphics[width=3.9cm]{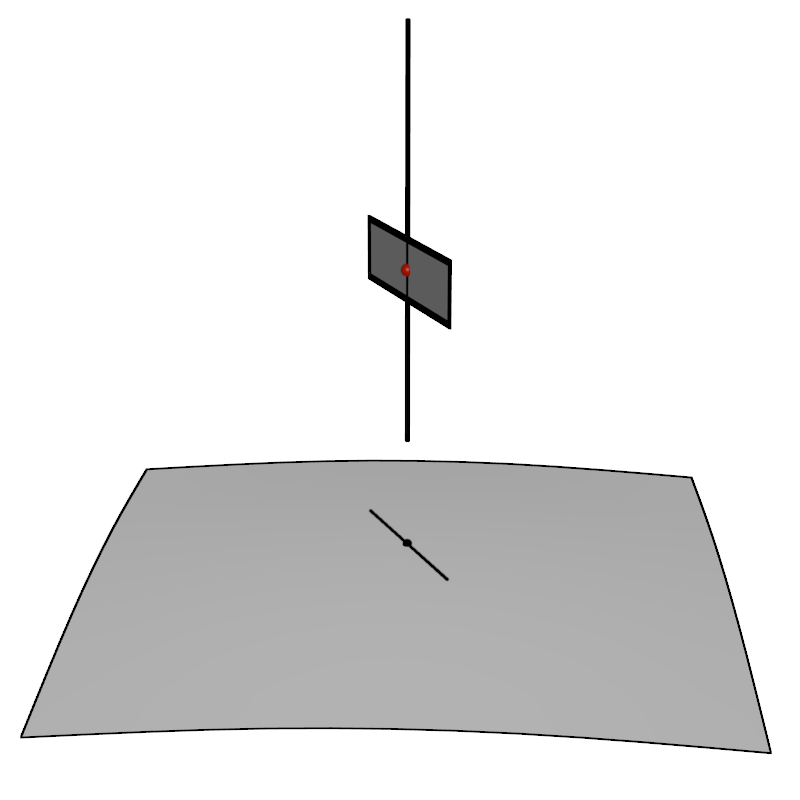}}
  \end{center}
  \caption{Canonical contact structure on the bundle of cooriented lines. At
	bottom is a portion of $S$ with a tangent line at some point. Above that point
	one gets the fiber by gluing top and bottom of the interval. The contact
	structure is shown at the point of the fiber corresponding to the line drawn below.}
  \label{fig:xi_S}
\end{figure}

Suppose first that $S$ is the torus $T^2 = \R^2/2\pi\Z^2$. 
Let $x$ and $y$ be the canonical $\S^1$-valued coordinates 
on $T^2$. A cooriented line tangent to $T^2$ at some point $(x, y)$ can
be seen as the kernel of a 1--form $\lambda$ which has unit norm with
respect to the canonical flat metric. So there is some angle $z$ such that
$\lambda = \cos(z)dx - \sin(z) dy$. Hence we have a natural identification of
$ST^*T^2$ with $T^3$. In addition the canonical contact structure can be
defined by $\cos(z)dx - \sin(z) dy$ now seen as a 1--form on $T^3$ called the
canonical contact form on $T^3$, see Figure~\ref{fig:xi_t3}.
\begin{figure}
  \begin{center}
		\includegraphics[width=8cm]{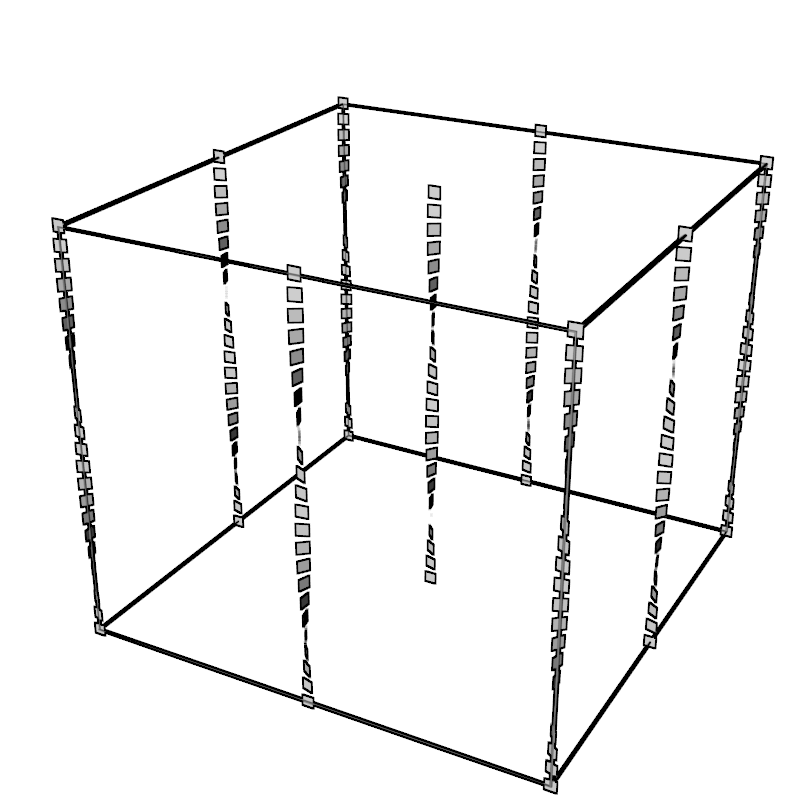}
  \end{center}
  \caption{Canonical contact structure on $T^3$. Opposite faces of the cube
	are glued to get $T^3$}
  \label{fig:xi_t3}
\end{figure}

When $S$ is the sphere $\S^2$, $ST^*S$ is endowed with a free transitive action
of $\SO_3(\R)$ so it is diffeomorphic to $\SO_3(\R)$. So there is a two-fold
covering map from $\S^3 \simeq \SU(2)$ to $ST^*\S^2$. The lifted plane field is
called the canonical contact structure on $\S^3$. We will see different ways of
describing this example later on.

\subsection{Contact structures and contact forms}

\begin{definition}
A \define{contact structure} on a 3--manifold is a plane field which is locally
diffeomorphic to the canonical contact structure on $ST^*T^2$.
A \define{contact form} is a 1--form whose kernel is a contact structure.
A curve or a vector field is \define{Legendrian} if it is tangent to a given
contact structure.
\end{definition}

As noted above all our manifolds will be oriented and diffeomorphisms in the
above definition shall preserve orientations.

\begin{theorem}[Darboux--Pfaff theorem]
\label{thm:Darboux}
A 1--form $\alpha$ is a contact form if and only if $\alpha \wedge d\alpha$ is a
positive volume form.
\end{theorem}

Let $\xi$ be the kernel of $\alpha$. The condition $\alpha \wedge d\alpha > 0$
will henceforth be called the contact condition for $\alpha$.  It is equivalent
to the requirement that $d\alpha_{|\xi}$ is non-degenerate and defines the
orientation of $\xi$ coming from the orientation of the ambient manifold and the
coorientation of $\xi$.

\begin{proof}
If $\xi$ is a contact
structure then the image of $\alpha$ in the local model is $f\alpha_0$ where $f$
is some nowhere vanishing function and $\alpha_0 = \cos(z)dx -\sin(z) dy$. 
So 
\begin{align*}
\alpha \wedge d\alpha &= f\alpha_0 \wedge (fd\alpha_0 + df \wedge \alpha_0)
= f^2 \alpha_0 \wedge d\alpha_0 \\
&= f^2\, dx \wedge dy \wedge dz
\end{align*}
which is a positive volume form. More generally the above computation proves
that the contact condition for a nowhere vanishing one-form depends only on its
kernel.

Conversely, suppose $\alpha \wedge d\alpha$ is positive.
Let $p$ be a point in $M$. We want to construct a coordinate chart around $p$
such that $\xi = \ker(\cos(z)dx -\sin(z) dy)$. We first choose a small surface $S$
containing $p$ and transverse to $\xi$. Then we pick a
non-singular vector field $X$ tangent to $S$ and $\xi$ near $p$ and a small
curve $c$ in $S$ containing $p$ and transverse to $X$, see
Figure~\ref{fig:darboux}.
\begin{figure}[htp]
	\begin{center}
	\parbox{5cm}{ \includegraphics[width=5cm]{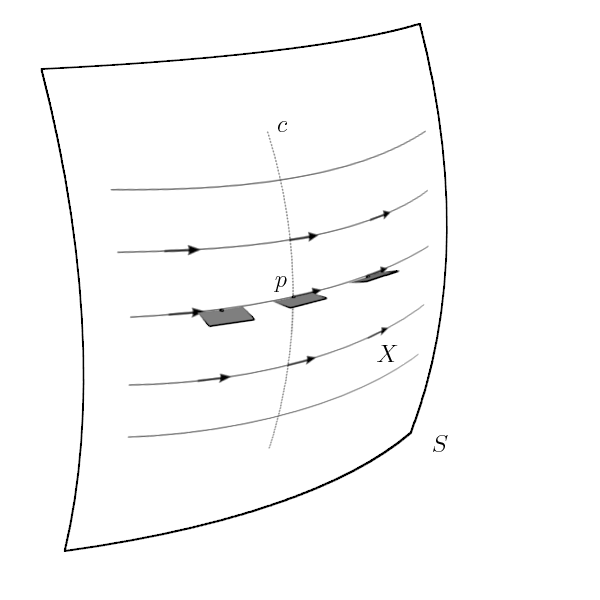}}
	\parbox{5cm}{ \includegraphics[width=5cm]{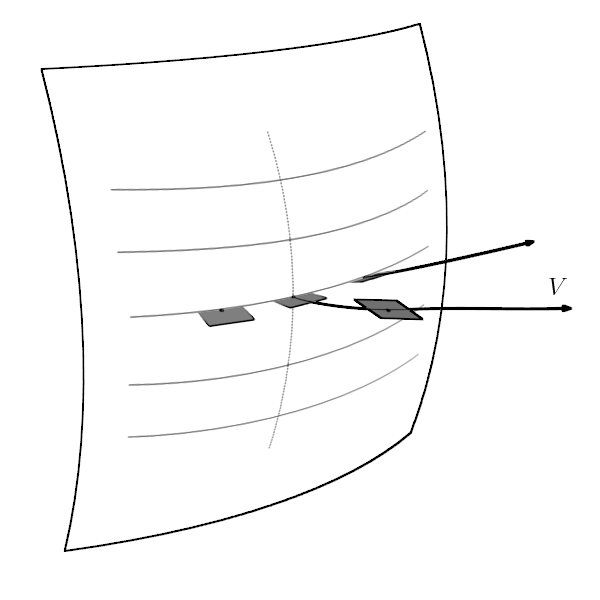}}
	\end{center}
	\caption{Proof of the Darboux--Pfaff theorem}
	\label{fig:darboux}
\end{figure}
Let $y$ be a coordinate on $c$. The flow of $X$ at time $x$ starting from $c$
gives coordinates $(x, y)$ on $S$ near $p$ in which $X = \partial_x$.

We now consider a vector field $V$ transverse to $S$ and tangent to $\xi$. The
flow of $V$ at time $t$ starting from $S$ gives coordinates $(x, y, t)$ near $p$
such that $\alpha = f(x, y, t) dx + g(x, y, t) dy$ because 
$\alpha(\partial_t) = \alpha(V) = 0$. Up to rescaling, one can use instead
$\alpha_1 = \cos z(x, y, t) dx - \sin z(x, y, t) dy$ for some function $z$
such that $z(x,y,0) = 0$.
Now it is time to use the contact condition.  We can compute
\[
\alpha_1 \wedge d\alpha_1 = \frac{\partial z}{\partial t} 
dx \wedge dy \wedge dt.
\]
Remember the contact condition for $\alpha$ is equivalent to the contact
condition for $\alpha_1$. So $\frac{\partial z}{\partial t}$ is positive and
the implicit function theorem then guaranties that we can use $z$ as a coordinate
instead of $t$.
\end{proof}

In the above proof, $z(x, y, t)$ was the angle between $\xi$ and
the horizontal $\partial_x$ is the plane normal to the Legendrian vector field
$\partial_t$.  We saw that the contact condition forces this angle to
increase. This means that the contact structure rotates around $\partial_t$.
The above proof essentially says that this rotation along Legendrian vector
fields characterizes contact structures.

We now focus on the difference between contact structures and contact forms.
The data of a contact form is equivalent to a contact structure and either a
choice of a Reeb vector field or a section of its symplectization.

\begin{definition}
A Reeb vector field for a contact structure $\xi$ is a vector field which is
transverse to $\xi$ and whose flow preserves $\xi$.
\end{definition}

If one has a Riemannian metric on a surface $S$ then the bundle of contact
elements of $S$ can be identified with the unit tangent bundle $STS$ and the
geodesic flow is then the flow of a Reeb vector field for the canonical contact
structure.

One can easily prove that each contact form $\alpha$ comes with a
canonical Reeb vector field $R_\alpha$ which is characterized by 
$d\alpha(R_\alpha, \cdot) = 0$ and $\alpha(R_\alpha) = 1$. All Reeb vector
fields arise this way.

Next, for any co-oriented plane field $\xi$ on a 3-manifold
$V$, one can consider the annihilator of $\xi$ in $T^*V$:
\begin{equation*}
  S_\xi := \bigl\{ \lambda \in T^*V\bigm|\,
  \text{$\ker \lambda = \xi$ and $\lambda (v) > 0$
    if $v$ is positively transverse to $\xi$}\bigr\} \;.
\end{equation*}
It is a good exercise to check that a plane field 
$\xi$ on $V$ is a contact structure if and only if $S_\xi$ is a
symplectic submanifold of $(T^*V, \ocan)$. In this case $S_\xi$ is
called the \emph{symplectization} of $\xi$.
The manifold $S_\xi$ is a principal $\R$--bundle where a real number $t$ acts by
$\lambda \mapsto e^t\lambda$.
Any contact form $\alpha$ is a section of this $\R$-bundle, and thus
determines a trivialization $\R \times V \to S_\xi$ given by
$(t, v) \mapsto e^t\alpha_v$. 
In this trivialization, the restriction of the canonical symplectic form
$\ocan$ becomes $d(e^t\alpha)$.

\section{Examples}

\subsubsection*{The canonical contact structure on $\R^3$}

The universal cover of $ST^*T^2$ is of course $\R^3$ and the lifted contact
structure is $\xi_0 = \ker \big(\cos(z)dx - \sin(z) dy\big)$ 
where $x$, $y$ and $z$ are now honest real-valued coordinates. The plane field 
$\xi_0$ is called the standard contact structure on $\R^3$. 
\begin{figure}[htp]
  \begin{center}
	\includegraphics[width=1cm]{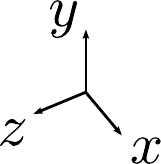}
	\includegraphics[width=10cm]{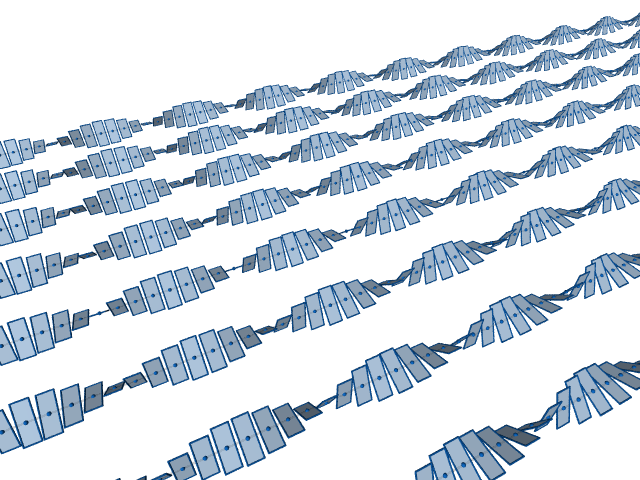}
  \end{center}
  \caption{Universal cover of the standard  contact structure on $\T^3$ seen
	from the side. It is invariant under translation in the vertical direction}
  \label{fig:std3}
\end{figure}

Depending on context, it can be useful to have different ways of looking at
$\xi_0$ using various diffeomorphisms of $\R^3$. The image of
$\xi_0$ under the diffeomorphism 
\[
\begin{pmatrix}
	x\\y\\z
\end{pmatrix} \mapsto
\begin{pmatrix}
t\\ p \\q 
\end{pmatrix} =
\begin{pmatrix}
\cos(z) & -\sin(z) & 0\\
\sin(z) & \cos(z) & 0\\
0 & 0 & 1
\end{pmatrix}
\begin{pmatrix}
x\\ y\\ z
\end{pmatrix}
\]
is drawn in Figure~\ref{fig:std1}. It admits the contact form $dt + pdq$ and
arises naturally on $\R^3$ seen as the space of 1-jets of functions from $\R$
to $\R$ (see e.g. \cite[Example 2.5.11]{GeigesBook} for more information on
this interpretation).
\begin{figure}[!h]
	\begin{center}
		\includegraphics[width=1cm]{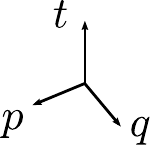}
	\includegraphics[height=6cm]{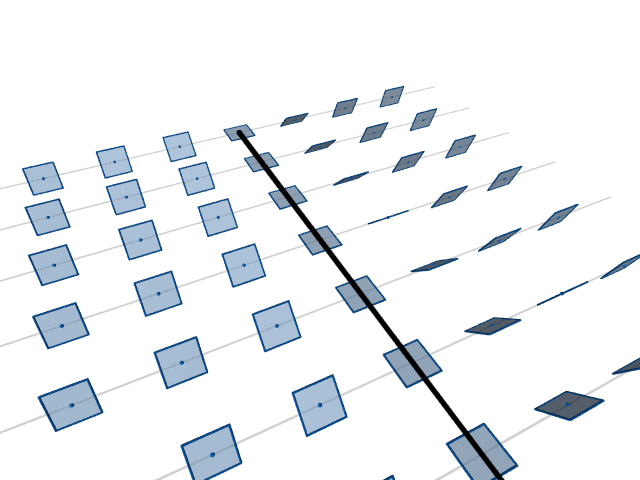}
	\end{center}
	\caption{$\ker(dt + pdq)$ on $\R^3$. It is invariant under
	translation in the vertical direction. It becomes vertical only if one goes
	all the way to $p = \pm\infty$.}
	\label{fig:std1}
\end{figure}

Figures~\ref{fig:std3} and~\ref{fig:std1} together are often confusing for
beginners.
First the thick black line $\{t = p = 0\}$ in Figure~\ref{fig:std1} is
Legendrian yet the contact structure does not seem to rotate along it. Second,
it seems the two pictures exhibit Legendrian foliations by lines with very
different behavior. In the second picture the contact structure turns half a
turn along each leave whereas it turns infinitely many turns in the first
picture.

Both puzzles are solved by the same picture. The diffeomorphism we used above
sends the foliation by Legendrian lines of Figure~\ref{fig:std3} to a foliation
containing the mysterious line $\{t = p = 0\}$ in Figure~\ref{fig:std1}
together with helices around that line, see Figure~\ref{fig:lien13}.
\begin{figure}[ht]
	\begin{center}
	\includegraphics[height=4cm]{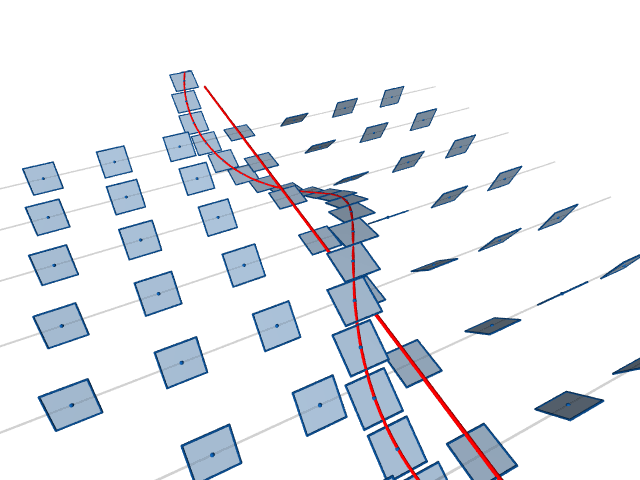}
	\includegraphics[height=4cm]{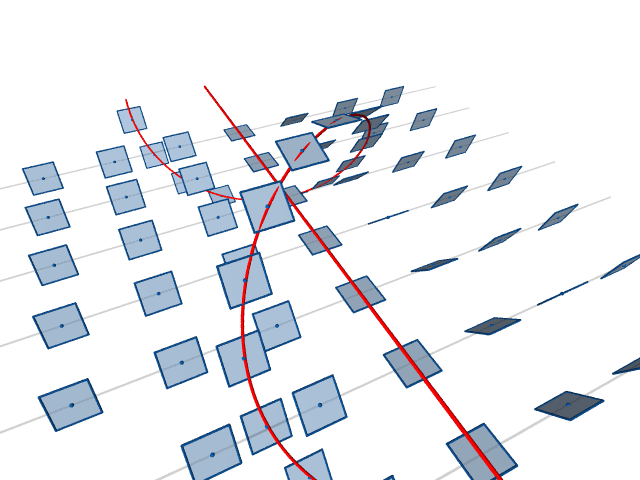}
	\end{center}
	\caption{The mysterious line in Figure \ref{fig:std1} together with two
	helices coming from the lines of Figure \ref{fig:std3}.}
	\label{fig:lien13}
\end{figure}

So we first see where is the foliation of Figure~\ref{fig:std3} inside
Figure~\ref{fig:std1}. And second we remember that it makes sense to say that a
plane field rotates along a curve only compared to something else. Contact
structures rotate along Legendrian curves compared to neighborhood leaves of
some Legendrian foliation. And indeed we see the contact structure turns
infinitely many times along the mysterious line compared to the nearby
Legendrian helices.

It is also sometimes convenient to consider the image of $\ker(dt + pdq)$ under the
diffeomorphism $(t, p, q) \mapsto (q, -p, t + \frac{pq}2)$. This image is the
kernel of $dz + \frac{1}{2} r^2 d\theta$ in cylindrical coordinates, see Figure
\ref{fig:std2}. In this model, one sees clearly that, at each point, there are
Legendrian curves going in every possible direction.
\begin{figure}
  \begin{center}
	\includegraphics[width=10cm]{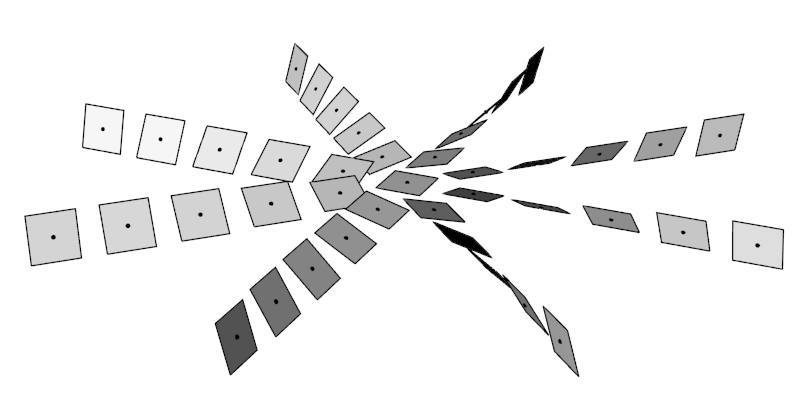}
  \end{center}
  \caption{Another view of the standard contact structure on $\R^3$}
  \label{fig:std2}
\end{figure}

Figure~\ref{fig:lien} shows how to deform Figure~\ref{fig:std2} to embed it
inside Figure~\ref{fig:std1}.
\begin{figure}
  \begin{center}
	\includegraphics[width=10cm]{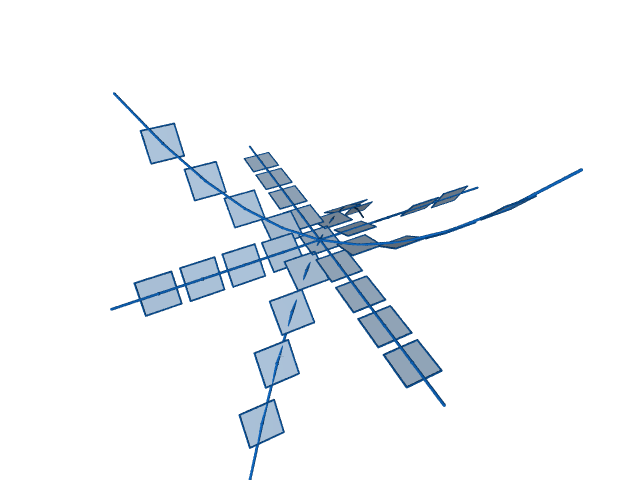}
  \end{center}
	\caption{Embedding of Figure~\ref{fig:std2} into Figure~\ref{fig:std1}}
  \label{fig:lien}
\end{figure}

Either of these contact structures (which are diffeomorphic by construction)
will be called the canonical contact structure on $\R^3$. Of course they can
all be used as the local model in the definition of a contact structure.

\subsubsection*{The canonical contact structure on $\S^3$}

We have already met the canonical contact structure on $\S^3$ coming from the
canonical contact structures on $ST^*\S^2$. One can prove that it is also
\begin{itemize}
\item 
the orthogonal of the Hopf circles for the round metric,
\item
a left-invariant contact structure on the Lie group $SU(2)$,
\item
$T\S^3 \cap JT\S^3$ when $\S^3$ is seen as the boundary of the unit ball in
$\mathbb{C}^2$ and $J$ denotes the action of multiplication by $i$ in 
$T\mathbb{C}^2$.
\end{itemize}

The complement of a point in the standard $\S^3$ is isomorphic to the standard
$\R^3$, see \cite[Proposition 2.1.8]{GeigesBook} for a computational proof
valid in any dimension.


\section{Isotopies}

\subsection{Isotopic contact structures and Gray's theorem}

Up to now we considered two contact structures to be the same if they are
conjugated by some diffeomorphism. One can restrict this by considering only
diffeomorphisms corresponding to deformations of the ambient manifold.
An isotopy is a family of diffeomorphisms $\varphi_t$ parametrized by 
$t \in [0, 1]$ such that 
$(x, t) \mapsto \varphi_t(x)$ is smooth and $\varphi_0 = Id$. The time-dependent
vector field generating $\varphi_t$ is defined as $X_t = \frac d{dt} \varphi_t$.
One says that two contact structures $\xi_0$ and $\xi_1$ are isotopic if there
is an isotopy $\varphi_t$ such that $\xi_1 = (\varphi_1)_*\xi_0$.
In particular such contact structures can be connected by the path of contact
structures $\xi_t := (\varphi_t)_*\xi_0$. It is then natural to consider the
seemingly weaker equivalence relation of homotopy among contact structures. The
next theorem says in particular that, on closed manifolds, this equivalence
relation is actually the same as the isotopy relation.

\begin{theorem}[Gray \cite{Gray}]
\label{thm:gray}
For any path $(\xi_t)_{t \in [0, 1]}$ of contact structures on a closed
manifold, there is an isotopy $\varphi_t$ such that $\varphi_t^*\xi_t = \xi_0$.

The vector field $X_t$ generating $\varphi_t$ can be chosen in
$\lim_{\varepsilon \to 0} \xi_t \cap \xi_{t + \varepsilon}$ at
each time $t$.
\end{theorem}

\begin{proof}
The proof of this theorem can be found in many places but without much geometric
explanations. So we now explain the picture behind it.  The key is to be able
to construct an isotopy pulling back $\xi_{t + \varepsilon}$ to $\xi_t$ for
infinitesimally small $\varepsilon$. It means we will construct the generating
vector field $X_t$ rather than $\varphi_t$ directly. The compactness assumption
will guaranty that the flow of $X_t$ exists for all time.

At any point $p$, if the plane $\xi_{t+\varepsilon}$ coincides with $\xi_t$
then we have nothing to do and set $X_t = 0$.
Otherwise, these two planes intersect transversely along a line
$d_{t,\varepsilon}$. 
The natural way to bring $\xi_{t+\varepsilon}$ back to $\xi_t$ is to rotate it
around $d_{t, \varepsilon}$. 
Since we know from the proof of Theorem \ref{thm:Darboux} that the flow of
Legendrian vector fields rotate the contact structure, we will choose $X_t$ in
the line $d_t := \lim_{\varepsilon \to 0} d_{t,\varepsilon}$, see
Figure~\ref{fig:gray}.
\begin{figure}[!htp]
  \begin{center}
	\includegraphics[width=8cm]{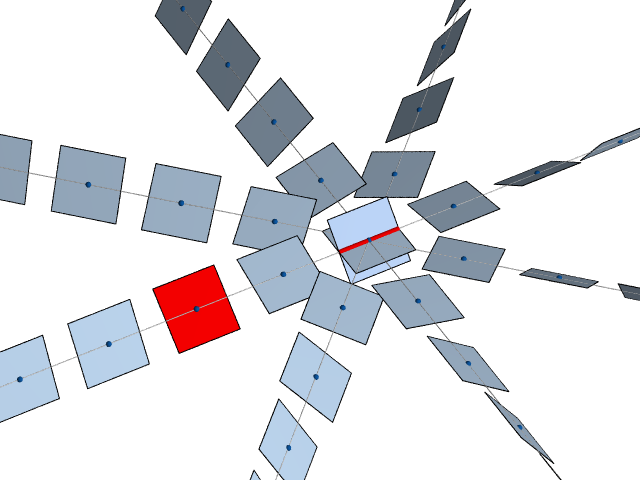}
  \end{center}
  \caption{Proof of Gray's theorem}
  \label{fig:gray}
\end{figure}
Let us compute $d_{t,\varepsilon}$~:
\[
d_{t,\varepsilon} = \{v \;|\; \alpha_{t + \varepsilon}(v) = \alpha_t(v) = 0\} =
\{v \in \xi_t \;|\; \textstyle \frac{1}{\varepsilon}
\displaystyle (\alpha_{t + \varepsilon} - \alpha_t)(v) = 0\}
\]
which gives, as $\varepsilon$ goes to zero:
$d_t = \xi_t \cap \ker(\dot\alpha_t)$. 

The contact condition for $\alpha_t$ is equivalent to the fact that
$(d\alpha_t)_{|\xi_t}$ is non-degenerate. So $X_t$ belongs to
$\xi_t \cap \ker(\dot\alpha_t)$ if and only if it belongs to $\xi_t$ and
$\iota_{X_t}d\alpha_t =f_t \dot\alpha_t$ on $\xi_t$ for some function $f_t$.

Moreover, we want $X_t$ to compensate the rotation expressed by 
$\dot \alpha_t$. 
A natural guess is then to pick the unique Legendrian vector field $X_t$ such
that $(\iota_{X_t} d\alpha_t)_{|\xi_t} = - (\dot\alpha_t)_{|\xi_t}$.

We now have a precise candidate for $X_t$ and we can compute to prove that it
does the job. Let $\varphi_t$ be the flow of $X_t$. Using Cartan's formula, we
get:
\begin{align*}
\frac{d}{dt} \varphi_t^* \alpha_t &= \varphi_t^*\big(\dot\alpha_t +
\Lie_{X_t}\alpha_t\big)\\
&= \varphi_t^*\big(\dot\alpha_t + \iota_{X_t} d\alpha_t\big).
\end{align*}
By construction, the term in the parenthesis vanishes on $\xi_t$ so it is
$\alpha_t$ multiplied by some function $\mu_t$ and we get:
\[
\frac{d}{dt} \varphi_t^* \alpha_t = (\mu_t \circ \varphi_t)
\varphi_t^*\alpha_t.
\]
So $\varphi_t^*\alpha_t$ stays on a line in the space of one forms.
This line is obviously the line spanned by $\varphi_0^*\alpha_0 = \alpha_0$ and
we then have $\ker \varphi_t^*\alpha_t = \ker \alpha_0 = \xi_0$ for all $t$.
It is not hard to see that $X_t$ is the only Legendrian vector field which
works.
\end{proof}

\begin{indented}
	Some compactness assumption is indeed necessary in Gray's theorem. There are
	counter-examples on $\R^2 \times \S^1$ discovered in \cite{Eliashberg_open}.
	
	Contact structures form an open set in the space of all plane fields.
	Gray's theorem proves that isotopy classes of contact structures on a closed
	manifold are actually connected components of this open set. In particular
	there are only finitely many isotopy classes of contact structures on a
	closed manifold.	

	The example of linear foliations on $T^3$ proves that Gray's theorem
	wouldn't hold for foliations.
\end{indented}

\subsection{Libermann's theorem on contact Hamiltonians}

Contact transformations of a contact manifold $(V, \xi)$ are diffeomorphisms of
$V$ which preserve $\xi$. The infinitesimal version of these are vector fields
whose flow consists of contact transformations. They are called contact vector
fields and are exactly those $X$ for which 
$\left(\Lie_X \alpha\right)_{|\xi} = 0$
for any contact form $\alpha$ defining $\xi$. Note that this condition is
weaker than $\Lie_X \alpha = 0$ which would imply that the flow of $X$
preserves $\alpha$ and not only its kernel $\xi$.

In the proof of Gray's theorem, we saw that one can rotate a contact structure
at will using the flow of a Legendrian vector field uniquely determined by the
rotation we want to achieve. The same idea allows to prove that any vector
field on a contact manifold can be transformed into a contact vector field by
adding a uniquely determined Legendrian vector field.  This is the geometric
fact underlying the existence of so-called contact Hamiltonians.

\begin{theorem}[Libermann \cite{Libermann}]
\label{thm:lib}
On a contact manifold $(V, \xi)$ the map which sends a contact vector field to
its reduction modulo $\xi$ is an isomorphism from the space of contact vector
fields to the space of sections of the normal bundle $TV/\xi$.
\end{theorem}

If we single out a contact form $\alpha$ then we get a trivialization
$TV/\xi \to V \times \R$ given by $(x, [u]) \mapsto (x, \alpha(u))$.
Sections of $TV/\xi$ can then be seen as functions on $V$ and the contact
vector field $X_f$ associated to a function $f$ using the preceding theorem is
called the Hamiltonian vector field coming from $\alpha$ and $f$.
Libermann's theorem both implies existence of $X_f$ and the fact that it is the
unique contact vector field satisfying $\alpha(X) = f$.
The situation is analogous to the case of Hamiltonian vector fields in
symplectic geometry but in the symplectic case there are symplectic vector
fields that are not Hamiltonian.
Note that the above interpretation when a contact form is fixed is
what Libermann originally discussed and also the most common use of the word
contact Hamiltonian.

\begin{proof}[Proof of Theorem~\ref{thm:lib}]
Let $X$ be any vector field on $V$. The theorem is equivalent to the assertion
that there is a unique Legendrian vector field $X_\xi$ such that 
$X + X_\xi$ is contact. Using any contact form $\alpha$, we have equivalent
reformulations:
\begin{align*}
X + X_\xi \text{ is contact} &\iff
\left(\Lie_{X + X_\xi}\alpha\right)_{|\xi} = 0 \\
&\iff 
\left(\iota_{X + X_\xi}d\alpha + d(\iota_X\alpha)\right)_{|\xi} = 0 \\
&\iff
\left(\iota_{X_\xi}d\alpha\right)_{|\xi} = -\left( \iota_X d\alpha+ d(\iota_X\alpha)\right)_{|\xi} \\
\end{align*}
and the later condition defines uniquely $X_\xi$ because $d\alpha_{|\xi}$ is
non-degenerate.
\end{proof}

\begin{remark}
\label{rem:cutoff}
A common use of contact Hamiltonians, and the only one we will need, is to
cut-off or extend a contact vector field. For instance if $X$ is a contact
vector field defined on an open set $U \subset V$ and $F$ is a closed subset 
of $V$ contained in $U$ then there is a contact vector field $\tilde X$
which vanishes outside $U$ and equals $X$ on $F$. If $L$ denotes the isomorphism
of Theorem~\ref{thm:lib} and $\rho$ is a function with support in $U$ such that
$\rho_{|F} \equiv 1$ then we can use $\tilde X = L^{-1}(\rho L(X))$.
\end{remark}

\chapter*{Setting the goals: the tight vs overtwisted dichotomy}

After the local theory and before starting our study of convex surfaces, we need
some motivation.

In Figure \ref{fig:std2} showing $\ker(dz + r^2 d\theta)$, the contact planes
rotate along rays perpendicular to the $z$-axis but are never horizontal away
from the $z$-axis.
On the other extreme one can instead consider a contact structure which turns
infinitely many times along these rays. A possible contact form for this is
$\cos(r)\,dz + r\sin(r)\, d\theta$ which is horizontal for each $r$ such that
$\sin(r) = 0$, ie $r =k\pi$. Figure \ref{fig:disque_vrille} shows what happens
along $z = 0$ and $r \leq \pi$. One sees a disk whose tangent space agrees with
$\xi$ at the center and along the boundary.
\begin{figure}[htp]
  \begin{center}
	\includegraphics[height=7cm]{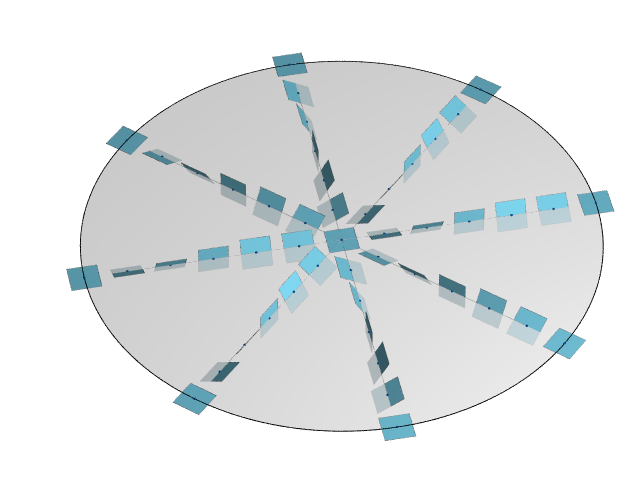}
  \end{center}
  \caption{An overtwisted contact structure}
  \label{fig:disque_vrille}
\end{figure}
\begin{definition}[Eliashberg]
A contact manifold is \emph{overtwisted} if it contains an embedded disk along which
the contact structure is as in Figure \ref{fig:disque_vrille}: the contact
structure $\xi$ is tangent to the disk in the center and along the boundary and
tangent to rays from the center to the boundary. A contact structure which is
not overtwisted is called \emph{tight}.
\end{definition}

It may look like this is the beginning of an infinite series of definitions
where ones looks at disks $z = 0$, $r \leq k\pi$ in the model above. But this
would bring nothing new as can be seen from the following exercise.

\begin{exercise}
Prove that any neighborhood of an overtwisted disk in a contact manifold
contains a whole copy of $(\R^3, \xiOT)$ where 
$\xiOT = \ker\big(\cos(r)\,dz + r\sin(r)\, d\theta\big)$.
\end{exercise}

The above exercise is pretty challenging at this stage but it can serve as a
motivation for the technology at the beginning of the next chapter. And, most
of all, it shows that not immediately seeing something in a contact manifold
does not mean it is not there (recall also Figure~\ref{fig:lien13}). This
begins to highlight the depth of the following two results whose proof is the
main goal of these lecture notes.

\begin{theorem}[Bennequin 1982 \cite{Bennequin}]
\label{thm:bennequin}
The standard contact structures on $\R^3$ and $\S^3$ are tight.
\end{theorem}

\begin{theorem}[Eliashberg 1992 \cite{Eliashberg_20_ans}]
All tight contact structures on $\R^3$ or $\S^3$ are isomorphic to the standard
ones.
\end{theorem}

Bennequin's theorem shows in particular that the standard contact structure on
$\R^3$ is not isomorphic to the overtwisted structure of Figure
\ref{fig:disque_vrille}.
In order to put this in perspective, recall that Figures \ref{fig:std3}
and \ref{fig:std1} show isomorphic contact structures. 
It may look like the difference between these is analogous to the difference
between Figure \ref{fig:std2} and \ref{fig:disque_vrille}. But Bennequin's
theorem proves that the later two pictures are really different.

Eliashberg's theorem shows that tight contact structures on $\S^3$ are rare.
By contrast, overtwisted contact structures abound. The Lutz--Martinet theorem,
revisited by Eliashberg, says that, on a closed oriented manifold, any plane
field is homotopic to an overtwisted contact structure \cite{Eliashberg_vrille}.
Recall that, because the Euler characteristic of a 3--manifold
always vanishes, all such manifolds have plane fields and even more, there are
always infinitely many homotopy classes of plane fields (for the classification
of homotopy classes of plane fields one can refer to \cite[Section
4.2]{GeigesBook}).

In \cite{CGH}, Colin, Giroux and Honda proved that only finitely
many homotopy classes of planes fields on each manifold can contain tight
contact structures. This is far beyond the scope of these lectures but see
Theorem \ref{thm:eliashberg_bennequin} for a weaker version due to Eliashberg
\cite{Eliashberg_20_ans}.

\chapter{Convex surfaces}

The goal of this chapter is to explain the following crucial observation by
Emmanuel Giroux in 1991: 
\begin{quote}
	If $S$ is a generic surface in a contact 3-manifold, all the information
	about the contact structure near $S$ is contained in an isotopy class
	of curves on $S$. 
\end{quote}
All this chapter except the last section comes from Giroux's PhD thesis
\cite{Giroux_91}, see also the webpage of Daniel Mathews for his translation of
that paper into English.

\section{Characteristic foliations of surfaces}
\label{sec:char_foliation}
After the local theory which explains what happens in neighborhoods of points
in contact manifolds, we want to start the semi-local theory which deals with
neighborhoods of surfaces.

The main tool will be characteristic foliations. The basic idea is to look at
the singular foliation given on a surface $S$ by the line field 
$TS \cap \xi$, see Figure~\ref{fig:def_xiS}.
\begin{figure}[htp]
	\begin{center}
			\includegraphics[width=5cm]{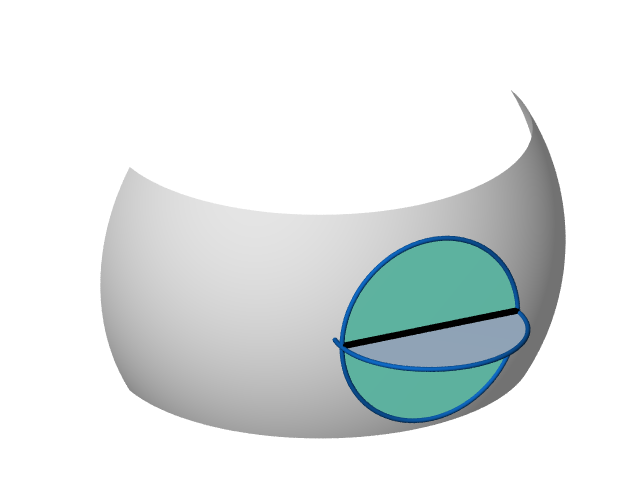}
	\end{center}
	\caption{Characteristic foliation of a surface as the intersection between the
	tangent space and the contact plane.}
	\label{fig:def_xiS}
\end{figure}

In order to define precisely what is a line field with singularities, we see
them as vector fields whose scale has been forgotten. It means they are
equivalence classes of vector fields where $X \sim Y$ if there is a positive
function $f$ such that $X = fY$. A singularity is then a point where some,
hence all, representative vanishes.  Note that $f$ should be positive
everywhere, including singularities.

One can think of a line as the kernel of a linear form rather than a subspace
spanned by a vector. This prompts an equivalent definition as  an equivalence
class of 1--forms where $\alpha \sim \beta$ if there is a positive function $f$
such that $\alpha = f\beta$.

To go from one point of view to the dual one, we can use an area form $\omega$
on the surface. The correspondence between vector fields and 1--forms is then
given by $X \mapsto \beta := \iota_X \omega$. The singular foliations $[X]$
defined by $X$ and $[\beta]$ defined by $\beta$ are indeed geometrically the
same since $X$ and $\beta$ vanish at the same points and elsewhere $X$ spans
$\ker \beta$.  In addition, one has the following commutative diagram which will
be useful later.

\begin{equation}
\label{eqn:forms_vectors}
\begin{CD}
	\text{vector fields} @>\sim>{\iota_\bullet \omega}>   \text{1--forms}\\
	@V{\Div}VV                  @VV{d}V \\
	\text{functions}     @>\sim>{\bullet \omega}>   \text{2--forms}\\
\end{CD}
\end{equation}

The left-hand side vertical arrow is the divergence map defined by the equality
$\Lie_X \omega = (\Div X) \omega$. So positive divergence means the flow of $X$
expands area while negative divergence means area contraction.
Divergence is not well defined for a singular foliation because it depends
on the representative vector field.  However, at a singularity of a foliation,
the sign of divergence is well defined because 
\[
\Lie_{fX}\omega = df \wedge \iota_X\omega + f(\Div X) \omega
\]
so, at points where $X$ vanishes, $\Div fX = f\Div X$. The same kind of
computation proves that this sign doesn't depend on the choice of the area form
within a given orientation class.

\begin{definition}
Let $S$ be an oriented surface in a contact manifold $(M, \xi)$ with $\xi = \ker
\alpha,$ co-oriented by $\alpha$. The characteristic foliation $\xi S$ of $S$ is
the equivalence class of the 1--form $\iota^*\alpha$ induced by $\alpha$ on $S$. 
\end{definition}

In particular, singularities of the characteristic foliation $\xi S$ are points
where $\xi = TS$ (maybe with reversed orientation).
At those points $d\iota^*\alpha = d\alpha_{|\xi}$ is non-degenerate so the
above commutative diagram proves that singularities of characteristic
foliations have non-zero divergence.

\paragraph{Examples}

Figures \ref{fig:xiS_spheres}, \ref{fig:xiS_torus_cvx} and \ref{fig:xiS_prelag}
show examples of characteristic foliations.

\begin{figure}[htp]
	\begin{center}
		\includegraphics[width=6cm]{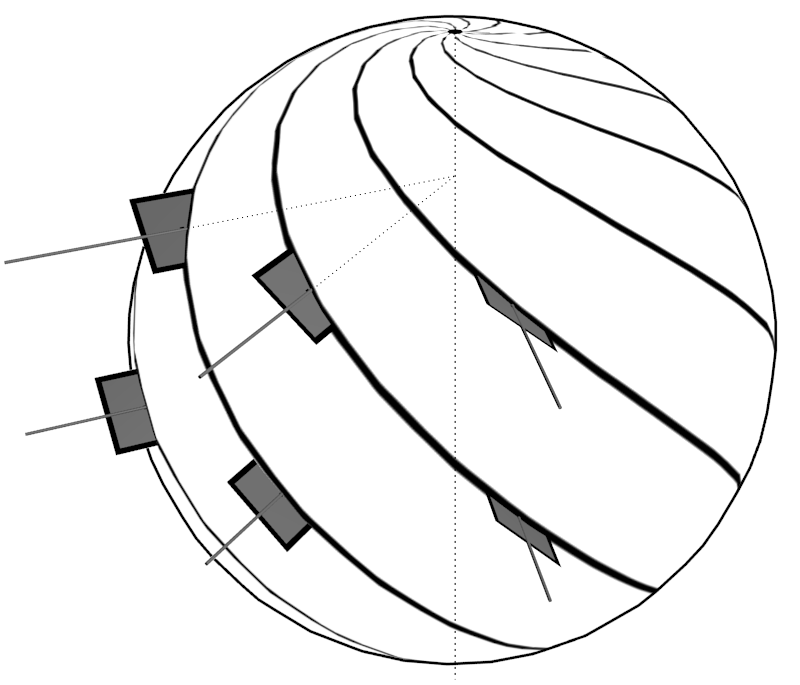}
	\end{center}
	\caption{Characteristic foliation of Euclidean spheres around the origin in
	$\R^3$ equipped with the canonical contact structure 
	$\xi = \ker (dz + r^2 d\theta)$. There are singular points at the intersection
	with the $z$-axis and all regular leaves go from a singularity to the other
	one.}
	\label{fig:xiS_spheres}
\end{figure}

\begin{figure}[htp]
	\begin{center}
		\includegraphics[width=8cm]{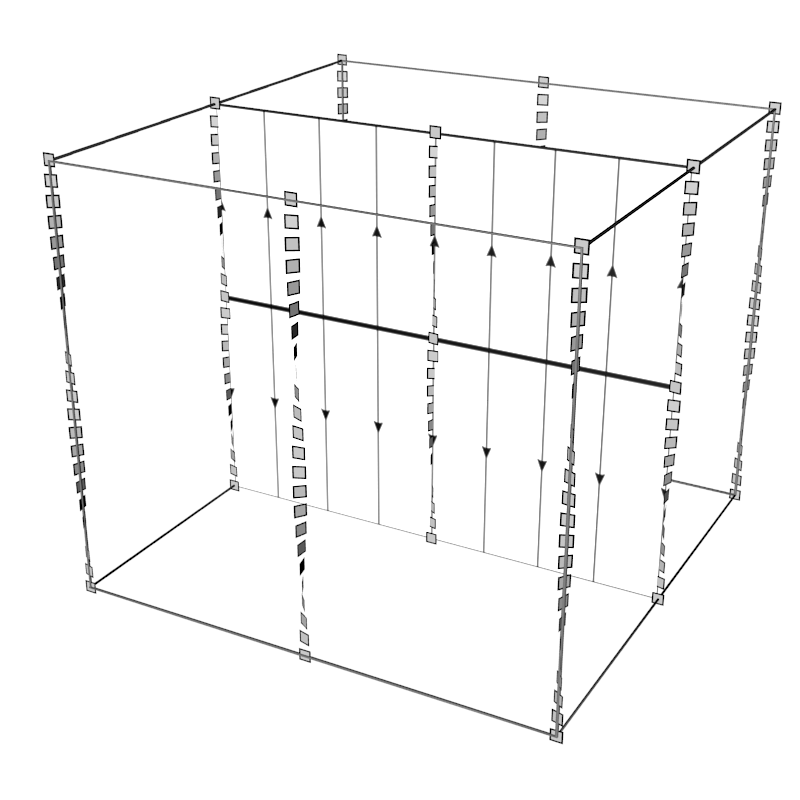}
	\end{center}
	\caption{Characteristic foliation of a torus $\{x = \text{constant}\}$ in
	$T^3$ equipped with its canonical contact structure 
	$\xi = \ker (\cos(z)dx - \sin(z)dy)$. One can see two circles made entirely of
	singularities where $\sin(z) = 0$, one appear in the middle of the picture and
	the other one can be seen both at bottom and at top.}
	\label{fig:xiS_torus_cvx}
\end{figure}

\begin{figure}[htp]
	\begin{center}
		\includegraphics[width=8cm]{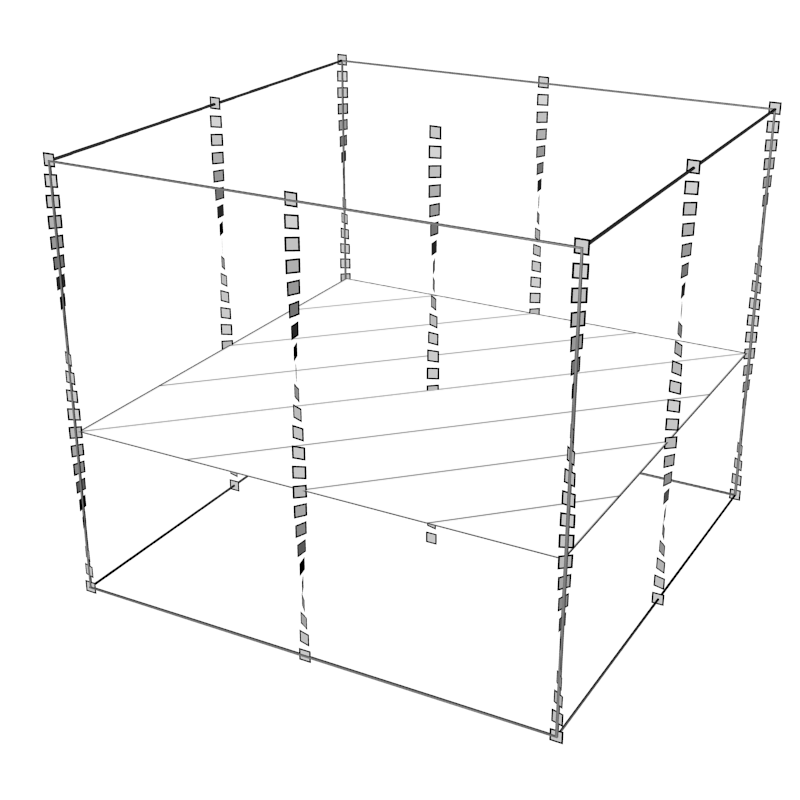}
	\end{center}
	\caption{Characteristic foliation of a torus $\{z = \text{constant}\}$ in
	$T^3$ equipped with its canonical contact structure 
	$\xi = \ker (\cos(z)dx - \sin(z)dy)$.}
	\label{fig:xiS_prelag}
\end{figure}

\subsection{Leaves of characteristics foliations}

The leaves (or orbits) of a singular foliation are the integral curves of any
vector field representing it. The intuitive notion of a singular foliation is
rather the data of leaves than an equivalence class of vector fields. In
contact geometry, this discrepancy does not generate any confusion thanks to
the following lemma.  It is a rather technical point but we discuss it here
anyway because it doesn't appear to be published anywhere else, although it is
mentioned in \cite[page 629]{Giroux_2000}.

\begin{lemma}[Giroux]
If two singular foliations on a surface have the same leaves and if their
singularities have non-zero divergence then they are equal.
\end{lemma}

The following proof can be safely skipped on first reading.

\begin{proof}
The statement is clear away from
singularities and a partition of unity argument brings it down to a purely
local statement. So we focus on a neighborhood of a singularity (which may  be
non-isolated though).

Let $Y$ and $Y'$ be vector fields on $\R^2$ which vanish at the origin and have
the same orbits.
\[
Y = f\partial_x + g\partial_y\quad \text{ et }\quad Y' = f'\partial_x + g'\partial_y.
\]
We will compute divergence using the Euclidean area form 
$\omega = dx \wedge dy$ (we know the sign of divergence of singular points does
not depend on this choice). So $\Div Y = \partial_x f + \partial_y g$.
All the following assertions will be true in a neighborhood of the origin that
will shrink only finitely many times. Since $\Div(Y)$ is non-zero, we can use a
linear coordinate change to ensure that $\partial_x f$ doesn't vanish.  The
implicit function theorem then gives new coordinates such that $f(x,y) = x$.
Because 
\[
f'(x,y) = f'(0,y) + x \int_0^1 \partial_x f'(t x, y) dt
\]
we can write $f' = x u(x,y) + v(y)$. Along the curve $\{x = 0\}$, the vector
field $Y$ is vertical (or zero) so the same is true for $Y'$. Hence $f'$ also
vanishes along this curve and $v$ is identically zero.  The condition that $Y$
and $Y'$ are either simultaneously zero or colinear is then:
\[
\left|
\begin{matrix}
x & x u \\
g & g'
\end{matrix}
\right| = 0
\]
which gives $g' = u g$ where $x$ is non-zero hence everywhere by continuity.
One then gets $Y' = uY$. 
In particular $\Div Y' = u\Div Y + du \wedge (\iota_Y dx \wedge dy)$. Away from
zeros of $Y$ and $Y'$, $u$ is positive because $Y$ and $Y'$ have the same
leaves. At a common zero, $\Div Y' = u\Div Y$ and, because singularities of
$Y'$ have non-zero divergence, the function $u$ doesn't vanish. Hence it is
positive everywhere (note that $Y$ and $Y'$ can't be everywhere zero).
\end{proof}

\section{Neighborhoods of surfaces}

Any orientable surface $S$ in an orientable 3-manifold has a neighborhood
diffeomorphic to $S \times \R$ (use the flow of a vector field transverse to
$S$). We will always denote by $t$ the coordinate on $\R$ and by $S_t$ the 
surface $S \times \{t\}$ for a fixed $t$. From now on, we will assume that $S$
is oriented and orient $S \times \R$ as a product.

Any plane field $\xi$ defined near $S$ has then an equation 
$\alpha = u_t dt + \beta_t$ where $u_t$ is a
family of functions on $S$ and $\beta_t$ is a family of 1--forms on $S$.
Note that the characteristic foliation of $S_t$ is the equivalence class of
$\beta_t$ since the latter is the 1--form induced by $\alpha$ on $S_t$.

The contact condition for $\xi$ (with respect to the product orientation) is
equivalent to
\begin{gather}
	\tag{$\star$}
	\label{eqn:ccg}
	u_t d\beta_t + \beta_t \wedge (du_t - \dot\beta_t) > 0
\end{gather}
where $\dot\beta_t$ denotes $\frac{\partial \beta_t}{\partial t}$.
This condition is a non-linear partial differential relation which is not so
simple. The main thrust of the following discussion will be to simplify it
by fixing some of the terms.

\subsubsection*{Reconstruction lemmas}

The easiest case is to fix the whole family $\beta_t$. In this case the contact
condition \eqref{eqn:ccg} is only about the family $u_t$
and becomes convex. In particular the space of solutions $u_t$ is connected and
we get:

\begin{lemma}[Global reconstruction]
\label{lemma:global_reconstruction}
If $\xi$ and $\xi'$ are positive contact structures on $S \times \R$ such
that $\xi S_t = \xi' S_t$ for all $t$ then $\xi$ and $\xi'$ are isotopic.
\end{lemma}

We give a detailed proof since it is a model of several later proofs.

\begin{proof}
There are equations $u_t dt + \beta_t$ and $u'_t dt + \beta'_t$ of
$\xi$ and $\xi'$.
The hypothesis of the lemma is that $\beta'_t = f_t \beta_t$ for some family
of positive functions $f_t$ on $S$. 
So another equation for $\xi'$ is $u'_t/f_t dt + \beta_t$. 
We have two solutions $u_t$ and $u_t'/f_t$ of the contact condition, Equation
\eqref{eqn:ccg}, with $\beta_t$ fixed. Since this
condition is convex, the space of its solutions is connected so we can find a
family of solution $(u^s_t)_{s \in [0, 1]}$ relating them (a linear
interpolation will do the job).  This family corresponds to a family of contact
structures $\xi_s = \ker( u^s_t dt + \beta_t)$
which Gray's theorem (Theorem \ref{thm:gray}) converts to an isotopy
of contact structures\footnote{One may worry about the fact that $S \times \R$
is non-compact but here the vector field constructed during the proof of this
theorem is tangent to $S_t$ which is compact for all $t$ hence its flow is well
defined for all times}.

Our discussion of Gray's theorem actually tells us more about what is going on.
Recall the vector field generating the isotopy at time $s$ can be chosen in the
intersections of $\xi_s$ and $\xi_{s + \varepsilon}$. So we see the isotopy is
stationary at each singular point of the characteristic foliations $[\beta_t]$.
At all other points it is tangent to the characteristic foliation and its flow
makes the contact structures we want to relate to rotate toward each other,
see Figure \ref{fig:reconstruction}.
\begin{figure}[ht]
	\begin{center}
		\includegraphics[width=7cm]{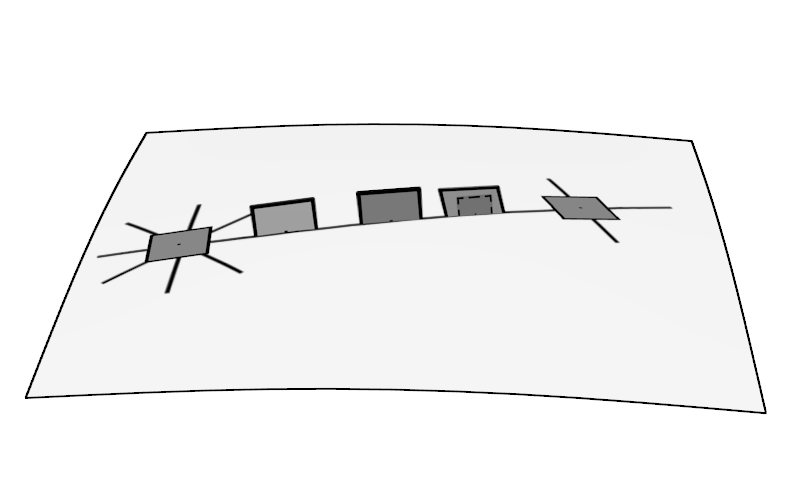}
	\end{center}
	\caption{Reconstruction lemmas. We have two contact structures printing the
	same characteristic foliation on a surface. One of them is drawn along an arc
	going from a singularity to another. The second one appears only at one point
	with dotted outline. At this point the isotopy constructed in the proof is
	tangent to the arc to make the contact structure rotate.}
	\label{fig:reconstruction}
\end{figure}
\end{proof}

If instead of fixing the whole family $\beta_t$ we fix only $\beta_0$ then 
we get the following lemma.

\begin{lemma}[Local reconstruction]
\label{lemma:local_reconstruction}
If $\xi$ and $\xi'$ are positive contact structures which prints the same
characteristic foliation on a compact embedded surface $S$ then there is a
neighborhood of $S$ on which $\xi$ and $\xi'$ are isotopic (by an isotopy
globally preserving $S$).
\end{lemma}

\begin{proof}
The contact condition along $S_0$ becomes a convex condition on $u_0$ and
$\dot \beta_0$. Again we can find a path of plane fields which, \emph{along
$S$}, are contact structures interpolating between $\xi$ and $\xi'$. Because the
contact condition is open, they will stay contact structures near $S$ and we can
use Gray's theorem again.
\end{proof}

\begin{exercise}
Prove that the two preceding lemmas are false for foliations.
\end{exercise}

We can now return to the challenging exercise of Chapter 1 with much better
chances of success. Recall that $\xiOT = \ker(\cos(r) dz + r\sin(r) d\theta)$.

\begin{exercise}
Use the local reconstruction lemma to prove that any neighborhood of an
overtwisted disk in a contact manifold contains a copy of 
$(\R^3, \xiOT)$. Hint: try to understand the 
characteristic foliation of the surface of Figure~\ref{fig:revolution}.
\end{exercise}

\begin{figure}[ht]
	\begin{center}
		\includegraphics[height=3cm]{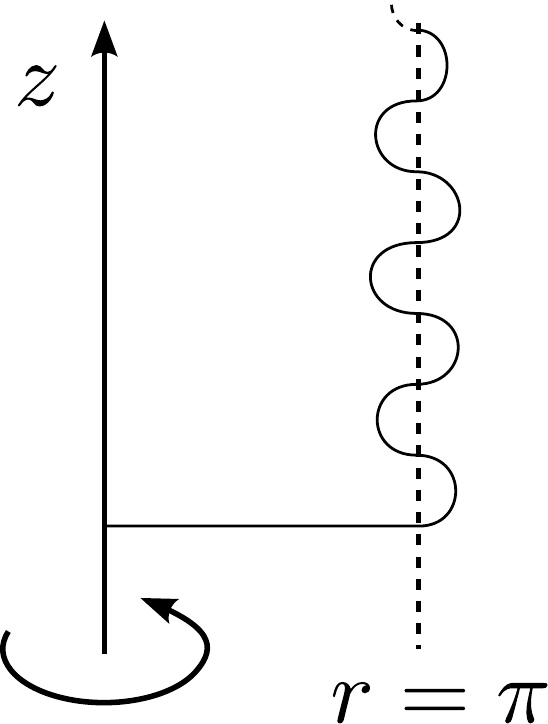}
	\end{center}
	\caption{Rotating the wavy curve around the $z$-axis in 
	$(\R^3, \xiOT)$ gives a plane having a
	characteristic foliation diffeomorphic to that of $\{z = 0\}$. Note that the
	curve is horizontal at each intersection with the $\{r = \pi \}$ axis.}
	\label{fig:revolution}
\end{figure}

As illustrated by the previous exercise, the reconstruction lemmas are already
quite useful by themselves. But the characteristic foliation is still a huge data
and it is very sensitive to perturbations of the contact structure or the
surface. This will be clear from the discussion of genericity of convex surfaces
and of the realisation lemma below.

\subsection{Convex surfaces}

\subsubsection*{Homogeneous neighborhoods}

The next step in our quest to simplify the contact condition \eqref{eqn:ccg}
seems to be fixing $u_t$ instead of $\beta_t$. But this still gives a non-linear
equation on the family $\beta_t$ if $\dot\beta_t$ is not zero. 
So we assume that $\beta_t$ does not depend on $t$: $\beta_t = \beta$. In
particular the families $(u_0, \beta)$ and $(u_t, \beta)$ both give contact
structures with the same characteristic foliation $[\beta]$ on each $S_t$. Hence
the global reconstruction Lemma tells us these contact structures are isotopic.
So we now assume that $u_t$ is also independent of $t$. 

In this situation, the contact structure itself becomes invariant under $\R$
translations, one says that $\partial_t$ is a contact vector field. Note that
this vector field is transverse to all surfaces $S_t$. Conversely if a contact
vector field is transverse to a surface then it can be cut-off away from the
surface using Remark~\ref{rem:cutoff} and then its flow defines a tubular
neighborhood $S \times \R$ with a $t$--invariant contact structure.

\begin{definition}[Giroux \cite{Giroux_91}]
A surface $S$ in a contact 3--manifold $(M, \xi)$ is $\xi$--convex if it is
transverse to a contact vector field or, equivalently, if it has a so called
homogeneous neighborhood: a tubular neighborhood $S \times \R$ where the
restriction of $\xi$ is $\R$--invariant.
\end{definition}

\begin{example}
In $T^3$ with its canonical contact structure, all tori 
$\{x = \text{constant}\}$ as in Figure \ref{fig:xiS_torus_cvx} are $\xi$--convex
since they are transverse to the contact vector field $\partial_x$.
\end{example}

\begin{example}
\label{ex:radial_field}
	In $\big(\R^3, \ker (dz + r^2 d\theta)\big)$, any Euclidean sphere around the
	origin is $\xi$--convex since they are transverse to the contact vector field
	$x\partial_x + y\partial_y + 2z\partial_z$.
\end{example}

In the convex case, the contact condition becomes: 
\begin{gather}
	\label{eqn:cccb}
	\tag{$\dagger$} u d\beta + \beta \wedge du > 0 
\end{gather} 
Using some area form $\omega$ and Equation \eqref{eqn:forms_vectors}, one can
rephrase it in terms of the vector field $Y$ $\omega$--dual to $\beta$ as: 
\begin{gather} 
	\label{eqn:cccY}
	\tag{$\dagger'$} u \Div_\omega Y - du(Y) > 0 
\end{gather}
Analogously to the previous section we see that, $u$ being fixed, the space of
solutions $\beta$ to \eqref{eqn:cccb} is contractible, this was our stated goal
when we asked $\beta_t$ to be independent of $t$. The miracle is that it
essentially stays true if one fixes only the zero set $\Gamma$ of $u$. Indeed,
away from $\Gamma$, we can divide our contact form $udt + \beta$ by $|u|$
to replace it by $\ker( \pm dt + \beta')$ where $\beta' = \frac 1{|u|} \beta$.
The condition \eqref{eqn:cccb} for $(\pm 1, \beta')$ is simply $\pm d\beta' > 0$
which is not only convex, it does not depend on $u$! Of course this discussion
needs some precise definitions which are provided below but the first miracle
has already happened: near a $\xi$--convex surface $S$, all the information
about $\xi$ is contained in $\Gamma$. It remains to see that such surfaces
are generic, the second miracle.

\subsubsection*{Dividing sets}

Let us take a look at $\Gamma = \{u = 0\}$. 
Along $\Gamma$, the contact condition \eqref{eqn:cccY} reads $-du(Y) > 0$.
So $\Gamma$ is a regular level set of $u$. Hence it is a one-dimensional
submanifold without boundary, ie a collection of disjoint simple closed curves
in $S$. Such collections will be referred to as multi-curves. 

The condition $-du(Y) > 0$ also implies that $\Gamma$ is transverse to $\xi S$.
More precisely, $Y$ goes from $S_+ = \{u > 0\}$ to $S_- = \{u < 0\}$ along
$\Gamma$ and the picture near $\Gamma$ is always as in Figure
\ref{fig:neighb_Gamma}. In the following discussion we will use several time the
fact that this picture is very simple and controlled to be less precise about
what happens near $\Gamma$.
\begin{figure}[htp]
	\begin{center}
		\includegraphics[width=5cm]{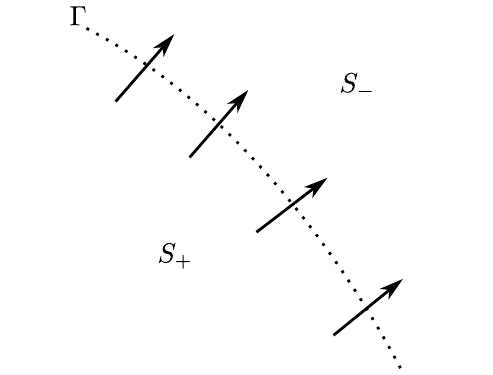}
	\end{center}
	\caption{Characteristic foliation near the dividing set $\Gamma$}
	\label{fig:neighb_Gamma}
\end{figure}

The last remarkable property of the decomposition of $S$ in $S_+$ and $S_-$ is
$Y$ expands some area form in $S_+$ and contracts it in $S_-$.
Indeed, if one sets $\Omega = \frac 1{|u|}\omega$ on $S \setminus \Gamma$ then 
$\Div_\Omega Y = \pm \frac 1{u^2}$ on $S_\pm$. One can actually modify $\Omega$
near $\Gamma$ so that $\Div_\Omega Y$ is positive on $S_+$, negative on $S_-$
and vanishes along $\Gamma$.

\begin{definition}
A singular foliation $\F$ of a surface $S$ is divided by an (embedded) multi-curve 
$\Gamma$ if there is some area form $\Omega$ on $S$ and a vector field $Y$
directing $\F$ such that:
\begin{itemize}
\item 
the divergence of $Y$ does not vanish outside $\Gamma$ --we set 
\[
S_\pm = \{p \in S ;\; \pm\Div_\Omega Y(p) > 0\}
\]
\item
the vector field $Y$ goes transversely out of $S_+$ and into $S_-$ along
$\Gamma$.  
\end{itemize}
\end{definition}

What we proved above is that the characteristic foliation of a $\xi$--convex
surface is divided by some multi-curve. Using the local reconstruction lemma
(Lemma \ref{lemma:local_reconstruction}), one can prove the converse to get:

\begin{proposition}
\label{prop:convex_is_divided}
A surface $S$ is $\xi$--convex if and only if $\xi S$ is divided.
\end{proposition}

\begin{proof}
We assume that $\xi S$ is divided by some multi-curve $\Gamma$. According to the
local reconstruction lemma, we only need to prove that there is a contact
structure $\xi'$ defined near $S$ such that $S$ is $\xi'$--convex and 
$\xi'S = \xi S$.
We set $\beta = \iota_Y \Omega$. In particular $\xi S = [\beta]$.
On $S \setminus \Gamma$, $\xi' = \ker \pm dt + \beta$ is a contact structure
which also prints $[\beta]$ on $S \setminus \Gamma$ and one can check that there
is no problem to extend it along $\Gamma$.
\end{proof}

Note that the dividing set is not unique for a given foliation. If $X$ is a
contact vector field transverse to the surface $S$ then the considerations above
prove that $\Gamma_X := \{ s \in S;\; X(s) \in \xi\}$ is a dividing set for $S$.

However, if one fixes $\beta$ in the contact condition \eqref{eqn:cccb}, it
becomes convex in $u$, hence the space of solutions $u$ is connected. This
implies that the space of multi-curves dividing a given foliation is connected
(in fact contractible).

\paragraph{Examples}

In the case of spheres of example \ref{ex:radial_field}, the dividing set
corresponding to the given vector field is the equator $\{ z = 0 \}$.

In the torus case of Figure \ref{fig:xiS_torus_cvx}, the dividing set coming
from $\partial_x$ is defined by $\cos(z) = 0$ so it is made of two circles
sitting between the singularity circles defined by $\sin(z) = 0$, see Figure
\ref{fig:gamma_convex_torus}.
\begin{figure}[ht]
	\begin{center}
		\includegraphics[width=10cm]{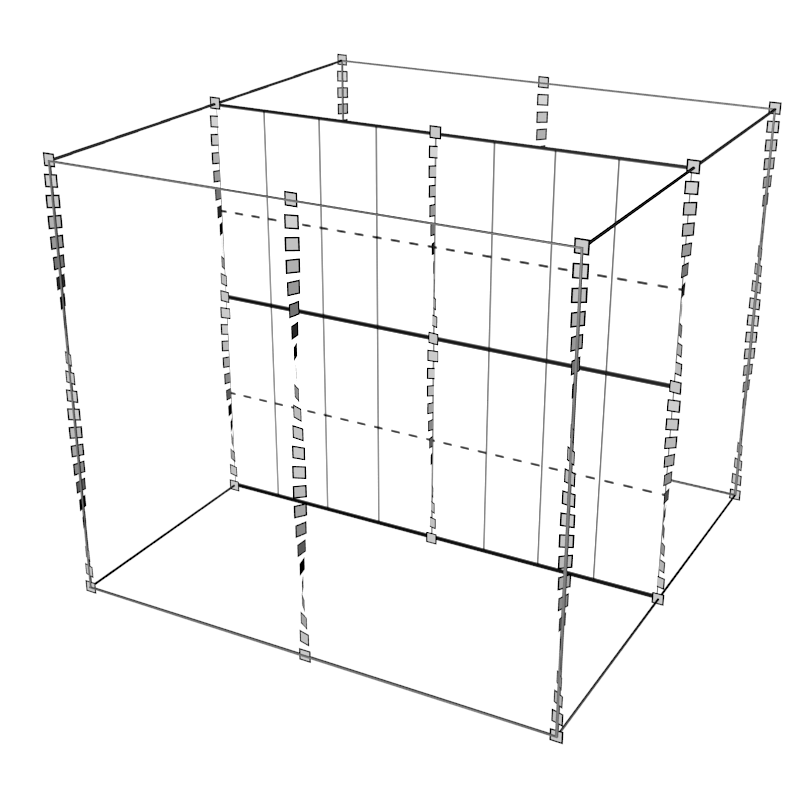}
	\end{center}
	\caption{A dividing set for the torus of Figure \ref{fig:xiS_torus_cvx}
	(dashed on the picture)}
	\label{fig:gamma_convex_torus}
\end{figure}

\subsubsection*{The realization lemma}

We are now ready to make precise the fact that the dividing set contains all the
information about the contact structure near a convex surface.

\begin{lemma}[Realization Lemma]
\label{lemma:realization}
\index{lemme!de réalisation}
Let $S$ be a $\xi$--convex surface divided by some multi-curve $\Gamma$. For any
singular foliation $\F$ divided by $\Gamma$, there is an isotopy $\delta_t$ with
support in an arbitrarily small neighborhood of $S$ and such that 
$\xi' = \delta_1^*\xi$ satisfies $\xi' S = \F$. 
Equivalently, one has $\xi \delta_1(S) = \delta_1(\F)$.
\end{lemma}

So any singular foliation divided by $\Gamma$ is printed on $S$ by some contact
structure isotopic to $\xi$ or, equivalently, it can be realized as the
characteristic foliation of a surface isotopic to $S$.

The proof of this very important lemma has already been essentially explained
right after stating condition \eqref{eqn:cccb}. It follows from the fact that
$\pm d\beta > 0$ is a convex condition and Gray's theorem as in the
reconstruction lemmas.

This lemma is often called Giroux's flexibility theorem but one can argue that
it is rather a rigidity result since all the information can be stored into a
tiny combinatorial data: the isotopy class of the dividing set.

\paragraph{Example}

Consider the convex torus of Figure \ref{fig:xiS_torus_cvx}. Its characteristic
foliation is highly non generic since it has two circles of singularities.  Yet
it is divided by two circles parallel to the singularity circles. 
Figure \ref{fig:xiS_MS_torus} shows a generic foliation divided by the same
curves but where singular circles have been replaced by regular closed leaves. 
\begin{figure}[htp]
	\begin{center}
		\includegraphics{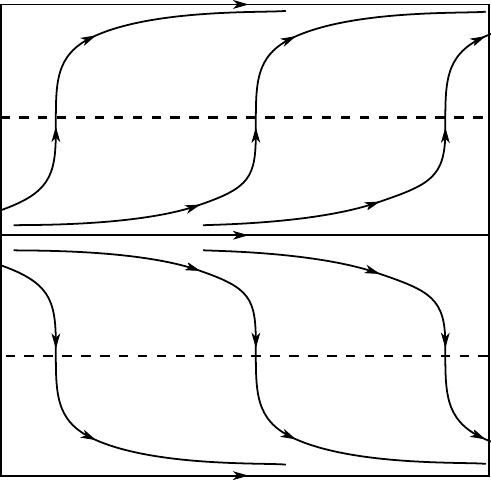}
	\end{center}
	\caption{A generic foliation of the torus divided by two curves}
	\label{fig:xiS_MS_torus}
\end{figure}

The realization lemma implies that the surface of
Figure~\ref{fig:xiS_torus_cvx} is isotopic to a surface which has
Figure~\ref{fig:xiS_MS_torus} as its characteristic foliation.
Figure \ref{fig:realised_torus} shows this surface explicitly.
\begin{indented}
The transition between these foliations play an important role in the
classification of tight contact structures on the product of a torus and an
interval, see \cite[Section 1.F]{Giroux_2000}
\end{indented}

\begin{figure}[htp]
	\begin{center}
		\includegraphics[width=11cm]{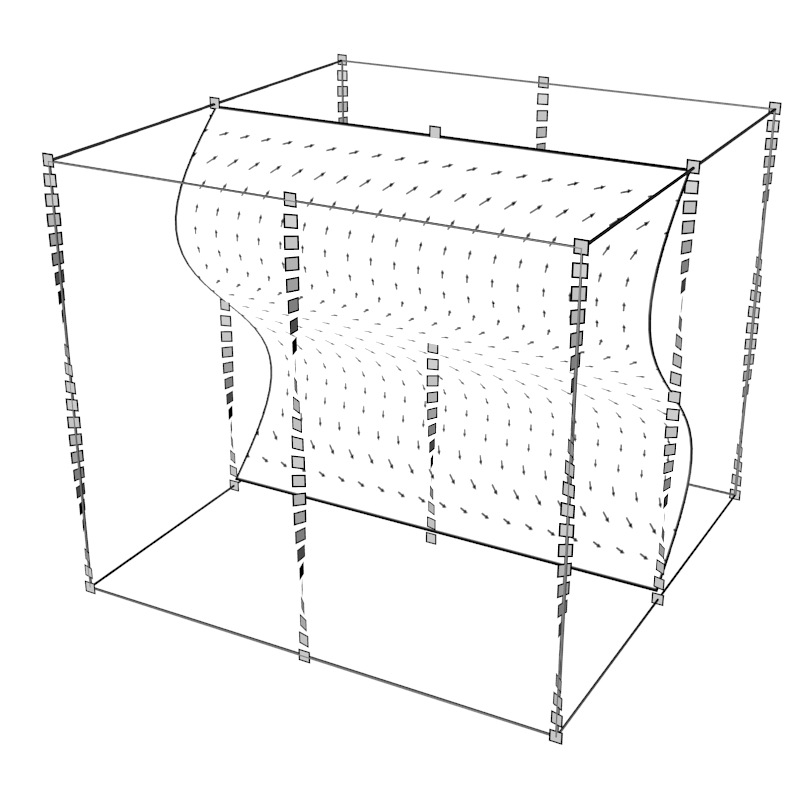}
	\end{center}
	\caption{A realization of Figure \ref{fig:xiS_MS_torus} as a deformation of
	the torus of Figure \ref{fig:xiS_torus_cvx}}
	\label{fig:realised_torus}
\end{figure}

In order to use the power of the realization lemma, we need to prove that
$\xi$-convex surfaces exist in abundance. We will first discuss some
obstructions to $\xi$-convexity then prove genericity of $\xi$-convex surfaces.

\subsection{Obstructions to convexity}

\subsubsection*{Degenerate closed leaves}

The most obvious obstruction to $\xi$-convexity for a closed surface $S$ is 
when $\xi S$ is defined by some $\beta$ with $d\beta = 0$, as in Figure
\ref{fig:xiS_prelag}, because then the contact condition \eqref{eqn:cccb} becomes
$\beta \wedge du > 0$ which implies that $u$ has no critical point. 
\begin{indented}
Surfaces with such characteristic foliations are called pre-Lagrangian. They are
either tori or Klein bottles and play an important role in some later part of
the theory.
\end{indented}

This obstruction idea can be extended remarking that it does not need the
whole of $S$, it can be applied along a closed leaf $L$ of $\xi S$. This is
easier to see in the dual picture of equation \eqref{eqn:cccY}.  Indeed, if
$\Div_\omega(Y)$ vanishes along $L$, condition \eqref{eqn:cccY} says that
$-u_{|L}' > 0$ whereas the restriction $u_{|L}$ necessarily has some critical
point.

\begin{definition}
	A closed leaf $L$ of a singular foliation is degenerate if there is a 1--form
	$\beta$ defining the foliation near $L$ and whose differential $d\beta$
	vanishes along $L$.  A non-degenerate leaf is called repelling (resp
	attracting) if there is some $\beta$ such that $d\beta$ is positive (resp
	negative) along $L$.
\end{definition}

The definition above is convenient for our purposes but one should keep in mind
that it is equivalent to the more geometrical definition through Poincar\'e's
first return map $\pi$ on a transverse curve $c$, see Figure
\ref{fig:poincare}. A closed leaf is degenerate if $\pi'(0) = 1$. A
non-degenerate closed leaf is attracting if $\pi'(0) < 1$ and repelling if
$\pi'(0) > 1$.
\begin{figure}[htp]
	\begin{center}
		\includegraphics{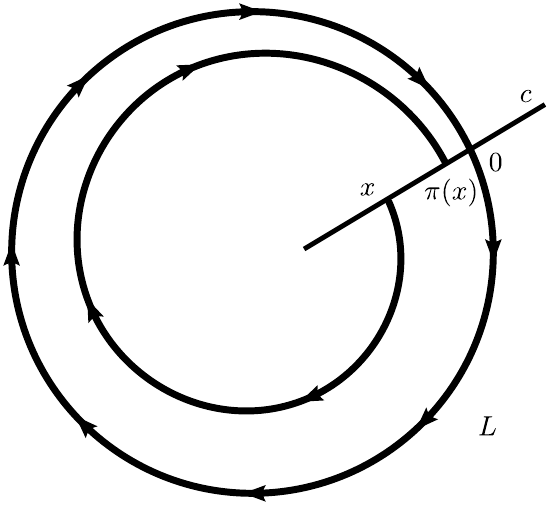}
	\end{center}
	\caption{Poincar\'e's first return map $\pi$ on a transversal $c$ to a closed
	leaf $L$.}
	\label{fig:poincare}
\end{figure}

The discussion preceding the definition proves that if $S$ is $\xi$--convex then
$\xi S$ has no degenerate closed leaves. 

\begin{figure}[htp]
	\begin{center}
		\includegraphics[width=8cm]{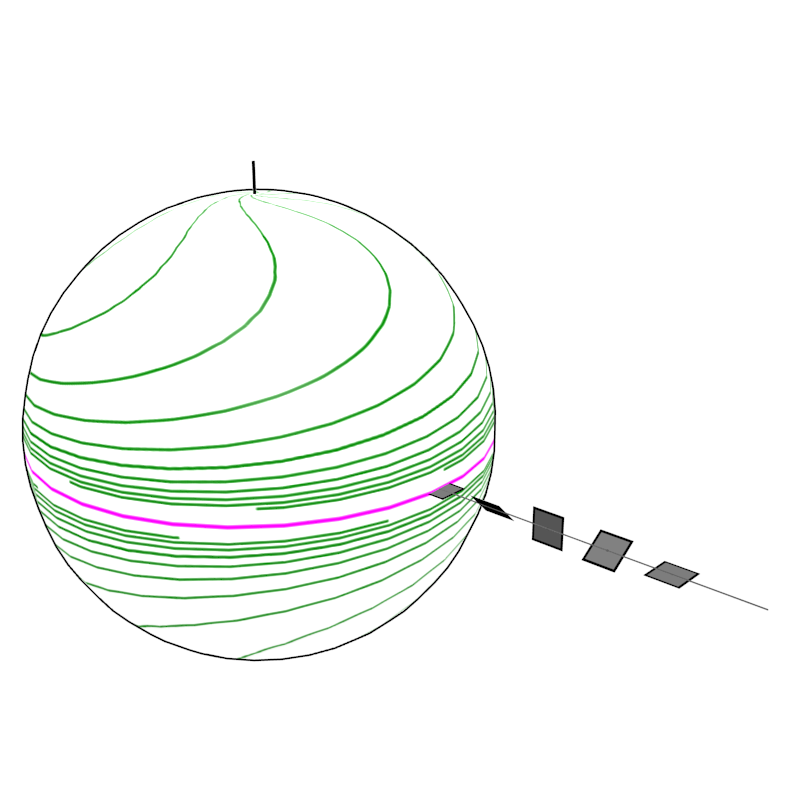}
	\end{center}
	\caption{A sphere or radius $\pi$ in the overtwisted
	$\R^3$. The equator is a degenerate closed leaf. Note how leaves spiral a lot
	more around a degenerate leaf than around a non-degenerate.}
	\label{fig:deg_leaf}
\end{figure}

\begin{figure}[htp]
	\begin{center}
		\includegraphics[width=8cm]{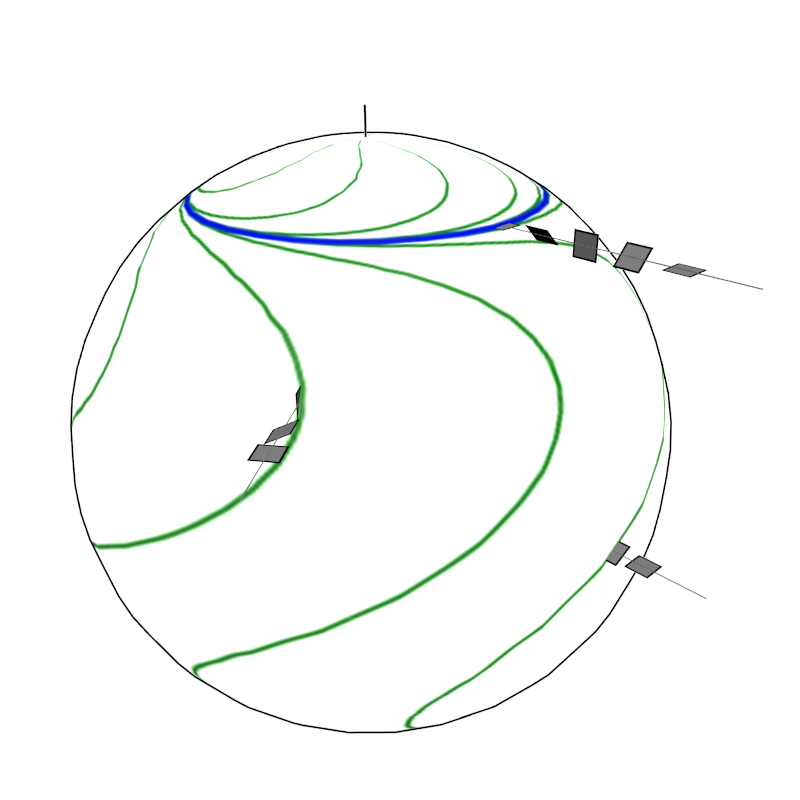}
	\end{center}
	\caption{A sphere or radius slightly less than $2\pi$ in the overtwisted
	$\R^3$. The intersection with the cylinder $\{r = \pi \}$ consists of two
	non-degenerate closed leaves (one of them is not visible in the
	picture).}
	\label{fig:non_deg_leaf}
\end{figure}

\begin{remark}
\label{rem:orbit_sign}
Suppose now that $S$ is indeed $\xi$--convex and $L$ is a (non-degenerate)
closed leaf of $\xi S$. Let $\Gamma$ be a dividing set for $\xi S$. Because 
$\xi S$ is transverse to $\Gamma$ and always goes out of $S_+$ and into $S_-$,
$L$ cannot meet $\Gamma$. Because $L$ is compact, the restriction of $u$ to $L$
has at least one critical point. At this point, the contact condition gives $u
d\beta > 0$. So repelling orbits are in $S_+$ and attracting orbits are in $S_-$.
\end{remark}

\subsubsection*{Retrograde connections}

Recall from Section \ref{sec:char_foliation} that the contact condition ensures
that all singularities of characteristic foliations have non-zero divergence
and hence have non-zero sign. Singularities of $\xi S$ correspond to points
where $S$ is tangent to $\xi$ and they are positive or negative depending on
whether the orientation of $\xi$ and $S$ match or not.

In generic characteristic foliations one sees only two topological types of
singularities: nodes and saddles. If one considers generic families of
characteristic foliations then saddle-nodes may appear, see Figure
\ref{fig:singularities}.
\begin{figure}[htp]
	\begin{center}
		\includegraphics[width=8cm]{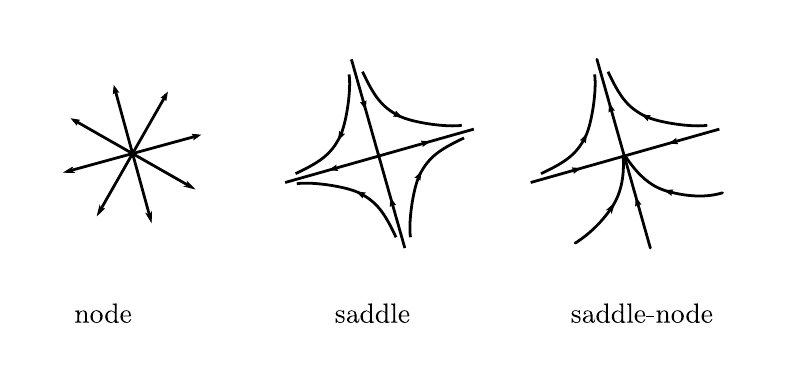}
	\end{center}
	\caption{Generic singularities of characteristic foliations}
	\label{fig:singularities}
\end{figure}
Since the sign of singularities corresponds to their divergence, positive nodes
are always sources while negative nodes are always sinks. The sign of saddles
cannot be read from topological pictures only.

Let $S$ be a $\xi$--convex surface so that $\xi = \ker(udt + \beta)$ near $S$.
We begin by a remark analogous to Remark \ref{rem:orbit_sign}.
At any singular point $p$ of $\xi S$, the contact condition \eqref{eqn:cccb}
give $ud\beta(p) > 0$. So singularities are positive in $S_+$ and negative in
$S_-$.

Suppose now that $p$ and $q$ are two singular points of $\xi S$ with opposite
signs and there is a regular leaf $L$ of $\xi S$ going from $p$ to $q$. Because
$L$ has to be transverse to $\Gamma$ and go from $S_+$ to $S_-$, the above
discussion proves that $p$ is positive and $q$ is negative.

\begin{definition}
In the characteristic foliation of a surface, a retrograde connection is a
leaf which goes from a negative singularity to a positive one.
\end{definition}

The discussion above proves that $\xi$--convex surfaces have no
retrograde connections. Note that retrograde connections cannot involve
nodes since the sign of nodes determine the local orientations of the
foliation. 

Leaves of characteristic foliations between two singularities of opposite signs
are always arcs tangent to the contact structure along which the contact
structure rotates half a turn compared to the surface. What makes retrograde
connections special is that the direction of rotation is opposite to the one
around Legendrian foliations.

\begin{example}[{\cite[Example 3.41]{Giroux_2000}}]
\label{example:connection_selle}
In $\R^2 \times \S^1$ with contact structure 
$\xi = \ker(\cos(2\pi z)dx - \sin(2\pi z)dy)$, we consider the family of
transformations 
\[
\varphi_t((x,y),z) = (R_{-4\pi t}(x,y), z+t)
\]
where $R_\theta$
denotes the rotation of angle $\theta$ around the origin of $\R^2$.
The orbit of a circle in $\R^2$ passing through the origin sweeps a torus $S$
whose characteristic foliation has two retrograde saddle connections along the
$z$-axis, see Figure \ref{fig:csr}. Indeed, along this axis, the tangent plane
$TS$ turns in the same direction as $\xi$ but twice as fast. It means that, seen
from $TS$, $\xi$ rotates one turn in the opposite direction. See Figure
\ref{fig:csr_texture} for a better view of the characteristic foliation.
\end{example}
\begin{figure}[p]
	\begin{center}
		\includegraphics[width=10cm]{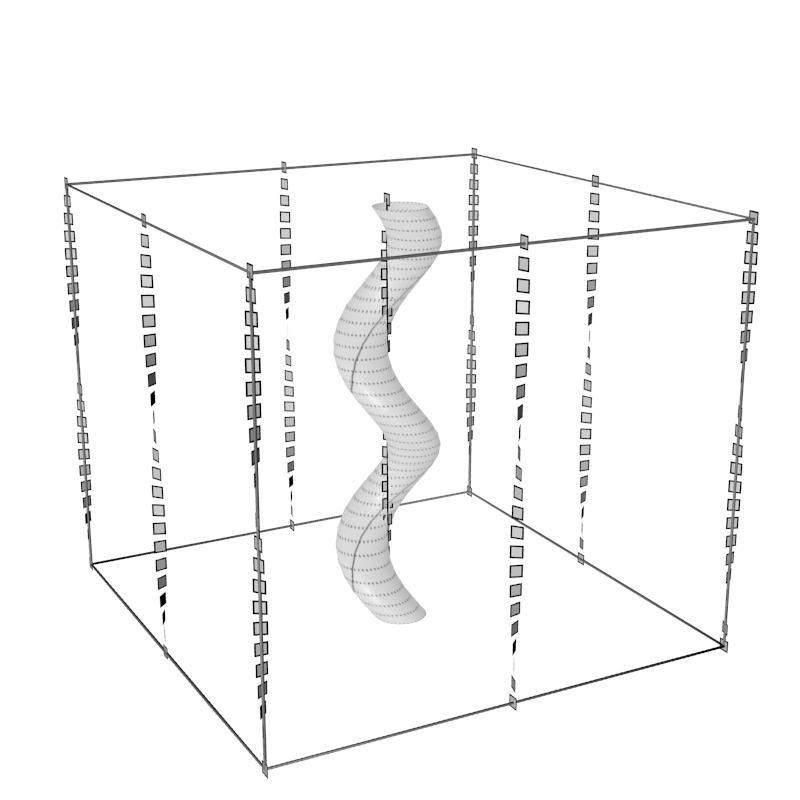}
	\end{center}
	\caption{A torus having a retrograde saddle connection}
	\label{fig:csr}
\end{figure}
\begin{figure}[p]
	\begin{center}
	\includegraphics[height=7cm]{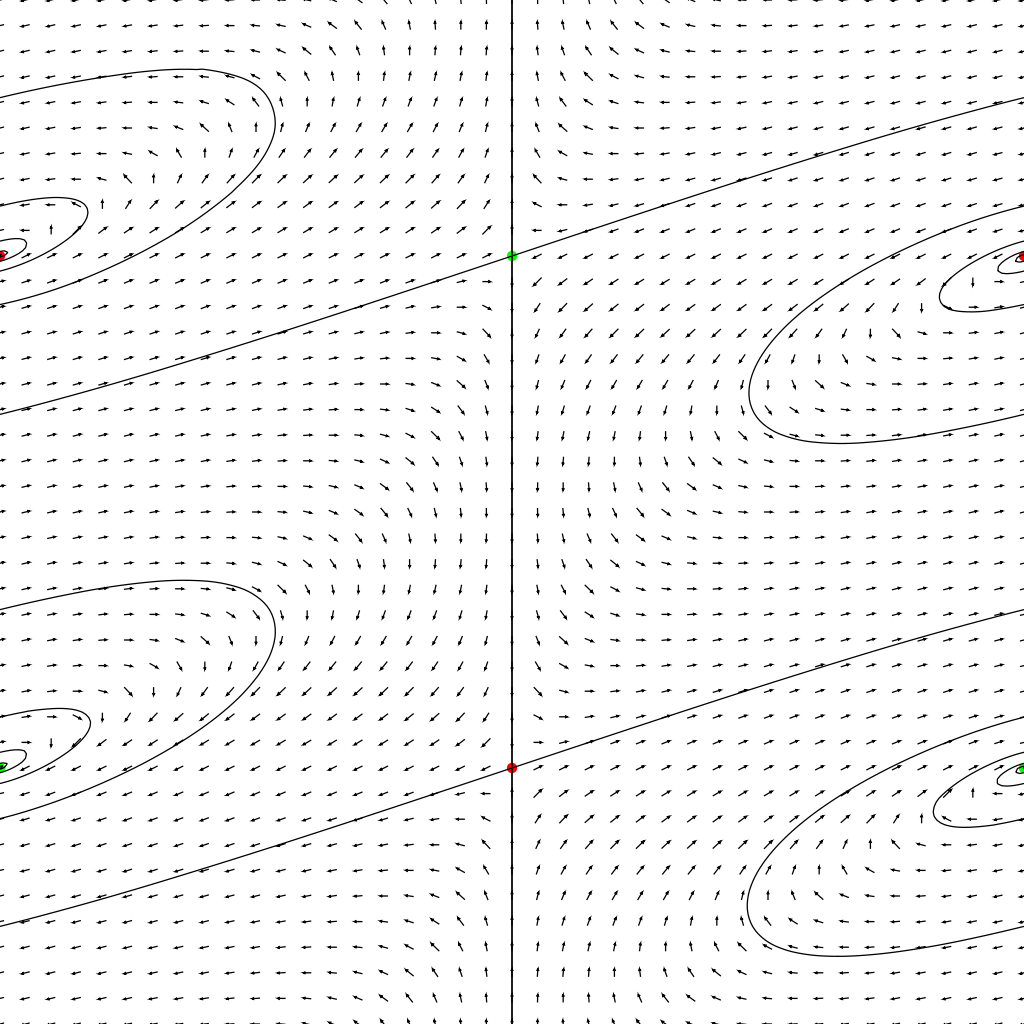}
	\end{center}
	\caption{A (double) saddle connection on the torus of Figure \ref{fig:csr}
	after top/bottom and left/right are glued. The top saddle is negative, the
	bottom one positive. The top node is positive, the bottom one negative. The
	curves drawn are all the separatrices of the saddles.}
	\label{fig:csr_texture}
\end{figure}

\subsection{Genericity of convex surfaces}

We are now ready to use generic properties of vector fields on surfaces to prove
that any surface in a contact manifold can be perturbed to a $\xi$--convex one.
See Figures \ref{fig:genericity_avant} and \ref{fig:genericity} for an example
and \cite[Proof of Proposition 2.10]{Giroux_2001} for more examples of the same
kind.

\begin{proposition}
Any closed surface in a contact 3--manifold $(M, \xi)$ is $C^\infty$--close to
a $\xi$--convex surface.
\end{proposition}
\begin{figure}[htp]
	\begin{center}
		\includegraphics[height=8cm]{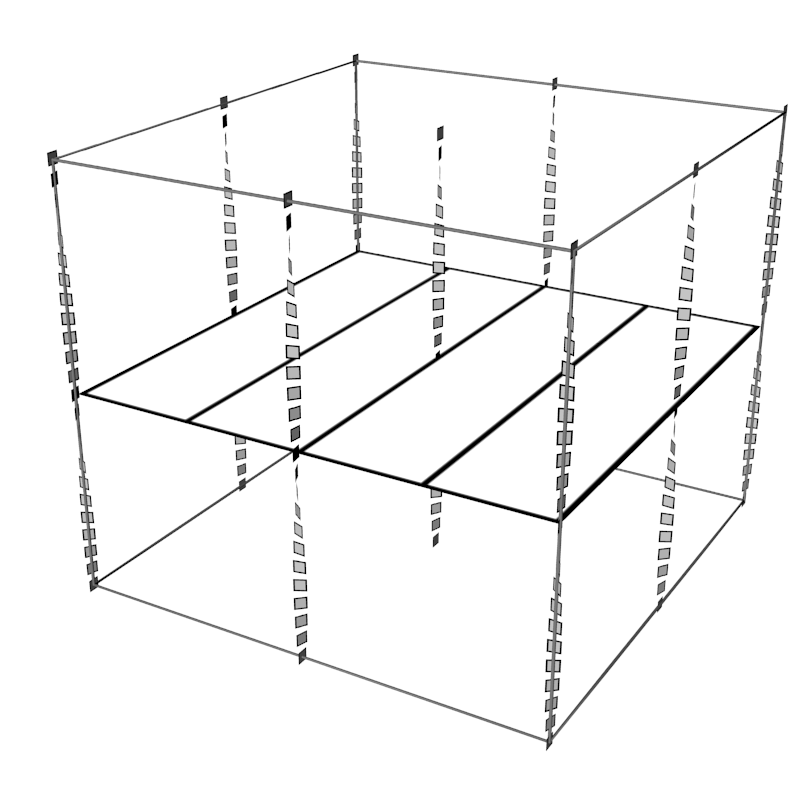}
	\end{center}
	\caption{A non-convex torus}
	\label{fig:genericity_avant}
\end{figure}
\begin{figure}[htp]
	\begin{center}
		\includegraphics[height=8cm]{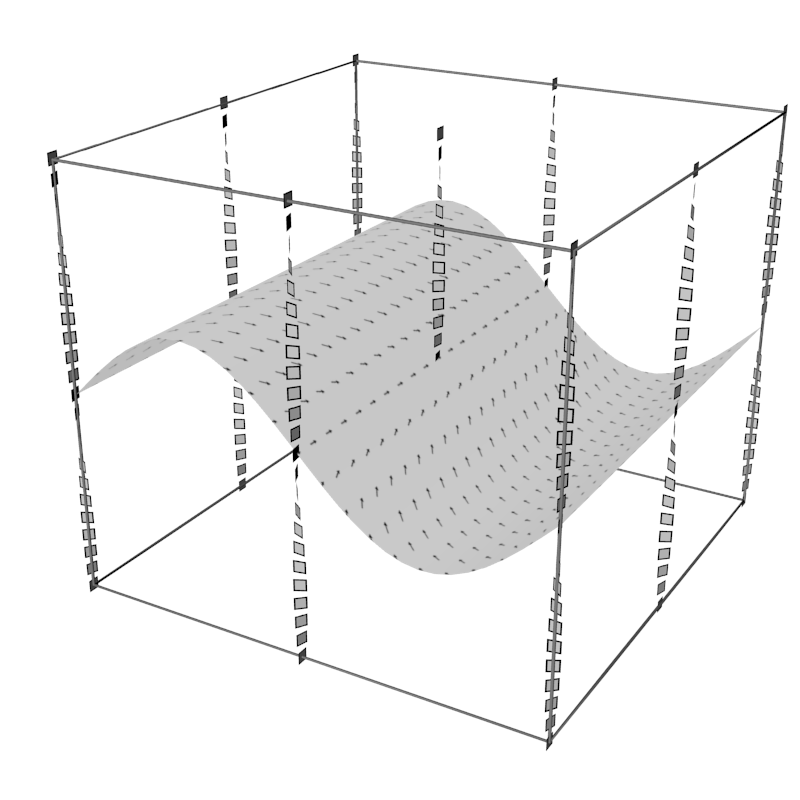}
	\end{center}
	\caption{Perturbation of the non-convex torus of Figure
	\ref{fig:genericity_avant} into a convex torus}
	\label{fig:genericity}
\end{figure}

Genericity of $\xi$--convex surfaces is a small dimensional phenomenon, it does
not hold for hypersurfaces in higher dimensions \cite{MoriGenericity}. In
dimension 3, $\xi$--convexity is a degenerate notion, much like ordinary
convexity in real dimension 1 and pseudo-convexity in complex dimension 1.

We first prove that any foliation sufficiently close to a characteristic
foliation $\xi_0 S$ is the characteristic foliation $\xi S$ coming from some
$\xi$ isotopic to $\xi_0$. Equivalently it means it is the characteristic
foliation printed by $\xi_0$ on some surface isotopic to $S$. Let $C$ be the
connected component of the space of contact structures which contains $\xi_0$.
The first point is that the map which maps $\xi$ in $C$ to the
characteristic foliation $\xi S$ is open. The second point is that Gray's
theorem imply that all $\xi$ in $C$ are isotopic to $\xi_0$.

So the genericity of $\xi$--convex surfaces will follow from the one of divided
foliations.
Essentially we will see that the obstructions to the existence of a dividing set
discussed above are the only ones provided that no non-trivial recurrence appear.
The precise requirement is expressed in the following definition.

\begin{definition}
\label{definition:pb}
A singular foliation on a closed surface satisfies the
\emph{Poincar\'e--Bendixson property} if the limit set of any half orbit is
either a singularity or a closed orbit or a union of singularities and 
orbits connecting them.
\end{definition}

The Poincar\'e-Bendixson theorem thus says that a singular foliation on a sphere
satisfies the Poincar\'e-Bendixson property as soon as its singularities are
isolated, see e.g. \cite{PalisDeMelo}.

\begin{proposition}
\label{prop:critère_convexité}
Let $S$ be a surface in a contact manifold $(V, \xi)$. If the characteristic
foliation $\xi S$ satisfies the Poincaré--Bendixson property then $S$ is
$\xi$--convex if and only if $\xi S$ has neither degenerate closed leaves nor
retrograde connections.
\end{proposition}

Genericity of $\xi$--convex surfaces then follows from Peixoto's theorem
stating that Morse-Smale foliations are generic on surfaces, see
\cite{PalisDeMelo} for a beautiful exposition of this result starting with the
basic of dynamical systems. A foliation is Morse-Smale if 
\begin{itemize}
\item 
it satisfies the Poincar\'e-Bendixson property,
\item
 all its singularities are nodes or saddles,
\item
all its closed leaves are non-degenerate,
\item
it has no saddle connections.
\end{itemize}

\begin{proof}[Proof of Proposition \ref{prop:critère_convexité}]
In the preceding sections, we have seen that the absence of degenerate closed
leaves and retrograde connections is necessary for convexity.

We now prove that it is sufficient when the Poincar\'e-Bendixson property holds.
In this proof we assume that all singularities are nodes, saddles or
saddle-nodes. This is true for generic families of characteristic foliations
with any number of parameters and is all we need in these lectures. In order to
save some more words we will even pretend there are no saddle-nodes. The reader
can replace any occurrence of the word ``saddle'' by ``saddle or saddle-node''
to get the more general proof.

During the discussion of obstructions to convexity, we have seen that
singularities and closed leaves should be dispatched into $S_+$ or $S_-$
according to their signs.
Another constraint comes from separatrices of saddles: since we want the
characteristic foliation to go transversely out of $S_+$ along $\Gamma$, stable
separatrices of positive saddles and unstable separatrices of negative
saddles cannot meet $\Gamma$.

So we build a subsurface $S_+'$ of $S$ by putting a small disk around each
positive singularity and narrow bands around positive closed leaves and stable
separatrices of positive saddles. If all these elements are sufficiently small,
the boundary of $S_+'$ can be smoothed to a curve transverse to the
characteristic foliation, see Figure \ref{fig:giroux_graph}.
\begin{figure}[htp]	
	\begin{center}
		\includegraphics{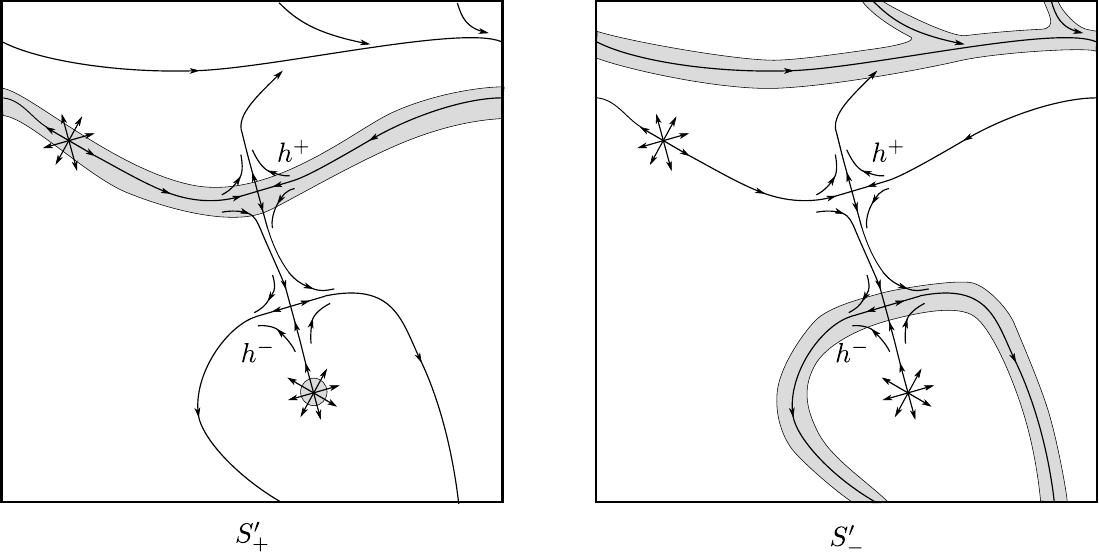}
	\end{center}
	\caption{Construction of a dividing set on a torus. One can check that
	$\partial S_+'$ and $\partial S_-'$ are indeed isotopic among dividing
	curves.}
	\label{fig:giroux_graph}
\end{figure}
In addition one can find an area form on $S_+'$ which
is expanded by $\xi S_+'$.  We can construct similarly a subsurface $S_-'$ and
a contracted area form on it. None of these subsurfaces is empty because of
Stokes' theorem which guaranties that an area form on a closed surface is never
exact.

Let $A$ be a component of the complement of $S_+' \cup S_-'$ in $S$. 
It has non-empty boundary and does not contain any singularity so $A$ is
an annulus. In addition it does not contain any closed leaf so
Poincar\'e-Bendixson's theorem guaranties that all leaves of the characteristic
foliation entering $A$ along some boundary component leave it through the other
boundary component. So we are indeed in the situation of Figure
\ref{fig:neighb_Gamma} and one can take the core of $A$ as a dividing curve.
The corresponding subsurfaces $S_\pm$ then retract onto $S_\pm'$.
\end{proof}

The proof above contains some useful information about how a dividing set can be
recovered from the important features of the characteristic foliation so we
record this in a definition and a corollary.

\begin{definition}
	\label{def:giroux_graph}
Given a foliation $\F$ satisfying the Poincar\'e-Bendixon property, we denote by
$G_+$ (resp $G_-$) the union of repelling (resp attracting) closed leaves, of
positive (resp negative) singularities and of the stable (resp unstable)
separatrices of these singularities. The union $G_+ \cup G_-$ is called the
Giroux graph of $\F$.
\end{definition}

\begin{indented}
Note that the terminology graph is a little stretched since one can have
separatrices accumulating on closed orbits (like in Figure
\ref{fig:giroux_graph}) or on connected singularities so the Giroux graph
equipped with the induced topology is not necessarily homeomorphic to a
CW-complex of dimension one.
\end{indented}

\begin{corollary}
	\label{cor:giroux_graph}
If a characteristic foliation satisfies the convexity criterion of Proposition
\ref{prop:critère_convexité} and $G_+ \cup G_-$ is its Giroux graph then, for
any dividing set, $S_+$ retracts on a regular neighborhood of $G_+$ and $S_-$ 
on a regular neighborhood of $G_-$.
\end{corollary}

\subsection{Giroux criterion and Eliashberg--Bennequin inequalities}

Until now, the discussion of this chapter does not make any distinction between
tight and overtwisted contact structures. We now start to discuss how convex
surfaces theory sees tightness.

\begin{theorem}[{Giroux criterion \cite[Theorem 4.5a]{Giroux_2001}}]
\label{thm:critère_Giroux}
In a contact manifold $(V, \xi)$, a $\xi$--convex surface divided by some
multi-curve $\Gamma$ has a tight neighborhood if and only if one of the
following conditions is satisfied:
\begin{itemize}
	\item no component of $\Gamma$ bounds a disk in $S$
	\item $S$ is a sphere and $\Gamma$ is connected.
\end{itemize}
\end{theorem}

The only application of this theorem we will present in detail is in the
classification of tight contact structures on $\S^3$ (existence by Bennequin
and uniqueness by Eliashberg). There we will only need
that, if $S$ is a sphere, then it has a tight neighborhood only if its dividing
set is connected. So we prove only this part of the theorem, we assume $S$ is
a sphere and $\Gamma$ is not connected.
Let $S'$ be a component of $S \setminus \Gamma$ which is a disk and denote by
$\gamma$ its boundary.  Let $S''$ be the other component containing $\gamma$ in
its boundary.  Since $\Gamma$ is not connected, $S''$ has more boundary
components. Using this, one can construct a foliation $\F$ on $S$ which is
divided by $\Gamma$, has a circle of singularities $L$ in $S''$, is radial
inside a disk bounded by $L$ and coincides with $\xi S$
outside $S' \cup S''$, see Figure \ref{fig:giroux_criterion}. 
\begin{figure}[htp]
	\begin{center}
		\includegraphics{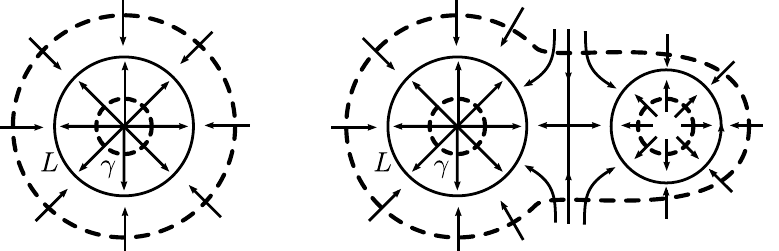}
	\end{center}
	\caption{Characteristic foliations for the Giroux criterion. The dividing set
	$\Gamma$ is dashed. On the left-hand side one has the simplest case when $S''$
	is an annulus. On the right hand-side one sees a possible foliation when $S''$
	has one more boundary component (on the right). Note that the disk bounded by
	the small component of $\Gamma$ on the right may contain more components of
	$\Gamma$. The extension to more boundary components uses the same idea.}
	\label{fig:giroux_criterion}
\end{figure}
In any neighborhood $U$ of $S$,
the realization Lemma gives a surface $\delta_1(S)$ which has $\delta_1(\F)$ as
its characteristic foliation. Then $\delta_1(L)$ is the boundary of an
overtwisted disk contained in $\delta_1(S)$ hence in $U$.

An important direct application of the Giroux criterion is Giroux's proof of the
following constraint on the Euler class of a tight contact structure
(originally due to Eliashberg). We will not use it in those notes but include
it here since it now comes for free.

\begin{theorem}[Eliashberg--Bennequin inequality \cite{Eliashberg_20_ans}]
\label{thm:eliashberg_bennequin}
Let $(M, \xi)$ be a 3--dimensional contact manifold. If $\xi$ is tight and $S$ 
is a closed surface embedded in $M$ then the Euler class of $\xi$ satisfies the
following inequality:
\[
|\langle e(\xi), S \rangle| \leq \max(0, -\chi(S))
\]
\end{theorem}

\begin{proof}
Using genericity of $\xi$--convex surfaces, one can homotop $S$ until it is
$\xi$--convex. This does not change the Euler class which can now be evaluated
as $\chi(S_+) - \chi(S_-)$ since singularities are distributed among $S_+$ and
$S_-$ according to their signs. If $S$ is a sphere then the Giroux criterion
says that both $S_+$ and $S_-$ are disks so $\langle e(\xi), S \rangle = 0$ and
the inequality is proved. So suppose now that $S$ has positive genus.  The
Giroux criterion says that no connected component of $S_+$ or $S_-$ is a disk.
This implies that both $\chi(S_+)$ and $\chi(S_-)$ are negative.  Hence both
$\chi(S_+) - \chi(S_-)$ and $-\chi(S_+) + \chi(S_-)$ are less than $-\chi(S_+) -
\chi(S_-)$ which is $-\chi(S)$. 
\end{proof}

\chapter{Bifurcations and first classification results}

The goal of this chapter is to prove that any tight contact structure on $\S^3$
has to be isotopic to the standard contact structure and that the later is
indeed tight. We will not give the original proofs due to Eliashberg
\cite{Eliashberg_20_ans} and Bennequin \cite{Bennequin} respectively. We will
rather use the technology of $\xi$--convex surfaces to prove them. These proofs
were obtained by Giroux along its way towards more general classification
results in \cite{Giroux_2000}. The classification result is a comparatively easy
special case of Giroux's preparation Lemma \cite[Lemma 2.17]{Giroux_2000} while
the tightness result follows from the bifurcation lemmas
\cite[Lemmas 2.12 and 2.14]{Giroux_2000}.

\section{The elimination lemma}

In the characteristic foliation of a surface, a saddle and a node 
are said to be in elimination position if they have the same sign and there is a
leaf from one to the other. Such a leaf is called an elimination arc.
Giroux's elimination lemma in its simplest form says one can perturb the surface
to replace a neighborhood of the elimination arc by a region without singularity
as in Figure \ref{fig:elimination}.
\begin{figure}[htp]
  \begin{center}
		\includegraphics{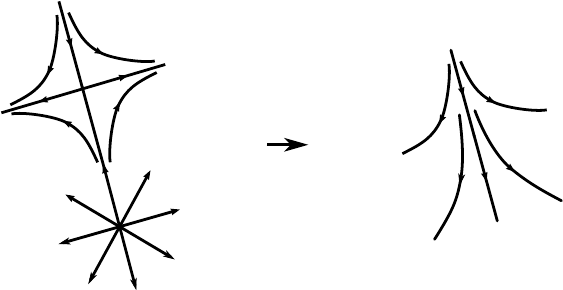}
  \end{center}
  \caption{Elimination of a pair of singular points.}
  \label{fig:elimination}
\end{figure}
\begin{figure}[htp]
	\begin{center}
		\includegraphics{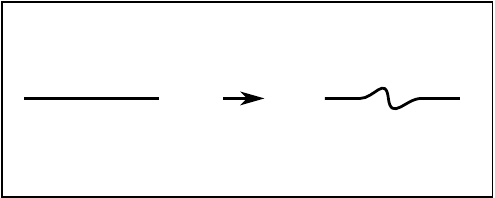}

		\includegraphics[width=6cm]{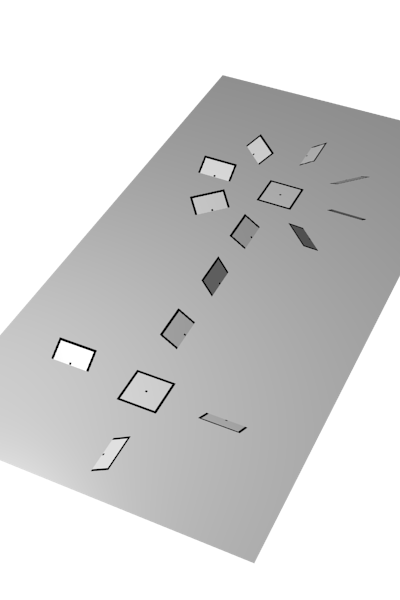}
		\includegraphics[width=6cm]{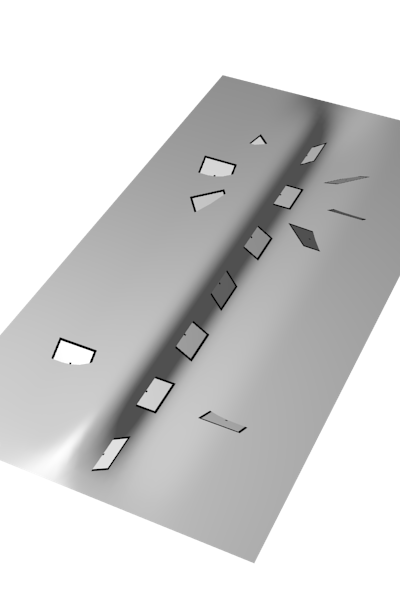}
	\end{center}
	\caption{The elimination move. The top box shows the move transverse to the
	elimination arc seen as the middle point of the segment. This move is cut off in the longitudinal direction.}
	\label{fig:elimination_arc}
\end{figure}

For the classification of tight contact structures on $\S^3$ we will need a
version of this process which keeps neighboring surfaces under control. 
\begin{indented}
We do not need much control though and the following version is simpler than
\cite[Lemma 2.15]{Giroux_2000} which is needed for the classification of tight
contact structures on torus bundles.
\end{indented}
Let $\xi$ be a contact structure on $S \times [-1, 1]$ and set
$S_t := S \times \{t\}$.
Suppose a node $e_0$ and a saddle $h_0$ are in elimination position 
on $S_0$. This configuration is stable so it persists for $t$ in some interval
$(-\varepsilon, \varepsilon)$. Let $C_t$ denote a continuous family of
elimination arcs between $e_t$ and $h_t$ on $S_t$.

\begin{lemma}[Giroux elimination lemma]
\label{lem:elimination}
Let $\delta$ be a positive number smaller than $\varepsilon$. Let $U$ a
neighborhood of $\bigcup_{|t| < \delta} C_t$ intersecting each $S_t$ in a disk
$D_t$ whose characteristic foliation is as in the left hand side of Figure
\ref{fig:elimination}.  One can deform $\xi$ in $U$ such that $\xi D_t$ has:
\begin{itemize}
\item 
no singular point when $|t| < \delta$,
\item
a saddle-node when $|t| = \delta$,
\item
a pair of singularities in elimination position when 
$|t| \in (\delta, \varepsilon)$.
\end{itemize}
In addition, one can impose that separatrices facing the elimination arc are
connected to the same points of $\partial D_t$ as before the deformation, see
Figure \ref{fig:elimination_family}.
\end{lemma}
\begin{figure}[htp]
  \begin{center}
		\includegraphics[width=10cm]{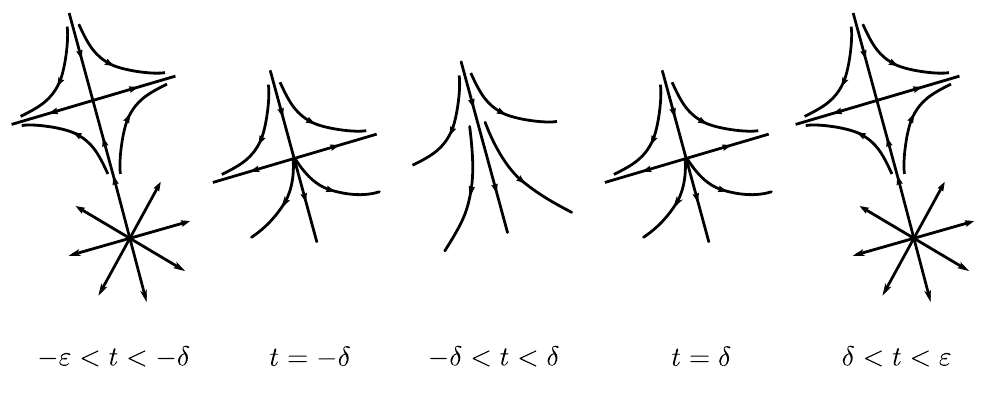}
  \end{center}
  \caption{Elimination in a family.}
  \label{fig:elimination_family}
\end{figure}
The corresponding manipulation transverse to the elimination arc is explained in
Figure \ref{fig:transverse_deformation_family}
\begin{figure}[htp]
	\begin{center}
		\includegraphics{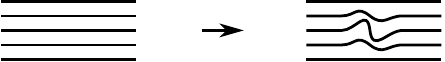}
	\end{center}
	\caption{The elimination move in family. The left hand-side shows the original
	surfaces $S_t$ stacked. The right hand-side performs the elimination, compare
	with top of Figure \ref{fig:elimination_arc}.
	}
	\label{fig:transverse_deformation_family}
\end{figure}

\section{Thickened spheres and Eliashberg uniqueness}

The goal of this section is to explain Giroux's proof of the classification of
tight contact structures on $\S^3$.

\begin{theorem}[Eliashberg \cite{Eliashberg_20_ans}]
\label{thm:eliashberg_s3}
Any tight contact structure on $\S^3$ is isotopic to the standard one.
\end{theorem}

By definition of contact structures, one can assume that $\S^3$ is the union
of two standard balls and a thickened sphere with standard $\xi$--convex
boundary as in Figure \ref{fig:xiS_spheres}.
This allows in particular to apply the following proposition.

\begin{proposition}
	\label{prop:thickened_sphere}
Let $\xi$ be a tight contact structure on a thickened sphere 
$S \times [0, 1]$. If $S_0$ and $S_1$ are $\xi$--convex then $\xi$ is
isotopic relative to the boundary to a contact structure $\xi'$ such that all
spheres $S_t$ are $\xi'$--convex.
\end{proposition}

\begin{proof}
First note that tightness prevents the apparition of any closed leaf in any 
$\xi S_t$ since it would bound an overtwisted disk. Then we need some theory of
one-parameter families of singular foliations on the sphere \cite{Sotomayor}.
Specifically, one can assume that each $\xi S_t$ has finitely many singularities
and at worse a saddle connection or a saddle-node (but not both at the same
time). Note that finiteness of saddle connections can be achieved by
perturbation thanks to the absence of closed leaves (compare Figure
\ref{fig:csr_dl}).
Using this, the Poincar\'e-Bendixson theorem and the criterion of Proposition
\ref{prop:critère_convexité}, one can see that all surfaces $S_t$ are
$\xi$--convex except for finitely many $t_1, \dots, t_k$ where: 
\begin{itemize}
\item 
all singularities of $\xi S_{t_i}$ are saddles or nodes
\item 
there is exactly one saddle connection on $\xi S_{t_i}$ and it is retrograde,
\end{itemize}
see Figure \ref{fig:s3before} for an example.
\begin{figure}[htp]
	\begin{center}
		\includegraphics[width=5cm]{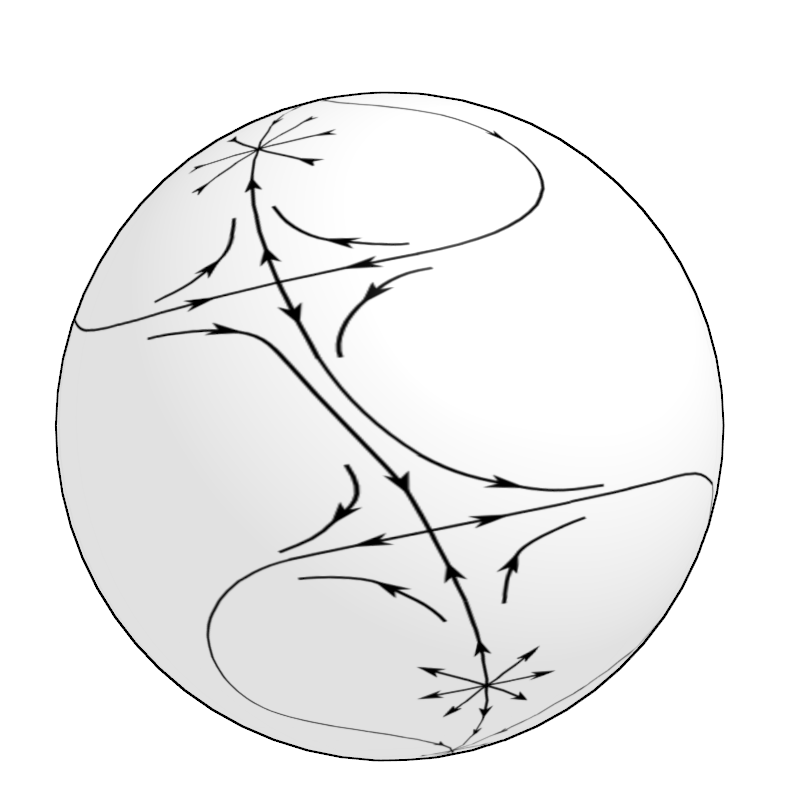}
		\includegraphics[width=5cm]{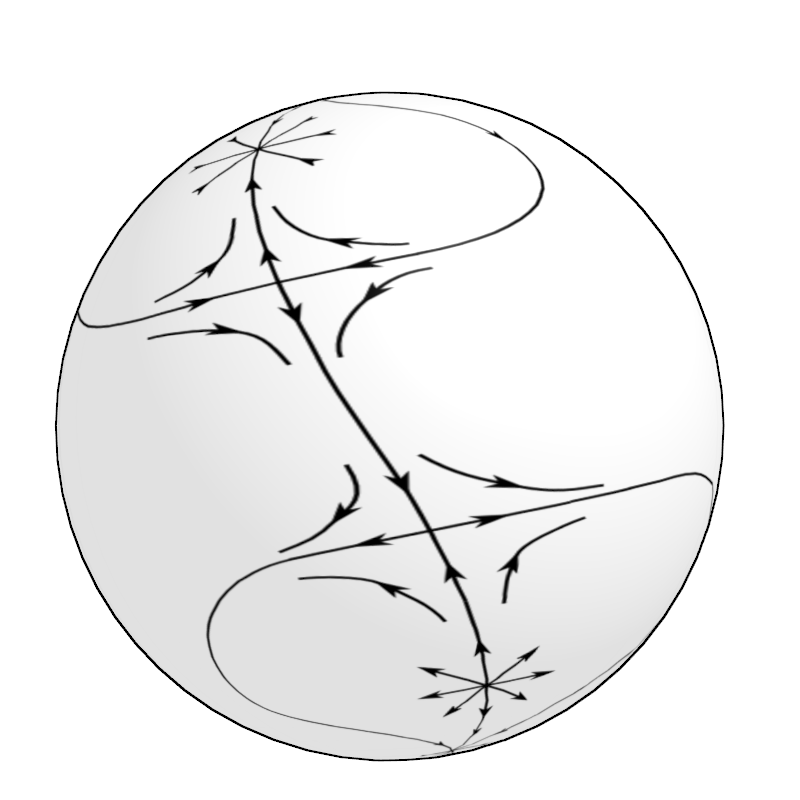}
		\includegraphics[width=5cm]{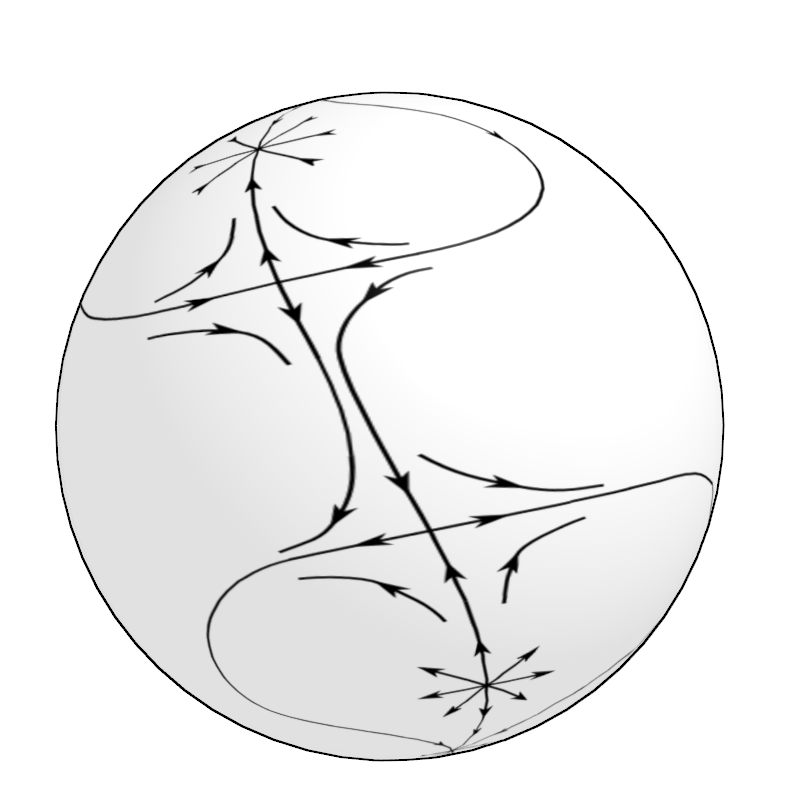}
	\end{center}
	\caption{Original movie}
	\label{fig:s3before}
\end{figure}

We will now modify $\xi$ near each $S_{t_i}$ in order to make all $S_t$
$\xi$--convex. Since we know closed leaf or non-trivial recurrence cannot arise,
it suffices to get rid of retrograde saddle connections. We concentrate on one 
$t_i$ at a time. Let $\varepsilon$ be a small positive number such that
$\xi S_t$ does not change up to homeomorphism when $t$ is either in
$[t_i-\varepsilon, t_i)$ or $(t_i, t_i + \varepsilon]$. In particular the
positive part $G^+$ of the Giroux graph deforms by isotopy in each of these
intervals. Theorem \ref{thm:critère_Giroux}, the Giroux criterion, and the link
between the Giroux graph and the dividing set explained in Corollary
\ref{cor:giroux_graph} guarantee that $G^+$ is a tree in each interval.
It implies that we can find elimination arcs between all positive saddles and
all but one positive nodes without using the separatrix which enters the saddle
connection at $t_i$ (recall in particular that the number of vertices in a tree 
is exactly the number of edges plus one). 

We now use Lemma \ref{lem:elimination}, the elimination lemma, to get rid of all
positive saddles for $t$ in $[t - \delta, t + \delta]$ for some positive
$\delta$ smaller than $\varepsilon$, see Figure \ref{fig:s3after}.
\begin{figure}[htp]
	\begin{center}
		\begin{tabular}{ccc}
		\includegraphics[width=3.8cm]{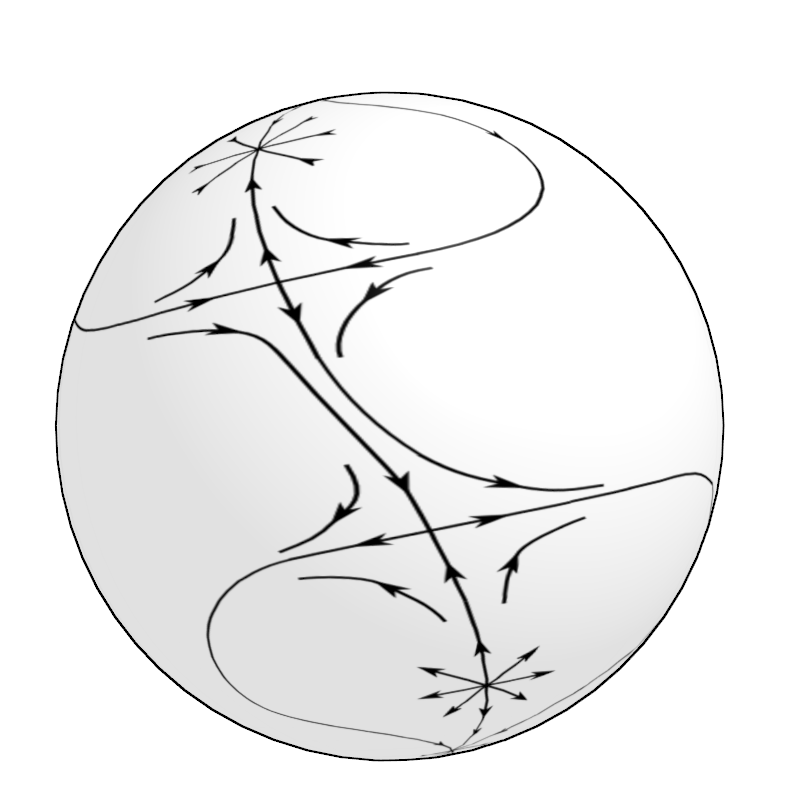}&
		\includegraphics[width=3.8cm]{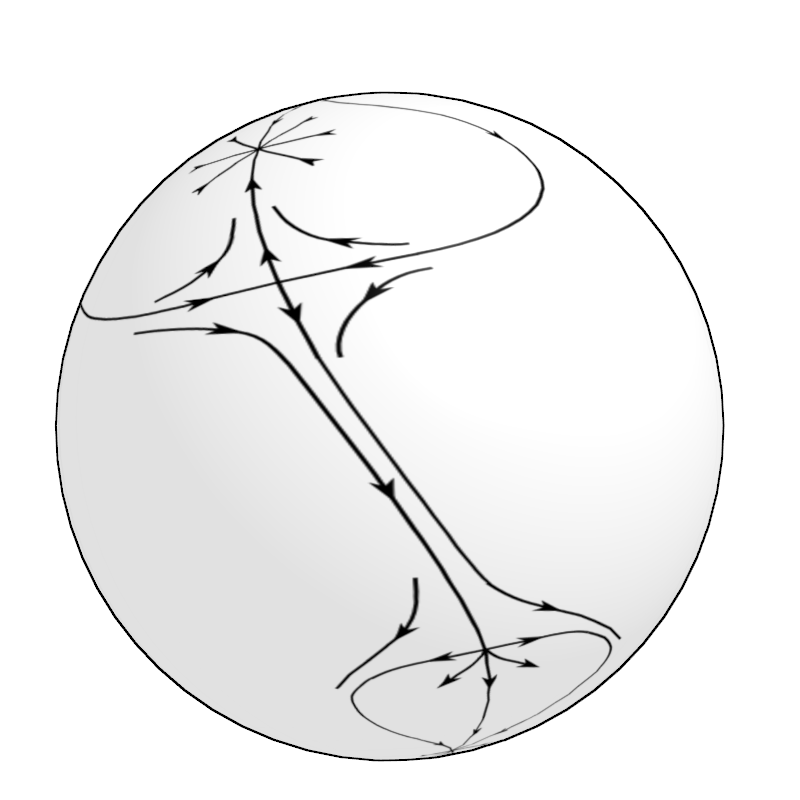}&
		\includegraphics[width=3.8cm]{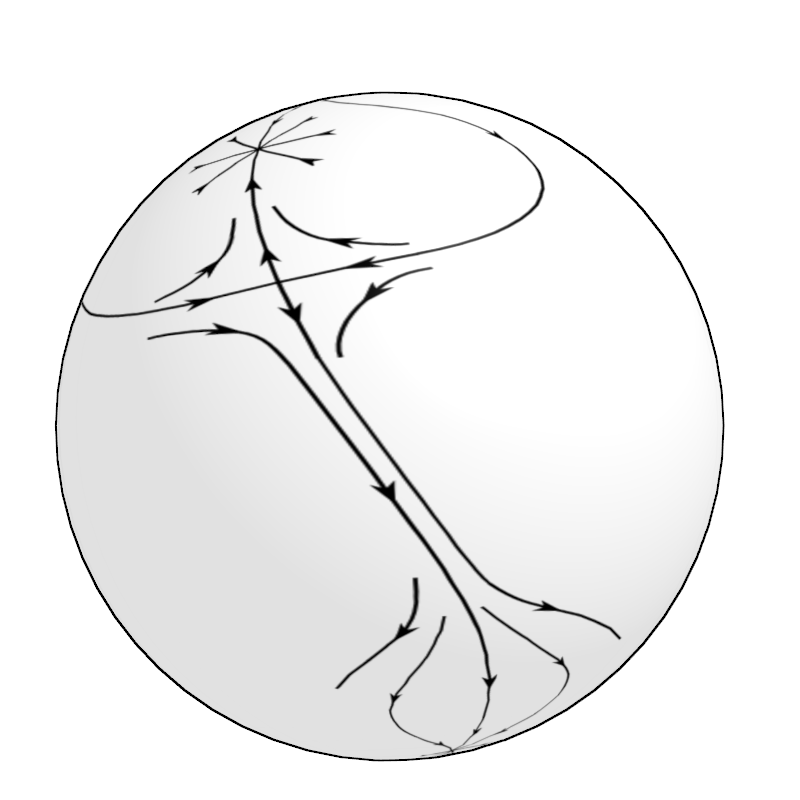}\\
		\includegraphics[width=3.8cm]{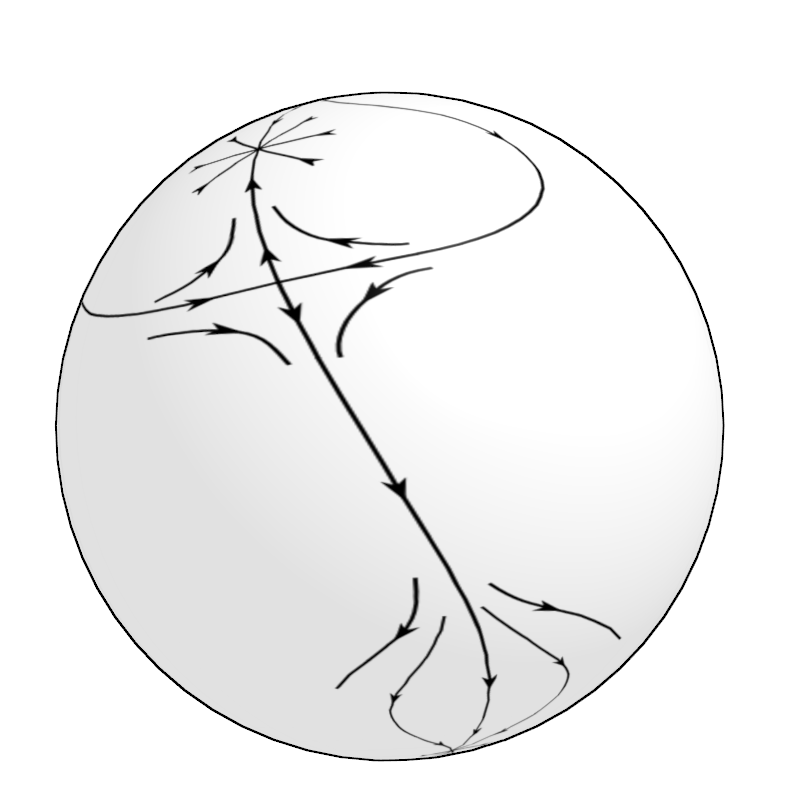}&
		\includegraphics[width=3.8cm]{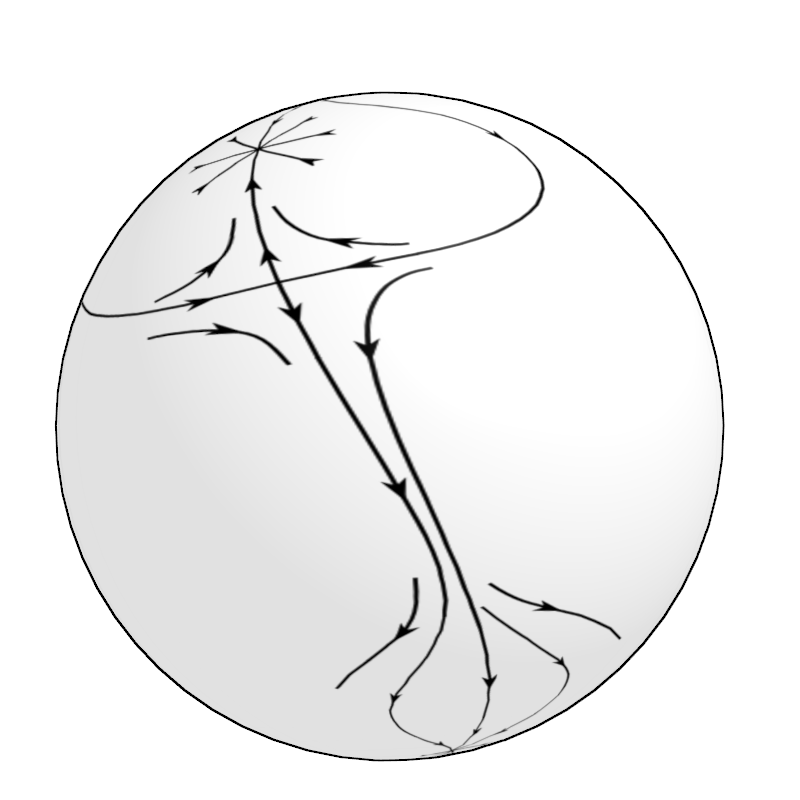}&
		\includegraphics[width=3.8cm]{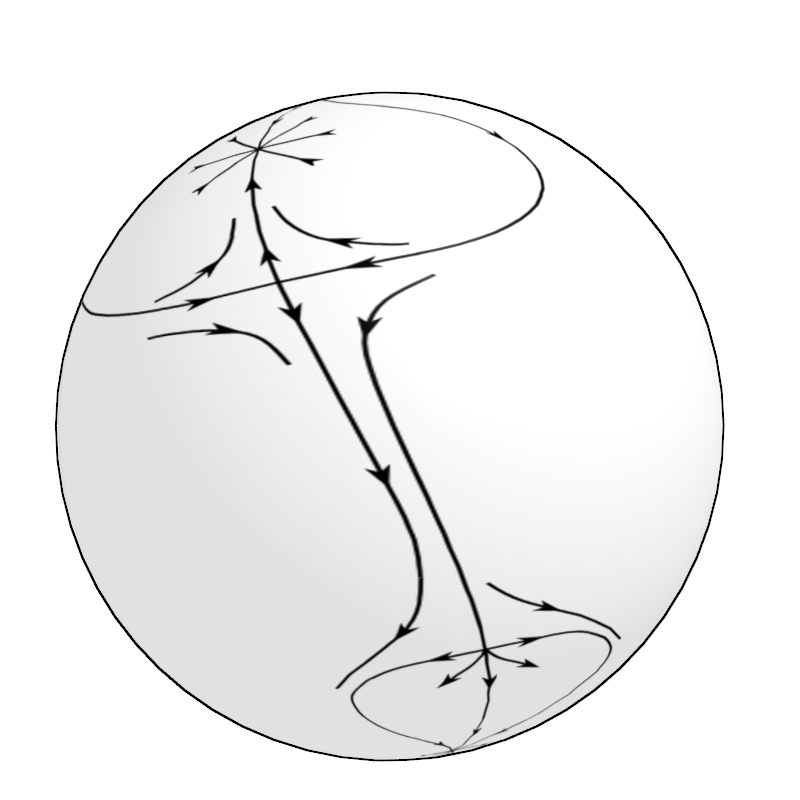}\\
		\end{tabular}
		\includegraphics[width=3.8cm]{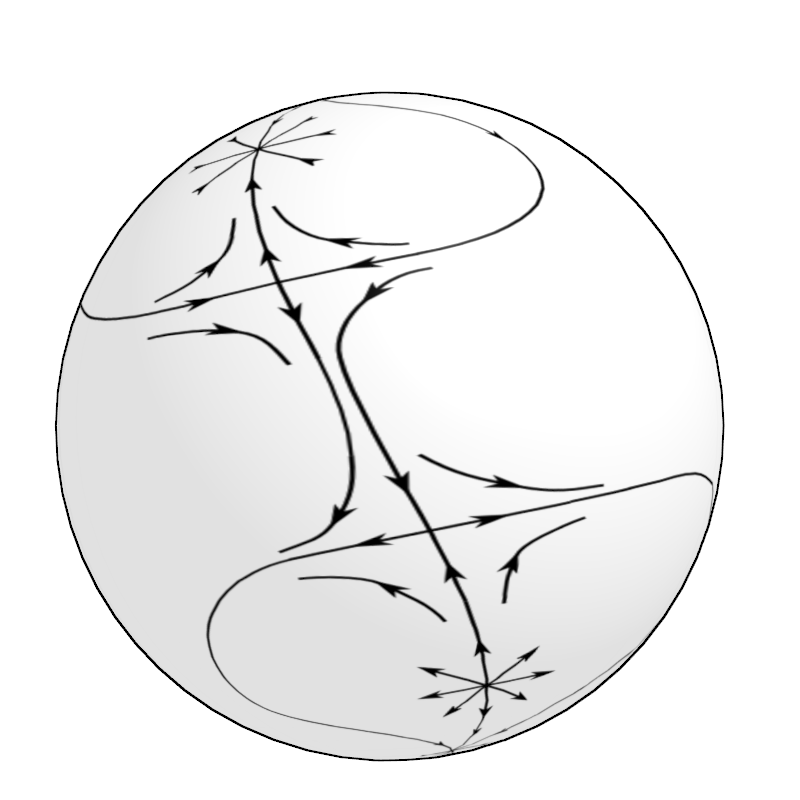}
	\end{center}
	\caption{Movie after elimination. The first picture is the same as
	in Figure \ref{fig:s3before} then a pair of singularity is replaced by a
	saddle-node then it disappears. The fourth picture corresponds to the central
	picture of Figure \ref{fig:s3before} but there is no more positive saddle so
	no saddle connection. The eliminated pair returns in the sixth picture as a
	saddle-node and the final picture is the same as in Figure \ref{fig:s3before}.}
	\label{fig:s3after}
\end{figure}
\end{proof}

Before continuing the proof of the theorem, we note two properties of the sphere
which were somehow surreptitiously used in the above proof. After the elimination
of the retrograde connections we needed the fact that no closed leaves could
appear, this is due to Schönflies theorem which would have provided an
overtwisted disk. We also needed the Poincar\'e-Bendixson theorem to prevent the
apparition of non-trivial recurrence. Suppose one tries to use the elimination
lemma to get rid of the bifurcation of Figure \ref{fig:croisement} (which is
bound to fail since the isotopy class of the dividing set changes during this
bifurcation). If one gets rid of both saddles then degenerate leaves arise. If
one gets rid of one saddle only (like we did for the sphere) then non-trivial
recurrence appear: we get a Cherry flow on the torus, see
\cite{PalisDeMelo}.

The proof of Theorem \ref{thm:eliashberg_s3} now follows from Giroux's
uniqueness lemma which allows to replace the contact structure obtained on the
thickened sphere of the previous proposition by the model.

\begin{lemma}[Uniqueness lemma {\cite[Lemma 2.7]{Giroux_2000}}]
Let $\xi_0$ and $\xi_1$ be two contact structures printing the same
characteristic foliations on the boundary of $S \times [0,1]$. If there is a
continuous family of multi-curves $\Gamma_t$ dividing both $\xi_0 S_t$ and
$\xi_1 S_t$ then $\xi_0$ and $\xi_1$ are isotopic relative to the boundary.
\end{lemma}

The proof of this lemma is similar to the ones of the previous chapter but the
path of contact structures is less obvious.

We now explain how to get the classification of tight contact structures on
$\S^2 \times \S^1$ without extra effort. Let $\xi$ be one of them and fix some
$S = \S^2 \times \{\theta_0\}$. Using genericity of $\xi$--convex surfaces, we
can perturb $\xi$ to make $S$ convex. Then the Giroux criterion tells us that
its dividing set is connected. Using the realisation lemma, we change $\xi$ by
isotopy until $\xi S$ is standard, ie as in Figure \ref{fig:xiS_spheres}. We can
then remove a homogeneous neighborhood of $S$ and we are back to a thickened
sphere where we can apply Proposition \ref{prop:thickened_sphere} and the
uniqueness lemma.

\section{Bifurcation lemmas}

We now consider a general closed surface $S$ and any contact structure $\xi$ on
$S \times I$ for some interval $I$. For each $t$ in $I$, one has the surface
$S_t := S \times \{t\}$ and its characteristic foliation $\xi S_t$. If some
$S_{t_0}$ is not $\xi$-convex then the characteristic foliations for $t$ close to
$t_0$ are not all $C^1$--conjugate to $\xi S_{t_0}$, otherwise the global
reconstruction lemma (Lemma \ref{lemma:global_reconstruction}) would give a
contradiction. We will now try to understand what really happens when this lack
of $\xi$--convexity is explained by the obstructions we discussed in the
previous chapter, ie it comes from a degenerate closed leaf or a retrograde
connection. We will see in particular that the bifurcation is much sharper than
expected: no foliation $\xi S_t$ is even $C^0$--conjuguate to $\xi S_{t_0}$ for
$t$ in a punctured neighborhood of $t_0$. Better, we will get a very precise
description of what happens.

\subsubsection*{The birth/death lemma}

Let $L$ be a degenerate closed leaf of the characteristic foliation
$\xi S_t$. This means that the Poincar\'e return map on any curve transverse
to $L$ is tangent to the identity. One says that $L$ is positive (resp
negative) if the second derivative of this map is positive (resp negative) at
the intersection point between $L$ and the transverse curve. If $L$ is either
positive or negative then one says that it is weakly degenerate.

\begin{lemma}[Birth/Death Lemma {\cite[Lemma 2.12]{Giroux_2000}}]
\label{lemma:birth_death}
A positive (resp negative) degenerate closed orbit indicates the birth (resp
death) of a pair of non-degenerate closed leaves when $t$ increases.
\end{lemma}
See Figure \ref{fig:rotative} for examples of these situations on a thickened
torus $T \times [0, 1]$.
\begin{figure}[htp]
	\begin{center}
		\includegraphics[width=10cm]{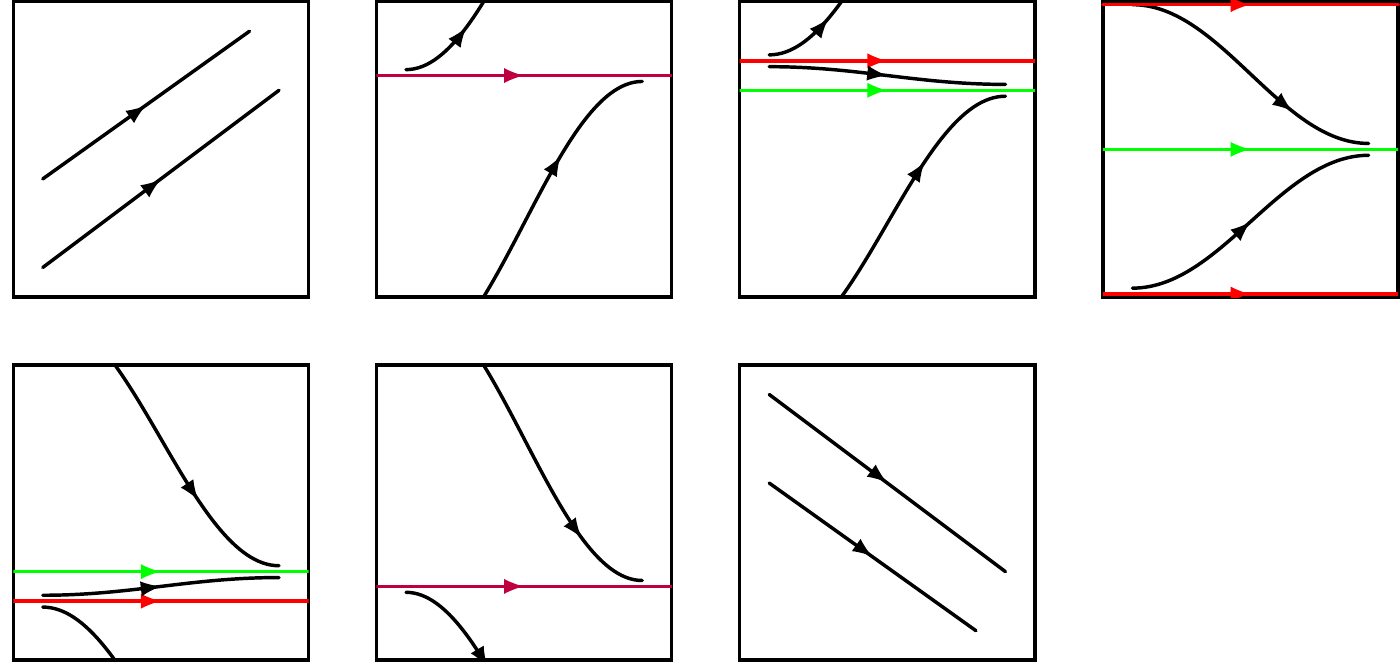}
	\end{center}
	\caption{Birth and death of closed leaves on a torus.}
	\label{fig:rotative}
\end{figure}
Looking at these pictures it is easy to prove a weak form of the birth-death
lemma which already shows how the contact condition enters. Since the contact
structure is transverse to all tori $T_t$, $t\in [0,1]$, one can lift
$\partial_t$ to a vector field tangent to $\xi$. The flow of this lift defines
a new product structure on $T \times [0,1]$ without changing the movie of
singular foliations $\xi T_t$ up to diffeomorphism.
So one can assume that all intervals $I_p = \{p\} \times [0, 1]$ are Legendrian.
If we think of foliations $\xi T_t$ as living all on $T$ then the contact
condition is equivalent to asking that, at each point $p$, $\xi T_t(p)$ rotates
clockwise as $t$ increases. Indeed, if $x$ and $y$ are coordinates on $T$,
there is a function $\theta$ such that
\[
\xi = \ker\big(\cos\theta(x,y,t)\, dx - \sin\theta(x,y,t)\, dy\big).
\]
The contact condition is then equivalent to $\partial_t \theta > 0$, compare
with the proof of the Darboux-Pfaff theorem (Theorem~\ref{thm:Darboux}).

Now the second picture in Figure \ref{fig:rotative}
shows a positive degenerate orbit $L$ in some $\xi T_{t_0}$. Let $A$ be a small
annulus around $L$.  Along $L$, the slope of $\xi T_{t_0}$ is zero and it is
positive in $A \setminus L$. So, for $t < t_0$ it was everywhere positive in $A$
and there were no closed leaf at all in $A$. For $t > t_0$,
the slope becomes negative along $L$ and stays positive along the
boundary of $A$. Then the complement of $L$ in $A$ is
made of two (half-open) annuli whose boundary are transverse to $\xi T$, see
Figure~\ref{fig:birth_death}. The Poincar\'e-Bendixson theorem guaranties that
each of these two sub-annuli contain at least one closed leaf for $t > t_0$
sufficiently close to $t_0$.
\begin{figure}[ht]
	\begin{center}
		\includegraphics[width=10cm]{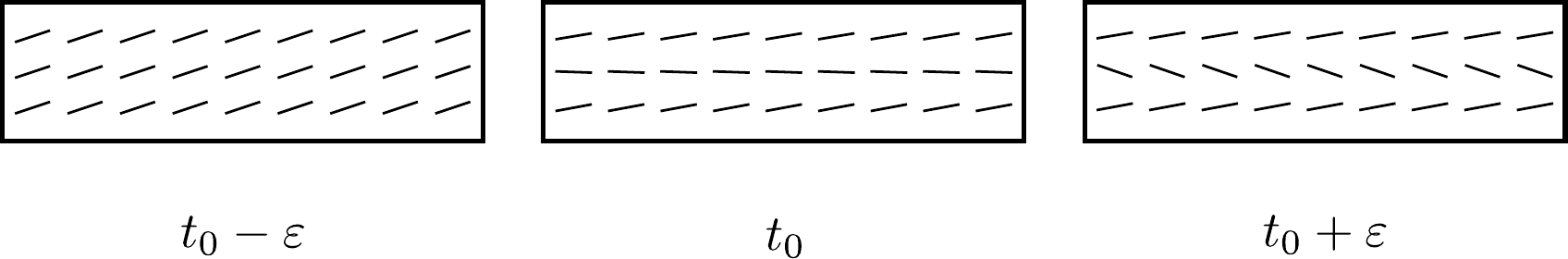}
	\end{center}
	\caption{Birth of at least a pair of periodic orbits. The annulus $A$ is
	obtained by gluing left and right. The circle $L$ is at mid-height of each
	annulus.}
	\label{fig:birth_death}
\end{figure}

So we proved the following weak version of the birth/death lemma which will be
sufficient for our purposes: if there is a positive degenerate closed orbit $L$
at time $t_0$ then there is an annulus $A$ around $L$ and some positive
$\varepsilon$ such that there is no closed leaves in $A$ for $t$ in
$(t_0-\varepsilon, t_0)$ and at least two for $t$ in 
$(t_0, t_0 + \varepsilon)$.  The death case on the bottom row of Figure
\ref{fig:rotative} is explained similarly. Note that nothing required $T$ to be
a torus in this explanation, one only has to work near $L$.

\subsubsection*{The crossing lemma}

\begin{lemma}[Crossing Lemma {\cite[Lemma 2.14]{Giroux_2000}}]
\label{lemma:croisement}
Assume that there is a retrograde connection at time $t_0$.  For $t$ close to
$t_0$, there is a negative singularity $b^-_t$, a positive one $b^+_t$, an
unstable separatrix $c^-_t$ of $b^-_t$ and a stable one $c^+_t$ of $b^+_t$ such
that $c^-_{t_0} = c^+_{t_0}$.

For $t$ close to $t_0$, one can track separatrices using their intersection with
an oriented curve positively transverse to $\xi S_t$. Then, for $t < t_0$ (resp
$t > t_0$), the separatrix $c^-_t$ is below (resp above) $c^+_t$.
\end{lemma}

Figure \ref{fig:croisement} shows a retrograde saddle connection on a torus
obtained by gluing top/bottom and left/right. Singularities in the lower part
are negative while those in the upper part are positive.  The saddle connection
is marked by an arrow. The crossing Lemma tells us that the negative separatrix
has to turn to its right after the connection.
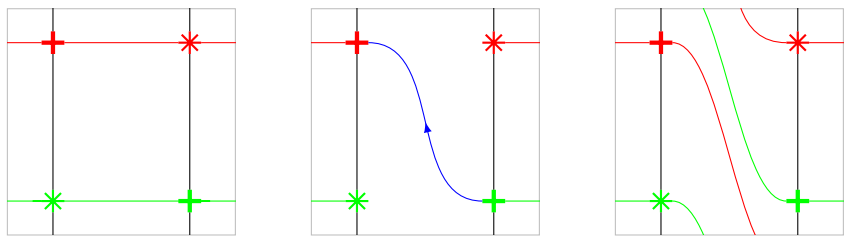
\begin{figure}[htp]
\centering

\beginpgfgraphicnamed{croisement}
\begin{tikzpicture}[x=\echelle, y=\echelle, color=gray!50]

\def\sing{
\sellep{h+ g}{.2, .85}; \sellep{h+ d}{1.2, .85}; \sellep{h+ h}{.2, 1.85};\sellep{h+ b}{.2, .85-1};

\sellen{h- g}{-.2, .15}; \sellen{h- d}{.8, .15}; \sellen{h- b}{.8, -.85};
\sellen{h- h}{.8, 1.85};
\noeudp{e+ g}{.8-1, .85}; \noeudp{e+}{.8, .85}; \noeudp{e+ b}{.8, -.15}
\noeudn{e-}{.2, .15}; \noeudn{e- d}{1.2, .15}; \noeudn{e- h}{.2, 1.15};
}

\def\grapheinvariable{
\sepp{e+ g}{droite}{h+ g}{gauche} ;
\sepp{e+}{droite}{h+ d}{gauche} ;

\sepn{h- g}{droite}{e-}{gauche} ;
\sepn{h- d}{droite}{e- d}{gauche} ;

\draw[black]  (e-.haut) -- (h+ g.bas);
\draw[black]  (e-.bas) -- (h+ b.haut);
\draw[black]  (h- d.haut) -- (e+.bas);
\draw[black]  (h- d.bas) -- (e+ b.haut);
\draw[black]  (h- h.bas) -- (e+.haut);
\draw[black]  (h+ g.haut) -- (e- h.bas);

}

\begin{image1}

 \sepp{h+ g}{droite}{e+}{gauche} ;
 \sepn{e-}{gauche}{h- d}{droite} ;

\end{image1}

\begin{image2}

 \sepc{h- d}{gauche}{h+ g}{droite};

\end{image2}

\begin{image3}
 
 \sepp[.5]{h+ g}{droite}{e+ b}{gauche} ;
 \sepp{h+ h}{droite}{e+}{gauche} ;

 \sepn[.5]{e- h}{droite}{h- d}{gauche} ;
 \sepn[.5]{e-}{droite}{h- b}{gauche} ;

\end{image3}

\end{tikzpicture}
\endpgfgraphicnamed
\caption{Retrograde saddle connection on a torus.}
\label{fig:croisement}
\end{figure}

The proof of the crossing lemma is rather delicate so we will only try to go as
far as explaining how the contact condition and the fact that the connection is
retrograde can enter the discussion. Each time we drop the $t$ subscript it
means $t = t_0$. Also we set $c = c^+ = c^-$.
Compared to the situation of the birth/death lemma, there is no hope to have a
neighborhood $S \times [0, 1]$ with $[0, 1]$ tangent to $\xi$ near $c$ since
$\xi$ is tangent to $S$ at $b^\pm$. However we will find at least one point on
$c$ where the characteristic foliation has to turn clockwise. If $Y_t$ is a
vector field defining $\xi S_t$, the contact condition \eqref{eqn:ccg} can be
expressed as: $u_t \Div Y_t - du_t(Y_t) + \dot\beta_t(Y_t) > 0$.  
The sign of singularities is the sign of $u_t$ so $u(b^-) < 0$ and $u(b^+) > 0$.
Hence there is some point $p$ on $c$ such that $u(p) = 0$ and $du(Y) \geq 0$.
Here we used that $c$, hence $Y$, is oriented from $b^-$ to $b^+$. At $p$, the
contact condition becomes $\dot\beta(Y) > du(Y)$ so $\dot\beta(Y) > 0$. This
is the announced rotation. Since $\beta(Y) = 0$, we have that, at $p$, $\xi S_t$
is positively transverse to $c$ for $t > t_0$ and negatively transverse for $t <
t_0$. Of course this observation is very far from proving the crossing lemma,
see \cite[Lemma 2.14]{Giroux_2000} for the full story.

\section{Bennequin's theorem}

The goal of this section is to prove that the standard contact structure on
$\R^3$ is tight. This was originally proved by Bennequin, without the word tight
which was introduced by Eliashberg. 

Suppose there is an overtwisted disk in the standard contact structure on
$\R^3$. Since it is compact, it is contained in some finite radius ball. We can
also assume it misses a small ball around the origin (for instance we can use
the contact vector field $\partial_z$ to push it upward until this is true).
Recall we saw in Example~\ref{ex:radial_field} there is a contact vector field
$X$ on $\R^3$ which is transverse to all Euclidean spheres around the origin. So
these spheres are all $\xi$--convex and divided by the equator $\{ z = 0\}$
where $X$ is tangent to $\xi$. 
The above discussion shows that Bennequin's theorem is a consequence of the following
statement.

\begin{theorem}[Bennequin seen by Giroux {\cite[Theorem 2.19]{Giroux_2000}}]
\label{thm:bennequin_giroux}
Let $\xi$ be a contact structure on a thickened sphere $S \times [-1, 1]$. 
If all spheres $S_t$ are $\xi$--convex with connected dividing set then $\xi$
is tight.
\end{theorem}

\subsubsection*{Families of movies}

In order to prove Theorem \ref{thm:bennequin_giroux}, we first need some
preparations from dynamical systems.
Suppose that $\xi_0$ and $\xi_1$ are two contact structures which print generic
movies on $S \times [-1, 1]$. If they are isotopic, one gets a 2-parameters
family $\xi_s S_t$ of characteristic foliations. Thom transversality and a
little bit of normal form theory tells us that we can perturb the family until all
these foliations have finitely many singularities which are either nodes,
saddles or saddle-nodes. Further perturbations allow to make sure that all
closed leaves have a Poincar\'e return map which is at worse tangent to the
identity up to order $2$, the worse case happening only for isolated values of
$(s, t)$. 

Up to this point there was nothing specific to the sphere. The
first special property of $\S^2$ which is crucial in the following is the
Poincar\'e-Bendixson theorem which says that, since we have isolated
singularities for all our foliations, the Poincar\'e-Bendixson property
automatically holds. In particular we can apply the convexity criterion of Proposition
\ref{prop:critère_convexité}.
In the square $[0, 1] \times [-1, 1]$ the set $\Omega$ of points $(s, t)$ such
that $S_t$ is $\xi_s$--convex is a dense open set.
We denote by $\Sigma$ the complement of $\Omega$. It is a union of injectively
immersed submanifolds of $[0, 1] \times [-1, 1]$.
In codimension 1, one sees:
\begin{itemize}
	\item $\Sigma^1_\text{dl}$ where the characteristic foliation has a single
		weakly degenerate closed leaf and no retrograde saddle connection and no
		degenerate singularity, see Figure \ref{fig:deg_leaf}
	\item $\Sigma^1_\text{sc}$ where the characteristic foliation has a single
		retrograde saddle connection and no degenerate closed leaf or singularity,
		see Figure \ref{fig:croisement}.
\end{itemize}
The bifurcation lemmas imply that these two subsets are injectively immersed
submanifold of the square transverse to the $t$ direction. In addition,
the bifurcation lemmas imply that components of $\Sigma^1_\text{dl}$ can accumulate only on $\Sigma^1_\text{sc}$,
see Figure \ref{fig:csr_dl} for an example of accumulation.
We set $\Sigma^1 = \Sigma^1_\text{dl} \cup \Sigma^1_\text{sc}$.
\begin{figure}[htp]
	\begin{center}
		\includegraphics{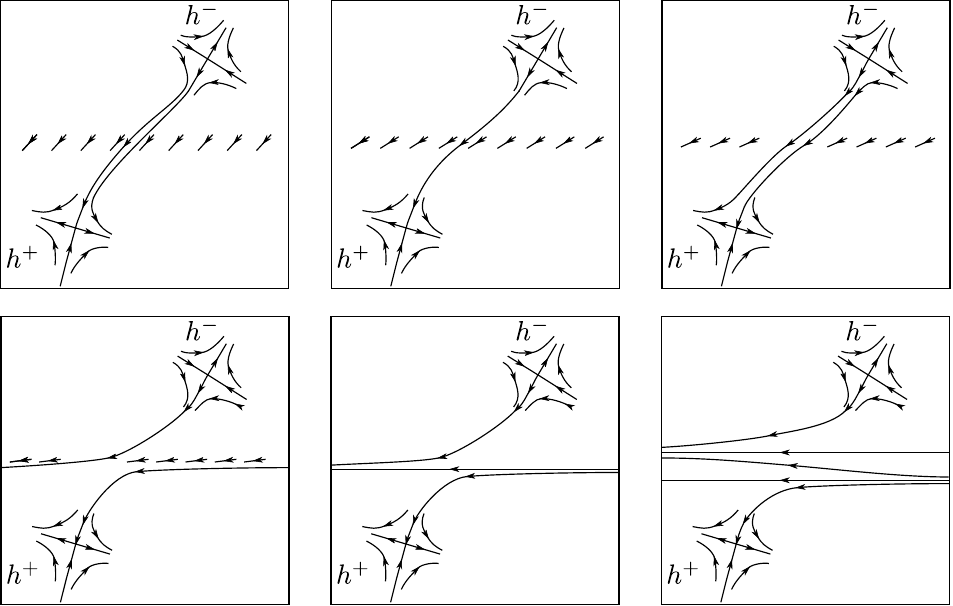}
	\end{center}
	\caption{Saddle connections accumulating a degenerate closed leaf. This is a
	movie of characteristic foliations on an annulus obtained by gluing the left
	and right sides of each square. A degenerate closed leaf is appearing in the
	middle. Leaves spiral more and more in this region, resulting in infinitely
	many retrograde saddle connections.}
	\label{fig:csr_dl}
\end{figure}

\begin{indented}
	The accumulation of retrograde saddle connections in Figure~\ref{fig:csr_dl} is
	not a phenomena which we can get rid of by perturbation: it is structurally
	stable in a 1-dimensional family, see~\cite{Sotomayor}. However, Giroux's
	discretization lemma \cite[lemma 15]{Giroux_transfo} states that any contact
	structure on the product $F \times I$ of a closed surface and an interval with convex
	boundary is isotopic relative to the boundary to a contact structure such
	that only finitely many $F_t$ are non-convex. This isotopy cannot be made
	arbitrarily small. It uses first the dynamics banalization lemma 
	\cite[Lemma 2.10]{Giroux_2000} which gets rid of non-trivial recurrence and
	then replaces degenerate leaves with retrograde saddle connexions. Both moves
	are non-perturbative.
\end{indented}

\noindent
In codimension 2, one sees:
\begin{itemize}
\item 
$\Sigma^{11}$ where two codimension one strata intersect transversely, see
Figure \ref{fig:sigma11sc} and also Figure \ref{fig:csr_texture} for a realistic
view of the central picture in the case of example
\ref{example:connection_selle}.
	\begin{figure}[p]
		\begin{center}
			\includegraphics{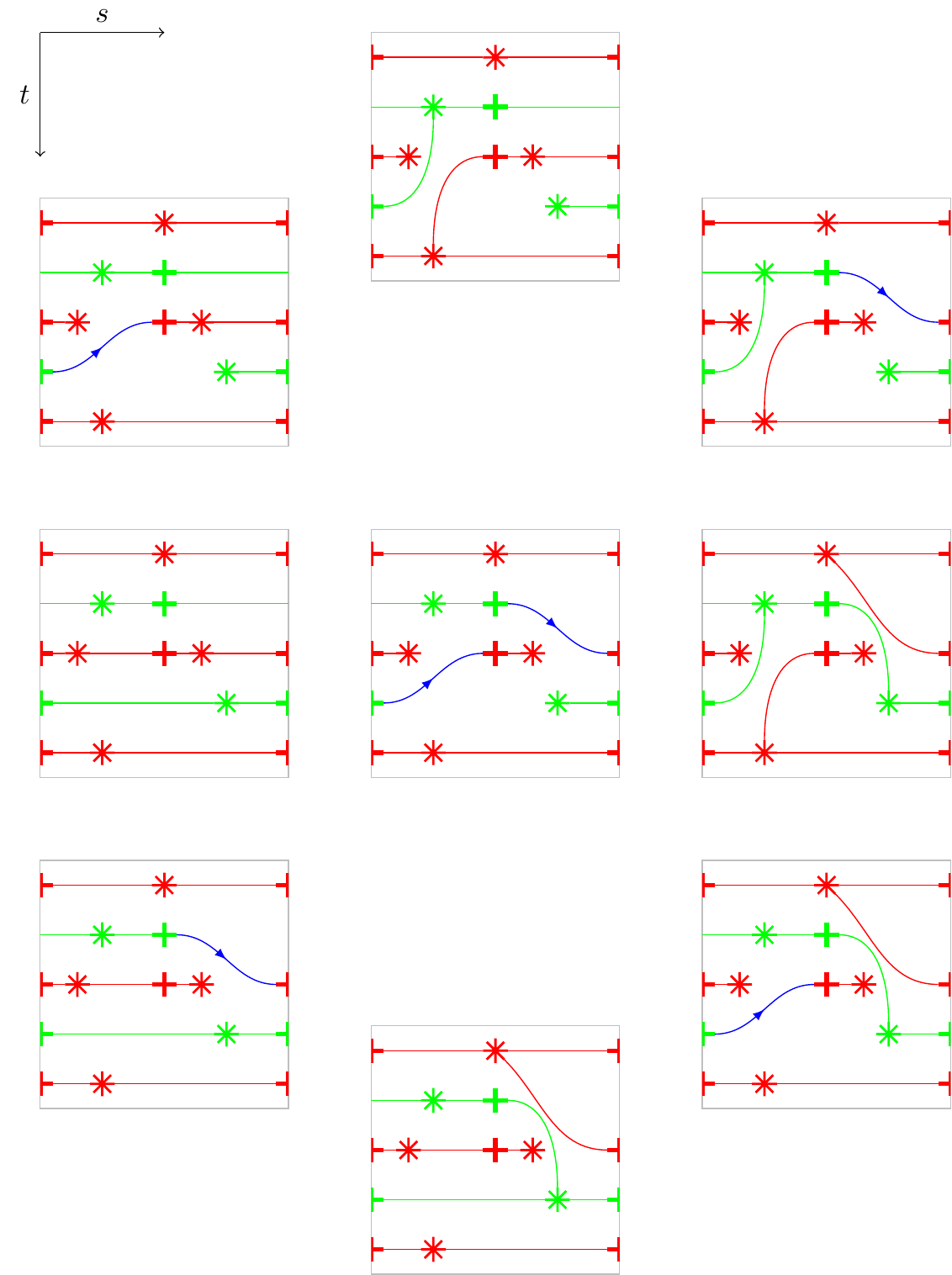}
		\end{center}
		\caption{Intersection of two strata of retrograde saddle connections on a
		torus. It is a good exercise to draw the Giroux graphs of all convex
		surfaces appearing to see the non-trivial effect of this codimension 2
		phenomenon on the dividing sets, contrasting with the discussion below.}
		\label{fig:sigma11sc}
	\end{figure}
\item
$\Sigma^2_\text{sc}$ where there is a retrograde connection between a saddle and
a saddle-node. These points adhere to exactly one stratum in 
$\Sigma^1_\text{sc}$, this typically happens in the proof of the classification
on $S^3$ as an intermediate step between Figures \ref{fig:s3before} and
\ref{fig:s3after}
\item
$\Sigma^2_\text{dl}$ where there is a degenerate orbit corresponding to the
fusion of two components of $\Sigma^1_\text{dl}$, see Figure \ref{fig:sigma2dl}
for the picture in the $(s, t)$ square and Figure \ref{fig:sigma2dlfoliations}
for the corresponding foliations.
\end{itemize}
\begin{figure}[htp]
	\begin{center}
		\includegraphics{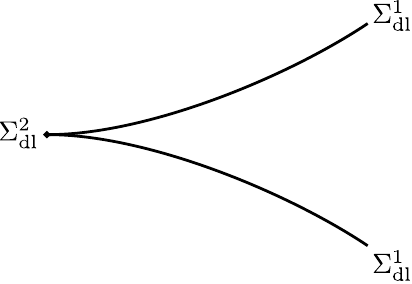}
	\end{center}
	\caption{The central point is in $\Sigma^2_\text{dl}$. It corresponds to a
	degenerate closed leaf with $\pi''(0) = 0$ but $\pi^{(3)}(0) < 0$, see Figure
	\ref{fig:sigma2dlfoliations} for the corresponding foliations.}
	\label{fig:sigma2dl}
\end{figure}
\begin{figure}[htp]
	\begin{center}
		\includegraphics{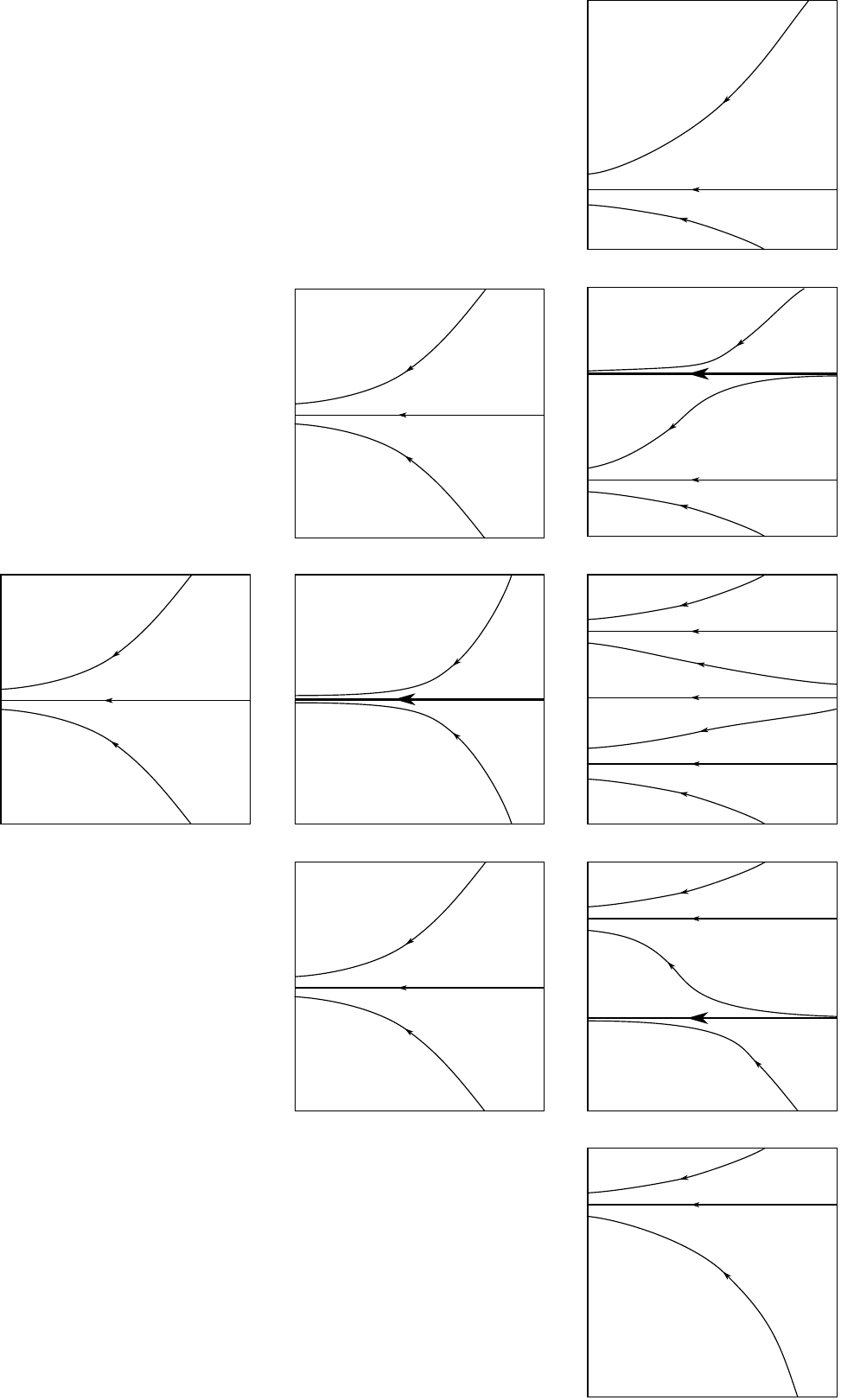}
	\end{center}
	\caption{Foliations corresponding to the strata of Figure \ref{fig:sigma2dl}. 
	Left and right of each square are glued to get an annulus. Thick closed leaves
	are the degenerate ones. The central picture corresponds to the annihilation
	of a birth and a death of non-degenerate closed leaves.}
	\label{fig:sigma2dlfoliations}
\end{figure}

\subsubsection*{Proof core}

We now prove Theorem \ref{thm:bennequin_giroux}.
Suppose there is some overtwisted disk in $(S \times [-1, 1], \xi)$.
Then there is some isotopy relative to the boundary bringing this disk onto the middle sphere
$S_0$. 
So this isotopy sends $\xi_0 = \xi$ to a contact structure $\xi_1$ such that
$S_0$ contains an overtwisted disk. Then it can be modified
in the same way genericity of convex surfaces is proved until $S_0$ is
$\xi_1$--convex and divided by a disconnected curve (use Corollary
\ref{cor:giroux_graph} to understand dividing sets here).
We can perturb $\xi_1$ to make sure it also prints a generic movie of
characteristic foliations and perturb the isotopy to be in the situation of the
preceding discussion on families of movies.

The set $\Omega$ of $(s, t)$ such that $S_t$ is $\xi_s$--convex 
is the disjoint union of $\Omega_c$ corresponding to connected dividing sets
and $\Omega_d$ corresponding to disconnected ones.

In addition, we know by construction that $\Omega_d$ intersects the
right vertical edge $\{s = 1\}$ so it is not empty. But it does not intersect
the left edge $\{s = 0\}$ by hypothesis of the theorem. 
More precisely, we can assume the closure of $\Omega_d$ does not meet $\{s =
0\}$ so the minimum $s_0$ of its projection to $[0, 1]$ is positive. Choose
$t_0$ such that $(s_0, t_0)$ is in the closure of $\Omega_d$.

The point $(s_0, t_0)$ cannot be in:
\begin{itemize}
\item 
$\Sigma^1$ because the later is transverse to the $t$ direction so components of
$\Omega$ adjacent to a point $(s,  t)$ in $\Sigma^1$ project to neighborhoods of
$s$
\item
$\Sigma^2_\text{sc}$ because each point $(s, t)$ in $\Sigma^2_\text{sc}$ adheres
to only one component of $\Sigma^1_\text{sc}$ so the intersection between
$\Omega$ and a small disc around $(s, t)$ is connected and projects to a
neighborhood of $s$.
\item
$\Sigma^2_\text{dl}$ because all components of $\Omega$ touching
$\Sigma^2_\text{dl}$ are in $\Omega_d$ because the corresponding foliations have
closed leaves.
\item
any point $\Sigma^{11}$ involving degenerate closed leaves, again because 
strata in $\Sigma^1_\text{dl}$ are transverse to the $t$-direction and indicate
birth or death of stable closed leaves giving disconnected dividing sets.
\end{itemize}

The only configuration which really needs to be carefully ruled out is that of
points in $\Sigma^{11}$ involving only $\Sigma^1_\text{sc}$ like in figure
\ref{fig:bif_crossing_bennequin}
\begin{figure}[htp]
	\begin{center}
		\includegraphics{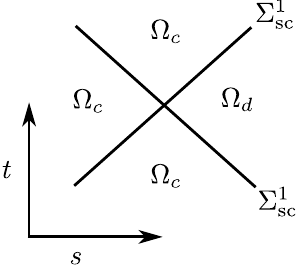}
	\end{center}
	\caption{The situation we must rule out for Bennequin's theorem}
	\label{fig:bif_crossing_bennequin}
\end{figure}
In this situation $\xi_{s_0} S_{t_0}$ has two retrograde saddle connections
which happen on different surfaces $S_t$ for $s$ in a punctured neighborhood of
$s_0$ and get swapped when $s$ goes through $s_0$, as in Figure
\ref{fig:sigma11sc}.  Note that characteristic foliations around $(s_0, t_0)$
have no closed leaf and we can also assume they do not have other saddle
connections that the ones we explicitly study.

To $\xi_s S_t$ we associate the oriented graph $\Gamma_+(s, t)$ (resp.
$\Gamma^-(s, t)$) whose vertices are
positive nodes and edges are the stable separatrices of positive
saddles (resp. negative saddles).  Since we do not have any closed leaf or
degenerate singularities near $(s_0, t_0)$, $\Gamma_+$ coincides as a set with
$G_+$ from definition \ref{def:giroux_graph} and $\Gamma^-$ is somehow dual to
$G_-$.  So, according to Corollary \ref{cor:giroux_graph}, when $S_t$ is
$\xi_s$--convex, there is a regular neighborhood of $\Gamma_+(s, t)$ whose
boundary divides $\xi_s S_t$.
Because $S$ is a sphere, we then get that $(s, t)$ is in $\Omega_c$ if and only
if $\Gamma_+(s, t)$ is a tree (ie a closed connected and simply connected
graph). 
We want to use the crossing lemma to understand how the graph changes when a
retrograde saddle connection happens, see Figure~\ref{fig:csr_arc}. 
\begin{figure}[htp]
	\begin{center}
		\includegraphics{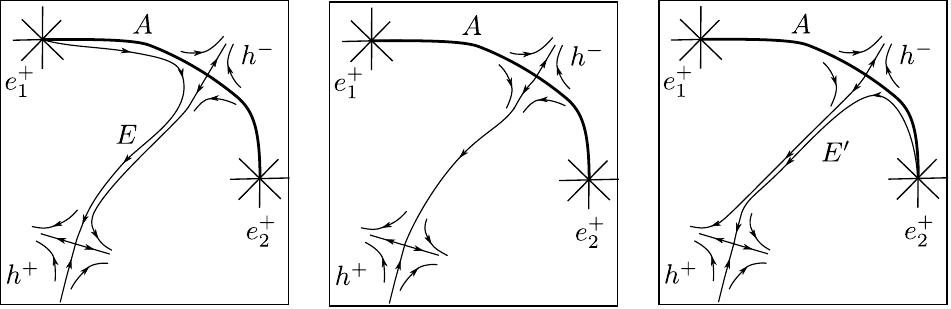}
	\end{center}
	\caption{Anatomy of a retrograde saddle connection}
	\label{fig:csr_arc}
\end{figure}

First we remark that, if we focus on a sufficiently small neighborhood of
$(s_0, t_0)$ in parameter space, the graph $\Gamma^-(s,t)$ deforms by isotopy
so we can assume it does not depend on $s$ and $t$. The same is true for
$\Gamma_+(s, t)$ as long as we stay in the complement of $\Sigma$.  Suppose now there
is a saddle connection involving a negative saddle $h^-$.  Let $A$ be the
closure of the union of its stable separatrices. The unstable separatrix of
$h^-$ entering the saddle connection coorients $A$ and, together with the
orientation of $S$, this orients $A$. We denote by $o(A)$ and $d(A)$ the origin
and destination of $A$.

During a bifurcation, exactly one edge $E$ of $\Gamma_+$ changes.
After the bifurcation, the edge $E$ is replaced by an edge $A(E)$ which is
obtained from the concatenation of $E$ and $A$ by a small push towards the
right which makes it avoid $o(A)$, see Figure \ref{fig:csr_arc_schematic} which also
explains how these things will be drawn schematically in the following.
\begin{figure}[htp]
	\begin{center}
		\includegraphics{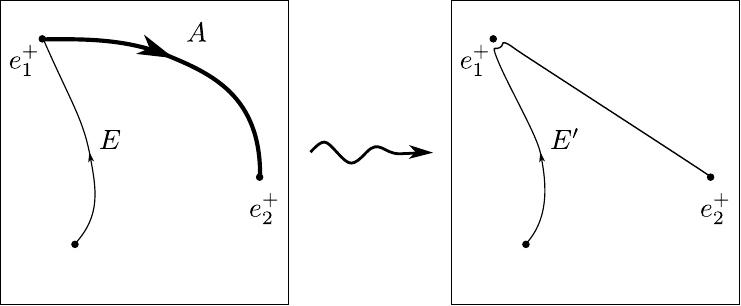}
	\end{center}
	\caption{A schematic view of the same retrograde saddle connection as in
	Figure \ref{fig:csr_arc}}
	\label{fig:csr_arc_schematic}
\end{figure}
Note that the edge $E$ is the edge which is immediately to the right of $A$ 
at $o(A)$ with respect to the cyclic ordering of edges of
$\Gamma_+ \cup \Gamma^-$ incident to $o(A)$.
So the oriented arc $A$ completely describes the bifurcation. 
We will denote by $A(\Gamma_+)$ the graph obtained from
$\Gamma_+$ after a bifurcation described by $A$ (up to isotopy). 

Returning to the codimension 2 bifurcation at $(s_0, t_0)$ we have two distinct
strata $\Sigma^1_\text{sc}(A_1)$ and $\Sigma^1_\text{sc}(A_2)$
corresponding to distinct (oriented) bifurcation arcs $A_1$ and $A_2$, see
Figure~\ref{fig:regions}.
\begin{figure}[ht]
	\begin{center}
		\includegraphics{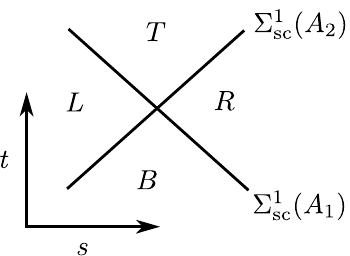}
	\end{center}
	\caption{Regions in the parameter space.}
	\label{fig:regions}
\end{figure}
We take the graph $\Gamma_+$ of the Bottom region as a reference and apply to it
the following proposition. Note that, on a tree, any ordered pair of vertices
determines a unique oriented segment.

\begin{proposition}
\label{prop:trees}
Suppose $\Gamma$ is a tree and $A_1$ and $A_2$ are bifurcation arcs for
$\Gamma$. The following properties are equivalent.
\begin{enumerate}
\item 
$A_1(\Gamma)$ is not a tree but $A_2(A_1(\Gamma))$ is a tree.

\item
On $\Gamma$, the oriented segment $S$ from $d(A_2)$ to $d(A_1)$ contains, in
that order: $d(A_2) \leq o(A_1) < o(A_2) \leq d(A_1)$ and, furthermore,
$S$ is immediately to the right of $A_1$ at $o(A_1)$ and $A_2$ at
$o(A_2)$.
\end{enumerate}
\end{proposition}

Note that condition 1 above holds if $\Gamma$ is the tree $\Gamma_+$ coming from
the Bottom region $B$ since we assume $T$ and $B$ are in $\Omega_c$ while $R$ is
in $\Omega_d$.  This proposition concludes the proof of
Theorem~\ref{thm:bennequin_giroux} because condition 2 above is symmetric in
$A_1$ and $A_2$ (here one should not forget that exchanging $A_1$ and $A_2$ will
reverse the orientation on $S$). So the graph $A_2(\Gamma)$ corresponding to the
left region $L$ is not a tree and $L$ is also in $\Omega_d$.

\begin{proof}
We first prove that property 1 implies property 2.
Let $E$ be the edge of $\Gamma$ modified by $A_1$. In particular 
$E$ has vertices $o(A_1)$ and some other vertex $v$ and $E$ is immediately to
the right of $A_1$ at $o(A_1)$. Because $\Gamma$ is a tree, $v$ can't be the
same as $o(A_1)$ and (the closure of) $\Gamma \setminus E$ is the disjoint
union of two trees $\Gamma_1$ containing $o(A_1)$ and $\Gamma_2$ containing
$v$, see Figure~\ref{fig:trees}.

\begin{figure}[htp]
	\begin{center}
		\includegraphics[width=12cm]{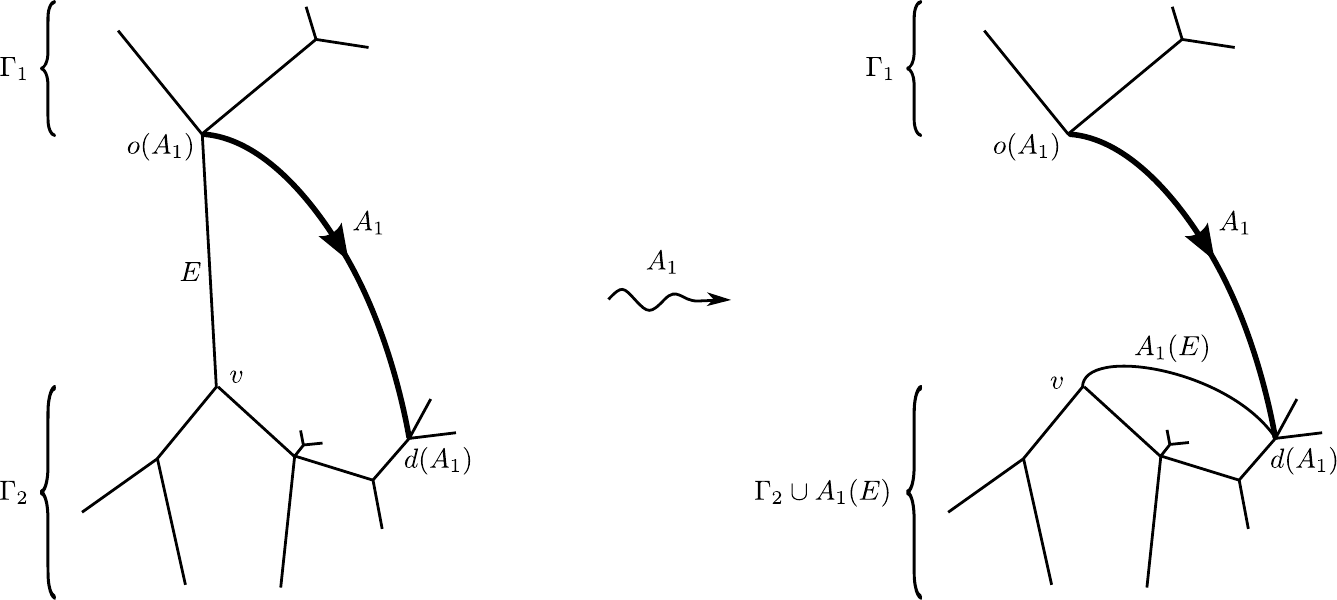}
	\end{center}
	\caption{Trees and graphs in the proof of Propostion~\ref{prop:trees}.}
	\label{fig:trees}
\end{figure}

Note that $d(A_1)$ cannot be in $\Gamma_1$ since otherwise $A_1(E)$ would go
from $\Gamma_1$ to $\Gamma_2$ and $A_1(\Gamma)$ would be a tree.

So $d(A_1)$ is in $\Gamma_2$ and this implies that $v$ in the segment 
$[o(A_1), d(A_1)] \subset \Gamma$. Also we learn that $A_1(\Gamma)$ is the
disjoint union of the tree $\Gamma_1$ and the graph $\Gamma_2 \cup A_1(E)$
which contains exactly one cycle $C$. This cycle contains $A_1(E)$ and its
vertices are all in $[v, d(A_1)] \subset \Gamma$, see Figure~\ref{fig:trees}
again.

Since $A_2(A_1(\Gamma))$ is a tree, the edge $E'$ modified by $A_2$ in $A_1(\Gamma)$
belongs to $C$ otherwise $C$ would persist in $A_2(A_1(\Gamma))$.
So we get that $o(A_2)$ is in $C$ (in particular it can't be the same as
$o(A_1)$). In addition $d(A_2)$ is in $\Gamma_1$ otherwise
$A_2(A_1(\Gamma))$ would stay disconnected. The last thing to check is that
$E'$ is part of the segment $[d(A_2), d(A_1)] \subset \Gamma$. The only edge of
$C$ which is not in this segment is $A_1(E)$. Remember $E'$ is immediately to
the right of $A_2$ at $o(A_2)$ so it cannot be $A_1(E)$ because that would
force $A_2$ to go into the disk bounded by $C$ which does not
contain $\Gamma_1$ (surreptitiously using Schönflies theorem again).

We now prove the converse implication. Since $S$ is immediately to the right of
$A_1$ at $o(A_1)$, it contains the edge $E$ of
$\Gamma$ moved by $A_1$. More precisely, $E$ is in the segment 
$[o(A_1), d(A_1)] \subset \Gamma$. So $A_1(\Gamma)$ is the disjoint union of a
tree $\Gamma_1$ and a graph $\Gamma_2$ containing a unique cycle $C$.
Since $S$ is immediately to the right of $A_2$ at
$o(A_2)$ and $o(A_1) \neq o(A_2)$, the edge $E'$ in $A_1(\Gamma)$ moved by
$A_2$ is either an edge in $S$ or $A_1(E)$. In both cases, it is contained in
$C$. So the cycle $C$ does not persist in $A_2(A_1(\Gamma))$ and $A_2(E_1)$
connects $\Gamma_2 \setminus E'$ to $\Gamma_1$. Hence $A_2(A_1(\Gamma))$ is a
tree.
\end{proof}

Now this proof is finished let us see where we used the contact condition and
not only properties of generic families of foliations with two parameters. The
first thing is that $\Sigma^1$ is transverse to the $t$ direction because of the
bifurcation lemmas. A second more subtle point is that the crossing lemma says
more: it tells the direction of the bifurcations: separatrices turn to their
right when $t$ increases. Figure \ref{fig:contre_ex} show how the above proof
would fail if $A_1$ and $A_2$ were allowed to act as switches in opposite
direction. In that figure one sees an example of the bad situation of Figure
\ref{fig:bif_crossing_bennequin}.
\begin{figure}[htp]
	\begin{center}
		\includegraphics[width=11.5cm]{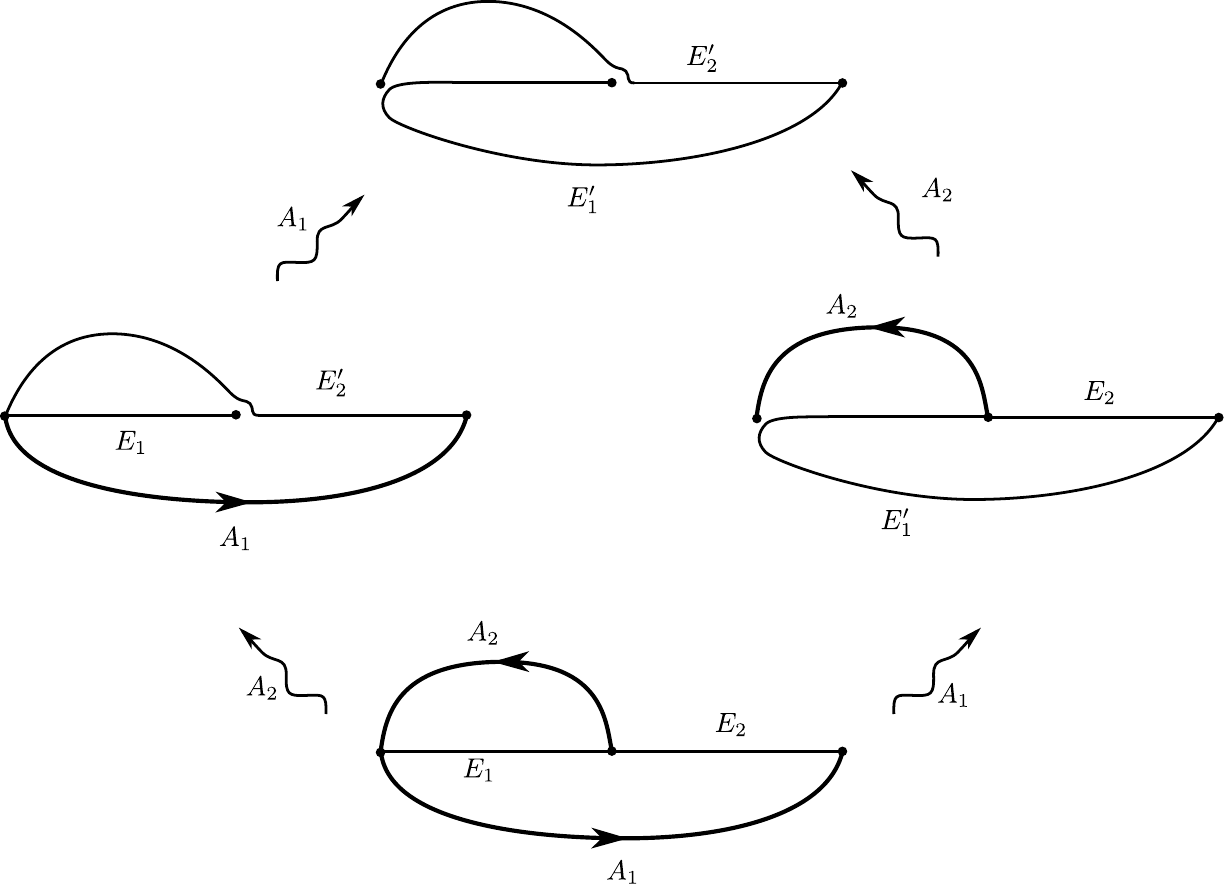}
	\end{center}
	\caption{How the discussion would fail if $A_1$ were reversed. In this example
	the reference graph has three vertices and two edges. Regions $L$, $T$ and $B$
	are tight whereas $R$ is overtwisted.}
	\label{fig:contre_ex}
\end{figure}
The explanation is that, if we assume that the bifurcation corresponding to 
$A_1$ acts in the wrong direction then, in Proposition~\ref{prop:trees}, we
must replace ``to the right of $A_1$'' by  ``to the left of $A_1$''
and we loose symmetry between $A_1$ of $A_2$. Of course if both $A_1$ and $A_2$
act in the wrong direction then we do not have any difference, this simply 
corresponds to considering negative tight contact structures on $\S^3$.

\vspace{2cm}

\textsc{Université Paris Sud, 91405 Orsay, France}

\textit{Email adress: } \verb+patrick.massot@math.u-psud.fr+

\textit{URL: } \verb+www.math.u-psud.fr/~pmassot/+

\end{document}